\documentclass[oneside, a4paper,11pt,reqno]{amsart}
\usepackage{yfonts}

\RequirePackage[colorlinks,citecolor=blue,urlcolor=blue]{hyperref}

\makeatletter
\def\timenow{\@tempcnta\time
  \@tempcntb\@tempcnta
  \divide\@tempcntb60
  \ifnum10>\@tempcntb0\fi\number\@tempcntb
  \multiply\@tempcntb60
  \advance\@tempcnta-\@tempcntb
  :\ifnum10>\@tempcnta0\fi\number\@tempcnta}
\makeatother

\usepackage[ddmmyyyy,hhmmss]{datetime}
\usepackage{color}

\usepackage{euscript}

\usepackage[applemac]{inputenc}
\usepackage{amssymb,amsmath,amsthm,dsfont}
\usepackage{graphics,graphicx}
\usepackage[T1]{fontenc}
\usepackage{color}
\usepackage{epsfig,psfrag}

\theoremstyle{plain}
\newtheorem{theo}{Theorem}
\newtheorem{prop}[theo]{Proposition}
\newtheorem{lemme}[theo]{Lemma}
\newtheorem{cor}[theo]{Corollary}

\theoremstyle{remark}
\newtheorem{remark}[theo]{Remark}
\newtheorem{example}[theo]{Example}

\title{Spectral analysis of a class of L\'evy-type processes and connection with some spin systems}

\def\psl{\langle}
\def\psrll{\rangle_{L^2(\mathcal{L})}}
\def\psrm{\rangle_{L^2(m)}}
\def\psrr{\rangle_{L^2(\mathbb{R})}}
\def\psri{\rangle_{L^2((-1,1))}}
\def\nr{\|_{L^2(\mathbb{R})}}
\def\nri{\|_{L^2((-1,1))}}

\def\crz{\mathcal{C}_0(\mathbb{R})}

\def\pot{V_{int}}
\def\pen{U_{mass}}
\def\partfct{Z}
\def\multip{M}

\newcommand*{\affmark}[1][*]{\textsuperscript{#1}}

\author{Gr\'egoire V\'echambre\affmark[1]}

\address{\affmark[1]Academy of Mathematics and Systems Science, Chinese Academy of Sciences, No. 55, Zhongguancun East Road, Haidian District, Beijing, China}
\email{vechambre@amss.ac.cn}

\begin{document} 

\maketitle

\pagestyle{myheadings}
\markboth{Right}{Spectral analysis of a class of L\'evy-type processes and connection with Interacting Particles Systems}

\begin{abstract}
We consider a class of L\'evy-type processes on which spectral analysis technics can be made to produce optimal results, in particular for the decay rate of their survival probability and for the spectral gap of their ground state transform. This class is defined by killed symmetric L\'evy processes under general random time-changes satisfying some integrability assumptions. Our results reveal a connection between those processes and a family of spin systems. This connection, where the free energy and correlations play an essential role, is, up to our knowledge, new, and relates some key properties of objects from the two families. When the underlying L\'evy process is a Brownian motion, the associated spin system turns out to have interactions of a rather nice form that are a natural alternative to the quadratic interactions from the lattice Gaussian Free Field. More general L\'evy processes give rise to more general interactions in the associated spin systems. 
\end{abstract}

{ \footnotesize
	\noindent{\slshape\bfseries MSC 2020.} Primary:\, 60G53, 60G51, 60J25, 60J35, 82B20 \ Secondary:\, 60J55, 82B31\\
	\noindent{\slshape\bfseries Keywords.} L\'evy-type processes, L\'evy processes with random time changes, spectral analysis, ground state, ground state transform, spectral gap, survival probability, spin systems, interface models, partition function, free energy, infinite-volume Gibbs states, correlations

\section{Introduction}

A fundamental problem about Markov processes is to determine the asymptotic behavior of their survival probability (when their life-time is finite) or the speed of convergence toward their stationary distribution (when they are ergodic). Obtaining such estimates for Markov processes is a well-studied issue on which spectral analysis methods offer great insight \cite{2005chen}, \cite{2005wang}, \cite{refId0}, \cite{GONG2006151}, \cite{GONG20145639}. However, even though those methods are extremely powerful, they unfortunately do not often provide explicit optimal estimates. This is especially true for Markov processes with jumps such as L\'evy-type processes (see for example Chapter 6.1 of \cite{levymatters3}). In particular, when the survival probability, or distance to the stationary distribution, of a Markov process decays exponentially fast, one can hope to establish the positivity of the decay rate, and sometimes to produce an interval that contains it, but it is unusual to obtain exact expressions of the decay rate. Unfortunately, there seem to be no hope to develop a methodology that would produce optimal estimates and exact expressions in full generality. It seems that a more reasonable approach is to isolate subclasses of sufficiently nice Markov processes and to develop optimal methodologies that are tailored for each class. For classes of Markov processes with jumps, this is already a challenging issue. 

The goal of the present paper is to isolate a non-trivial and fairly large class of L\'evy-type processes for which we can set up an adapted spectral analysis methodology that gets through, allowing to provide optimal estimates and useful exact expressions. The most natural and convenient way to represent this class is via randomly time-changed L\'evy processes (but see also \eqref{cbregeneral} below for an SDE representation). Several important classes of Markov processes admit representations via randomly time-changed L\'evy processes, which are natural and interesting objects. Let us mention in particular Positive self-similar Markov processes (pssMp's) that include Bessel processes \cite{Lamperti2}, \cite{5e0b08215cf94b4f9738d908edd90e6f}, \cite{pardosurvey}, Continuous-state Branching processes (CSBP's) \cite{Lamperti3}, \cite{Lambert2008PopulationDA}, some generalizations of CSBP's \cite{polcsbp}, \cite{gencsbp}, diffusions in random potentials (which can be represented as functions of randomly time-changed Brownian motions) \cite{Brox}, \cite{KawazuTanaka}, \cite{Singh}, \cite{psvech}, or also the skew-product representation for planar Brownian motion \cite{legallstfloor}, and many others. In our case, we consider a class of processes that are represented by killed symmetric L\'evy processes under rather general random time-changes. Let us mention that some of the questions we study on these processes may be reformulated in terms of integral functionals of L\'evy processes (see Section \ref{afortclp} below for more details and a brief account on that topic). The main specificity of the class we are interested in is the integrability assumptions on the characteristic exponent of the underlying L\'evy process and on a function appearing in the random time-change. Those assumptions seem to be what make nice the class of processes and allow our methodology to work efficiently. We show that several key properties of those processes are related to properties of a class of spin systems. In particular, we prove optimal estimates for the asymptotic of the survival probability of a process in our class and show that the decay rate can be expressed in terms of a normalized version of the free energy of a spin system whose interactions are determined by the potential density of the underlying L\'evy process. When the killing rate $r$ goes to $0$, the study of the spin system allows, under some conditions, to determine the exact asymptotic of the decay rate of the survival probability in terms of $r$. We also study the ground state transform of a process in our class and show that its spectral gap can be expressed in terms of the free energy and of the decay rate of correlations of the associated spin system. Further properties of the processes and of their ground state transforms are related to aspects of the spin systems such as infinite-volume Gibbs states. This surprising connection between our class of processes and equilibrium statistical mechanics is, up to our knowledge, new. While this connection turns out to be useful for understanding the processes we are interested in, it also opens a way for a deeper study of the class of spin systems, which is interesting on its own. In particular, when the underlying L\'evy process is a Brownian motion, the associated spin system turns out to have interactions of a rather nice form that are a natural alternative to the quadratic interactions from the lattice Gaussian Free Field (GFF), and more general L\'evy processes give rise to more general interactions in the associated spin system. 

An interesting aspect of our methodology and results is that they are available for rather general random time-changes while, as we can see in the end of Section \ref{afortclp}, several classical classes of Markov processes consist of functions of L\'evy processes with a specific random time-change. The ideas introduced in the present paper can be applied to more general settings. In particular we aim to apply, in a subsequent study, those ideas to a fairly large class of L\'evy-type processes on Lie groups. 

\subsection{A family of L\'evy-type processes} \label{afortclp}

Let $\xi$ be a real symmetric L\'evy process. Let $\psi_\xi(\cdot)$ be the characteristic exponent of $\xi$, i.e. for any $t \geq 0$ and $y \in \mathbb{R}$ we have $\mathbb{E}[e^{iy\xi(t)}]=e^{t\psi_\xi(y)}$. According to the L\'evy-Khintchine formula and the symmetry of $\xi$, $\psi_\xi(\cdot)$ can be expressed by 
\begin{align*}
\psi_\xi(y) = -\frac{A_{\xi}}{2} y^2 + \int_{\mathbb{R}} (e^{iyu}-1-yu\textbf{1}_{[-1,1]}(u)) \Pi_{\xi}(du), 
\end{align*}
where $A_{\xi}$ is a non-negative number and $\Pi_{\xi}$ is a symmetric measure on $\mathbb{R}$ satisfying $\int_{\mathbb{R}} (1 \wedge u^2) \Pi_{\xi}(du)<\infty$. $A_{\xi}$ and $\Pi_{\xi}$ are called respectively the Gaussian component and the L\'evy measure of $\xi$. 

We now prepare for the definition of a process $X^{m,\xi,r}_x$ as a time changed version of $x+\xi(\cdot)$, where $x\in \mathbb{R}$ is the starting position of the process. As for many classical processes, we consider a time change given by an integral functional of $\xi$, but we do not impose a specific form for the random time-change. More precisely, we consider a time change defined as the inverse of the function $(t \mapsto \int_0^t m(x+\xi(s)) ds)$, where $m$ is a general positive function satisfying some assumptions. We assume the function $m$ to be square-integrable on $\mathbb{R}$. This integrability assumption is not met by some classical families of Markov processes that can be represented via randomly time-changed L\'evy processes (see the end of this subsection). However, in our case, it will be essential to define and study key quantities and operators, and to draw a connection with a spin system. We also assume $m$ to be continuous on $\mathbb{R}$ and vanishing at $\infty$ and $-\infty$ (we denote this by $m \in \crz$). The assumptions on $m$ are gathered as follows: 
\begin{equation}\tag{Condition 1}\label{mheavytailed}
m \in \crz \cap L^2(\mathbb{R}), \ \forall x \in \mathbb{R}, m(x) > 0. 
\end{equation}
We will sometimes make the stronger assumption 
\begin{equation}\tag{Condition 1'}\label{mheavytailed1}
m \in \crz \cap L^1(\mathbb{R}), \ \forall x \in \mathbb{R}, m(x) > 0. 
\end{equation}
Note that \eqref{mheavytailed1} implies \eqref{mheavytailed}. 

Let $r>0$. The killed symmetric L\'evy process $\xi^r$ is defined by $\xi^r(s):=\xi(s)$ for $s<e_r$ and $\xi^r(s):=\dagger$ for $s\geq e_r$, where $e_r$ is an exponential random variable with parameter $r$, independent of $\xi$, and $\dagger$ is a cemetery state. For any $x \in \mathbb{R}$ we define a random time-change as follows: 
\begin{equation}\label{defchgttps}
A_x(s) := \int_0^{s} m(x+\xi(u))du. 
\end{equation}
Since $m$ is bounded, continuous, and positive on $\mathbb{R}$, $A_x(\cdot)$ is almost surely continuous and increasing on $[0,\infty)$. In particular it has an inverse $A_x^{-1}(\cdot)$ that is almost surely continuous and increasing on $[0,A_x(\infty))$. For $t \geq A_x(\infty)$ we set $A^{-1}_x(t):=\infty$ by convention. We refer to Lemma \ref{axinfty} for a condition for $A_x(\infty)$ to be infinite. For any $x \in \mathbb{R}$, we define the c\`ad-l\`ag process $X^{m,\xi,r}_x$ by $X^{m,\xi,r}_x(t) := x+\xi^r(A_x^{-1}(t))$, or more precisely, 
\begin{equation}\label{defX}
X^{m,\xi,r}_x(t) := \left\{
\begin{aligned}
& x+\xi(A_x^{-1}(t)) \ \text{if} \ A^{-1}_x(t) < e_r, \\
& \dagger \ \text{if} \ A^{-1}_x(t) \geq e_r. \end{aligned} \right. 
\end{equation}
Note that $X^{m,\xi,r}_x$ is almost surely killed in finite time. Indeed, let $\zeta_x := A_x(e_r) \in (0,\infty)$, then note from \eqref{defX} that $X^{m,\xi,r}_x(t) \in \mathbb{R}$ for $t <\zeta_x$ and that $X^{m,\xi,r}_x(t) =\dagger$ for $t \geq\zeta_x$. 
We write $X^{m,\xi,r}$ to refer to the process defined in \eqref{defX} without specification of a starting point. To $X^{m,\xi,r}$ we naturally associate a family $(P_t)_{t \geq 0}$ of linear operators, defined on the space $\mathcal{B}_b(\mathbb{R})$ of bounded measurable functions $f:\mathbb{R}\rightarrow \mathbb{C}$, by 
\begin{equation}\label{defsgx}
P_t.f(x):=\mathbb{E} \left [ f(x+\xi(A_x^{-1}(t))) \textbf{1}_{A_x^{-1}(t)<e_r} \right ]. 
\end{equation}
It is proved in Lemma \ref{markovianity} that $X^{m,\xi,r}$ is an homogeneous Markovian process, so that the family $(P_t)_{t \geq 0}$ is a semigroup. The following proposition establishes that $X^{m,\xi,r}$ is even a Feller process (or sub-Feller process, in the terminology of some authors). 
\begin{prop} \label{fellerity}
Assume that $r>0$ and \eqref{mheavytailed} holds, then the semigroup $(P_t)_{t \geq 0}$ defined by \eqref{defsgx} is a Feller semigroup in the sense of Definition 1.2 of \cite{levymatters3}. 
\end{prop}
Proposition \ref{fellerity} is proved in Section \ref{5.2Markov-Feller}, together with some useful facts about $X^{m,\xi,r}$. Proposition \ref{fellerity} requires that $r>0$ and that \eqref{mheavytailed} holds. In the rest of this paper, we will always make those assumptions when dealing with $(P_t)_{t \geq 0}$ so that, in particular, $(P_t)_{t \geq 0}$ enjoys all the properties of Feller semigroups. 

Our primary goal is to characterize the decay rate of the survival probability $\mathbb{P}(\zeta_x > t)$ of $X^{m,\xi,r}_x$ via spectral analysis. In general, estimating the survival probability of a Markov process is not an easy issue. For pssMp's without positive jumps, several properties of their life-time are investigated in \cite{patie2012} via elaborated tools from complex analysis, yielding in particular the behavior of the survival probability of those processes. In our case, since $\zeta_x = A_x(e_r)\leq e_r \|m\|_{\infty}$, we have the trivial bound $\mathbb{P}(\zeta_x > t) \leq e^{-rt/\|m\|_{\infty}}$, but what can be said about the exact decay rate of $\mathbb{P}(\zeta_x > t)$? Roughly, the further away the process $X^{m,\xi,r}_x$ is from $0$, the faster time runs, so the faster killing occurs. Therefore, the behavior of $\mathbb{P}(\zeta_x > t)$ strongly depends on the behaviors of both the underlying L\'evy process $\xi$ and the function $m$ (that determines the random time-change). Note also that, since $\zeta_x = A_x(e_r)=\int_0^{e_r} m(x+\xi(s)) ds$, studying $\mathbb{P}(\zeta_x > t)$ amounts to studying the right distribution tail of the integral functional $\int_0^{e_r} m(x+\xi(s)) ds$, which is also not an easy issue. Integral functionals of L\'evy processes, or more generally of Markov processes, are well studied objects with a particular focus on determining conditions for finiteness/infiniteness of those functionals \cite{10.3150/19-BEJ1167}, \cite{10.3150/18-BEJ1034}; this has applications ranging in several domains such as asymptotic properties of Schr\"odinger semigroups \cite{batty1992} or properties of SDEs driven by stable L\'evy processes \cite{bagdoerkyp}. The moments of those functionals are studied in \cite{FITZSIMMONS1999117}. Our results show that $\mathbb{P}(\zeta_x > t)$ has order $K_{m,\xi,r}(x) e^{-t\gamma_{m,\xi,r}}$, where $\gamma_{m,\xi,r}$ and $K_{m,\xi,r}(x)$ are characterized uniquely in terms of a spin system. Under some conditions, we moreover determine the asymptotic of the coefficient $\gamma_{m,\xi,r}$ when $r$ is small, that is, when the killing occurs at a low rate, so that the behavior of $\xi$ has a strong influence. In this last case, it appears from our results that the trivial bound $\mathbb{P}(\zeta_x > t) \leq e^{-rt/\|m\|_{\infty}}$ is particularly bad. We also establish further spectral properties of the process $X^{m,\xi,r}$ and of its ground state transform, which draws more connections with the associated spin system. In particular, we express the spectral gap of the ground state transform of $X^{m,\xi,r}$ in terms of the associated spin system. 

The survival probabilities of other classes of Markov processes have already been studied. For example, the integral functionals expressing the life-time of pssMp's are the so-called \textit{exponential functionals of L\'evy processes}, which have been intensively studied \cite{Bertoinyor}, \cite{MR1648657}, \cite{blr}, \cite{pardo2013}, \cite{ref6patie}, \cite{patieref8}, \cite{bersteingamma}, \cite{foncexpovech}. Determining the asymptotic of the survival probability of pssMp's requires knowledge on the right distribution tail of exponential functionals, which have been studied in particular in \cite{rivero2005}, \cite{Maulik2006156}, \cite{10.230725464921}, \cite{rivero2012}. In the other way, non-trivial properties of exponential functionals of L\'evy processes are sometimes derived from the study of pssMp's \cite{pierre2009}, \cite{patie2012}. Many nice properties of exponential functionals come from the exponential function that allows to exploit the independence of increments of L\'evy processes. This includes the computation of moments of exponential functionals (see for example \cite{Bertoinyor}), and more generally a useful functional equation satisfied by their Mellin transforms \cite{MR1648657}, \cite{Maulik2006156}, \cite{Kuznetsov2013}, \cite{Patie2013393}. In our case, the exponential function is replaced by a general function $m$ satisfying \eqref{mheavytailed}, so we are deprived from those nice properties of exponential functionals (note also that our case excludes exponential functionals, since \eqref{mheavytailed} is not satisfied by the exponential function). 

In expressing $\gamma_{m,\xi,r}$, a key object that we use is the potential density of the killed L\'evy process $\xi^r$. It is defined as the density of the potential measure $V^r_\xi(dx) := \int_0^{\infty} e^{-rt} \mathbb{P} (\xi(t) \in dx) dt$. For the potential density to be well-defined we need to assume 
\begin{equation}\tag{Condition 2}\label{hypcaractexpol-1}
\int_{|y|>1} \frac1{|\psi_\xi(y)|} dy < \infty. 
\end{equation}
Since $\xi$ is symmetric, \eqref{hypcaractexpol-1} coincides with conditions (43.5) and (43.6) from \cite{Sato}. The later imply absolute continuity of potential measures (see Remark 43.6 in \cite{Sato}, see also Lemma \ref{linkpsizpotential} in the Appendix), which has important consequences in potential theory. In particular, it implies that any one point set is reached by $\xi$ with positive probability (see Theorems 43.3 and 43.5 in \cite{Sato}). 
Note also that \eqref{hypcaractexpol-1} implies that $\xi$ is of type C in the sense of Definition 11.9 of \cite{Sato}. Indeed, if $\xi$ was of type $A$ or $B$, then from Lemma 43.11 in \cite{Sato} we would have $|\psi_\xi(y)|= o(|y|)$ as $|y|$ is large, so \eqref{hypcaractexpol-1} would not be satisfied. Recall from Theorem 21.9 of \cite{Sato} that a L\'evy process is of type $C$ in the sense of Definition 11.9 of \cite{Sato} if and only if its sample paths on finite intervals have infinite variation almost surely. 
Under \eqref{hypcaractexpol-1}, for any $r>0$, we denote by $v^r_\xi(\cdot)$ the potential density of $\xi^r$ (see Lemma \ref{linkpsizpotential} for an expression and some properties). One can see that, under \eqref{hypcaractexpol-1}, $\mathbb{E}[\zeta_x] = \mathbb{E}[A_x(e_r)]=\int_{\mathbb{R}} v^r_\xi(y)m(x+y)dy$. This suggests that, if we assume \eqref{hypcaractexpol-1}, we may hope for an expression of $\gamma_{m,\xi,r}$ in term of $m$ and $v^r_\xi(\cdot)$ as well. In our case, the potential density plays an important role in our results and in our methodology, so we assume \eqref{hypcaractexpol-1}. According to Theorem 43.9 of \cite{Sato}, \eqref{hypcaractexpol-1} is never satisfied by L\'evy processes in $\mathbb{R}^d$ for $d>1$. This is why we restrict our study to the case of real valued processes. Finally, if $\xi$ is a symmetric $\alpha$-stable L\'evy process on $\mathbb{R}$ with $\alpha>1$, then $-\psi_\xi(y) = c |y|^{\alpha}$ for some $c>0$ (see for example Theorem 14.15 in \cite{Sato}) so \eqref{hypcaractexpol-1} is clearly satisfied. 

\begin{remark} [A parallel with CSBP's]
\eqref{hypcaractexpol-1} may seem familiar as it resembles Grey's condition for a CSBP, which involves the Laplace exponent of the underlying L\'evy process and is necessary and sufficient for extinction of the CSBP in finite time with positive probability \cite{grey}. Moreover, under the Grey condition, it is possible to derive exact expressions related to a CSBP, for example for the probability of non-extinction until time $t$ \cite{grey}. 
\end{remark}

In the literature CSBP's are sometimes represented via SDEs. It allows to define their extensions and to study them via stochastic analysis technics \cite{dawsonli}, \cite{gencsbp}, \cite{csbppardopalau}, \cite{csbplihexu}. In our case, we can also give an alternative definition of $X^{m,\xi,r}_x$ as follows: Let $(X(t))_{t \geq 0}$ be the solution of the SDE 
\begin{align}
X(t) = x + \int_0^t \sqrt{\frac{A_{\xi}}{m(X(s))}} dW(s) & + \int_0^t \int_{[-1,1]} \int_0^{\frac1{m(X(s-))}} z \tilde M_1(ds,dz,du) \label{cbregeneral} \\
& + \int_0^t \int_{\mathbb{R}\setminus [-1,1]} \int_0^{\frac1{m(X(s-))}} z M_2(ds,dz,du), \nonumber
\end{align}
where $W$, $\tilde M_1(ds,dz,du)$ and $M_2(ds,dz,du)$ are independent. $W$ is a standard Brownian motion. $M_1(ds,dz,du)$ (resp. $M_2(ds,dz,du)$) is a Poisson random measure on $[0,\infty) \times [-1,1] \times [0,\infty)$ (resp. $[0,\infty) \times (\mathbb{R}\setminus [-1,1]) \times [0,\infty)$) with intensity measure $\textbf{1}_{z \in [-1,1]} ds \times \Pi_{\xi}(dz) \times du$ (resp. $\textbf{1}_{z \notin [-1,1]} ds \times \Pi_{\xi}(dz) \times du$), and $\tilde M_1(ds,dz,du) := M_1(ds,dz,du) - \textbf{1}_{z \in [-1,1]}ds \times \Pi_{\xi}(dz) \times du$. Let $e_r$ be an independent exponential random variable with parameter $r$. Provided that the SDE \eqref{cbregeneral} is well-posed we set $X^{m,\xi,r}_x(t) := X(t)$ when $\int_0^t (1/m(X(s)))ds < e_r$ and $X^{m,\xi,r}_x(t) := \dagger$ when $\int_0^t (1/m(X(s)))ds \geq e_r$. The $X^{m,\xi,r}_x$ obtained by this procedure has the same law as the one defined by \eqref{defX}. We will not justify the well-posedness of the SDE \eqref{cbregeneral}, nor the equivalence of the two definitions. We only work with the definition \eqref{defX} of $X^{m,\xi,r}$ (via \eqref{defsgx}), which allows us to conveniently exploit properties of the L\'evy process $\xi$. 

\begin{remark}
If $\xi$ is a symmetric $\alpha$-stable L\'evy process on $\mathbb{R}$ with $\alpha>1$, then we see by the time-change method that the solution of \eqref{cbregeneral} is equal in law to the solution of the more compact SDE $dX(t) = m(X(t))^{-1/\alpha} d\xi(t)$. 
\end{remark}

$X^{m,\xi,r}$ is defined via \eqref{defX} and a choice of $\xi$, $m$ and $r$ satisfying "$\xi$ is symmetric", $r>0$, \eqref{mheavytailed} and \eqref{hypcaractexpol-1}. Some classical families of Markov processes can be represented via randomly time-changed L\'evy processes, as in \eqref{defX}, but with different assumptions on $\xi$, $m$ and $r$. Let us recall, for some of these classical families, which choices of $\xi$, $m$ and $r$ they correspond to. 
\begin{itemize}
\item \textit{pssMp's:} There is no restriction on the L\'evy process $\xi$, one takes $m(x):=e^{\alpha x}$, where $\alpha>0$ is the self-similarity index, and $r \geq 0$. The pssMp is then defined as the exponential of the resulting time-changed L\'evy process \cite{Lamperti2}, \cite{5e0b08215cf94b4f9738d908edd90e6f}, \cite{pardosurvey}. See also \cite{selfsimvech} for generalizations of pssMp's where the L\'evy process $\xi$ is replaced by more general processes involving exponential functionals of bivariate L\'evy processes. 
\item \textit{CSBP's:} $\xi$ has to be chosen as a L\'evy process with no negative jumps killed upon hitting $\{0\}$, one takes $m(x)=x^{-1}$ and $r=0$, see \cite{Lamperti3}, \cite{Lambert2008PopulationDA}. 
\item \textit{Generalizations of CSBP's:} One can see that the process considered in \cite{polcsbp} corresponds to taking $\xi$ and $r$ as for CSBP's but $m(x)=x^{-\theta}$ for some $\theta>0$. In \cite{gencsbp}, their process is parametrized by three functions $\gamma_0,\gamma_1,\gamma_2$. The particular case $\gamma_0=\gamma_1=\gamma_2$ corresponds to taking $\xi$ and $r$ as for CSBP's and $m(x)=1/\gamma_0(x)$. 
\end{itemize}

\subsection{Some spin systems} \label{interpartsyst}

We now step in a framework of equilibrium statistical mechanics and define the spin system that will be related to the process $X^{m,\xi,r}$. All the spin system-related objects mentioned in this subsection will later be seen to be related to properties of the process $X^{m,\xi,r}$ (or its ground state transform), and thus allow to draw a connection between the later and the spin system. To introduce the model we adopt (and recall) the terminology of equilibrium statistical mechanics that is used for example in \cite{friedlivelenik2017}. We fix $n \geq 2$ and consider a spin system on the one dimensional box $\Lambda_n:=\{1,\dots,n\} \subset \mathbb{Z}$. Each site $j \in \Lambda_n$ is viewed as a particle and has a \textit{spin} $\omega_j \in \mathbb{R}$. A configuration of the system is determined by the spins of all particles, that is, by a vector $(\omega_1,\dots,\omega_n) \in \mathbb{R}^{\Lambda_n}$. $\mathbb{R}^{\Lambda_n}$ is thus called the \textit{configuration space} of the system. To each configuration $(\omega_1,\dots,\omega_n) \in \mathbb{R}^{\Lambda_n}$ we associated an \textit{energy} $\mathcal{H}_{\Lambda_n}(\omega_1,\dots,\omega_n)$. We assume nearest-neighbors interactions in the system (with periodic boundary condition) that favors agreement of the spins of neighboring sites. For this, we more precisely assume that the contribution to the energy of the pair of neighboring sites $\{j,j+1\} \subset \Lambda_n$ is given by $\pot(\omega_{j+1}-\omega_j)$ where $\pot$ is assumed to be a continuous even function that converges to $\infty$ at $\infty$. Spin systems with this form of interactions fall into the category of \textit{gradient models}. Additionally, we penalize large values of the spins by assuming that each site $j \in \Lambda_n$ also adds a contribution $\pen(\omega_j)$ to the energy, where $\pen$ is assumed to be a continuous function that converges to $\infty$ at $\infty$. The energy of the configuration $(\omega_1,\dots,\omega_n) \in \mathbb{R}^{\Lambda_n}$ is thus given by 
\begin{align}
\mathcal{H}_{\Lambda_n}(\omega_1,\dots,\omega_n) = \pot(\omega_2-\omega_1) + \dots + \pot(\omega_{n}-\omega_{n-1}) + \pot(\omega_1-\omega_{n}) + \pen(\omega_1) + \dots + \pen(\omega_n). \label{hamilt} 
\end{align}
The function $\mathcal{H}_{\Lambda_n} : \mathbb{R}^{\Lambda_n} \rightarrow \mathbb{R}$ is classically called the \textit{Hamiltonian}. That function allows to define a probability measure called the \textit{Gibbs distribution} on the configuration space $\mathbb{R}^{\Lambda_n}$ via 
\begin{align}
\mu_{\Lambda_n}(A) = \frac1{\partfct_n} \int_{A} e^{-\mathcal{H}_{\Lambda_n}(\omega_1,\dots,\omega_n)} d\omega_1\dots d\omega_n, \label{gibbs}
\end{align}
for any Borel set $A \subset \mathbb{R}^{\Lambda_n}$. In the above expression, $\partfct_n$ denotes the \textit{partition function} of the system of $n$ particles and is defined by 
\begin{align}
\partfct_n := \int_{\mathbb{R}^{\Lambda_n}} e^{-\mathcal{H}_{\Lambda_n}(\omega_1,\dots,\omega_n)} d\omega_1\dots d\omega_n. \label{defzntrans} 
\end{align}
Of course, the Gibbs measure $\mu_{\Lambda_n}(\cdot)$ from \eqref{gibbs} is well-defined only if $\partfct_n<\infty$. We will work with particular choices of $\pot$ and $\pen$ that ensure $\partfct_n<\infty$ for all $n \geq 2$. The \textit{normalized free energy} of the system is defined by 
\begin{align}
\mathcal{E} := \lim_{n \rightarrow \infty} -\frac1{n} \log (\partfct_{n}), \label{deffreeenergytrans}
\end{align}
when the latter limit exists. That quantity is often related to some physical properties of spin systems. The model presented above ranges in the category of \textit{effective interface models}. Such models often play the role of approximations for interfaces occurring in more realistic models. A particularly famous example of effective interface model is the \textit{lattice Gaussian Free Field} (GFF). In particular, the lattice GFF on $\Lambda_n$ corresponds to the above model with the choice $\pot(x)=\beta x^2$ and $\pen(x) = \lambda x^2$, for some $\beta,\lambda \geq 0$ (see Chapter 8 of \cite{friedlivelenik2017}). 

Among objects of interest are the limit distributions of spins in fixed finite regions, as the size $n$ of the global system goes to infinity. The distribution of the spins in the region $\{1,\dots,k\}$ within the system of size $n$ is determined by the quantities 
\begin{align}
L^k_n(f_1,\dots,f_k):= \int_{\mathbb{R}^{\Lambda_n}} f_1(\omega_1)\dots f_k(\omega_k) \mu_{\Lambda_n}(d\omega_1\dots d\omega_n) = \frac1{\partfct_n} \int_{\mathbb{R}^{\Lambda_n}} f_1(\omega_1)\dots f_k(\omega_k) e^{-\mathcal{H}_{\Lambda_n}(\omega_1,\dots,\omega_n)} d\omega_1\dots d\omega_n, \label{defznf} 
\end{align}
for $f_1,\dots,f_k \in \mathcal{C}_b(\mathbb{R})$ (where $\mathcal{C}_b(\mathbb{R})$ denotes the space of bounded continuous functions $f : \mathbb{R} \rightarrow \mathbb{C}$). 
In equilibrium statistical mechanics, a function of the spins of a finite set of particles, such as $\prod_{j=1}^k f_j(\omega_j)$, is classically called a \textit{local observable} and quantities such as $L^k_n(f_1,\dots,f_k)$ are called \textit{finite volume thermal averages of local observables} (see for example \cite{quencorrfunct}). For any $k\geq 1$, $(L^k_n)_{n \geq 2 \vee k}$ defines a sequence of probability measures on $\mathbb{R}^k$. For any $n\geq 2$ we see that $L^1_n(d\omega_1)=l_1^n(\omega_1)d\omega_1$ where 
\begin{align}
l_1^n(\omega_1) := \frac{1}{\partfct_n} \int_{\mathbb{R}^{n-1}} e^{-\mathcal{H}_{\Lambda_n}(\omega_1,\dots,\omega_n)} d\omega_2\dots d\omega_n. \label{defdnesityln} 
\end{align}
Since we are interested in the infinite volume limit, we consider 
\begin{align}
\mathcal{L}_k(d\omega_1 \dots d\omega_k) := \lim_{n \rightarrow \infty} L^k_n(d\omega_1 \dots d\omega_k), \label{deflimstaeonetyppart}
\end{align}
when the latter limit exists for the convergence in distributions. The limit laws $\mathcal{L}_k$ are often called \textit{infinite-volume Gibbs states}. Determining such limits of marginals is closely related to the problem of identifying \textit{infinite volume Gibbs measures} (see Chapters 6 and 8 in \cite{friedlivelenik2017}). However, the study of infinite volume Gibbs measures is beyond the scope of this article so we will not develop that aspect. 

In equilibrium statistical mechanics, it is also of interest to study the decay rate of correlations of local observables, since it provides information on correlation lengths in the global system. Let us consider, in the infinite volume limit system, observables of the spins of two particles that are at distance $k$ from each other (for example the first and the $(k+1)^{th}$ particle). We assume that those observables are given by functions $f,g \in \mathcal{C}_b(\mathbb{R})$ of the spins of the two particles. Then their correlation is given by 
\begin{align}
C_k(f,g) := \int_{\mathbb{R}^{k+1}}f(\omega_1)g(\omega_{k+1})\mathcal{L}_{k+1}(d\omega_1 \dots d\omega_{k+1})-\left (\int_{\mathbb{R}}f(\omega)\mathcal{L}_{1}(d\omega) \right ) \left (\int_{\mathbb{R}}g(\omega)\mathcal{L}_{1}(d\omega) \right ). \label{defcorrel}
\end{align}
An important question is to determine the exact decay rate of those correlations as $k$ goes to infinity. 

Let us notice a convenient duality satisfied by our spin systems. 
\begin{remark} [Fourier duality for the partition function] \label{altexpr}
Let $m$ be a function satisfying \eqref{mheavytailed1} and let $\hat m$ denote the characteristic function of the measure $m(x) dx$, ie $\hat m(z) :=\int_{\mathbb{R}} e^{-izx}m(x)dx$. Fix $r>0$, and let $\xi$ be a symmetric L\'evy process satisfying \eqref{hypcaractexpol-1}. Then under the choice $\pot:=-\log(v^r_\xi(\cdot))$ and $\pen:=-\log(m(\cdot))$ we have that, for all $n\geq 2$, $\partfct_n$ (defined by \eqref{defzntrans}) is well-defined and satisfies $\partfct_n = \hat \partfct_n /(2\pi)^n$ where 
\begin{align}
\hat \partfct_n := \int_{\mathbb{R}^n} \frac{\hat m(z_2-z_1) \times \dots \times \hat m(z_{n}-z_{n-1}) \times \hat m(z_1-z_n)}{(r-\psi_\xi(z_1)) \times \dots \times (r-\psi_\xi(z_n))} dz_1\dots dz_n. \label{altexpr1}
\end{align}
This is justified in Appendix \ref{proofaltexpr}. In particular we have 
\begin{align}
\partfct_2 = \frac1{(2\pi)^2} \int_{\mathbb{R}^2} \frac{|\hat m(z_1-z_2)|^2}{(r-\psi_\xi(z_1)) \times (r-\psi_\xi(z_2))} dz_1 dz_2. \label{altexpr2}
\end{align}
\end{remark}
Remark \ref{altexpr} can be interpreted as a duality between a system and a dual system, obtained by taking the Fourier transforms of $e^{-\pen}$ and $e^{-\pot}$ and reversing their roles. Under some conditions of symmetry and positivity, the dual system can also satisfy the requirements of our definitions. In Section \ref{rgoesto0}, this duality will facilitate the study of the behavior of the normalized free energy $\mathcal{E}$ as $r$ goes to $0$, in order to prove Corollary \ref{asymptrgoto0} below. 

In the spin system we introduced above, a rather natural choice for the interaction function $\pot$ would be $\pot(x)=\beta |x|$, where the parameter $\beta$ is usually interpreted physically as the inverse of the temperature. A key observation is that this can be related to the potential density of a killed Brownian motion. Indeed, when $\xi^r$ is a standard Brownian motion killed at rate $r$ then $v^r_\xi(x)=e^{-\sqrt{2r} |x|}/\sqrt{2r}$ (see e.g. Example 30.11 of \cite{Sato}). This leads us to naturally wonder how the spin system with interaction function $\pot(x)=-\beta |x|$ can be related to a killed Brownian motion. More generally, this suggests that negative logarithms of potential densities of general symmetric L\'evy processes are good candidates to generalize classical interaction functions in spin systems, and motivates investigating the relation between 1) the systems with those general interaction functions and 2) the corresponding killed L\'evy processes. In the following subsection we state our main results that establish the connection between the two objects. To the best of our knowledge, such a connection between the two problems is new (both in the Brownian case and in the general case). Let us mention that a limited parallel can be made between the connection we just mentioned and the so-called \textit{random walk representations} for spin systems. Those representations connect spin systems to random walks on their lattices and are fundamental tools for the study of the lattice GFF in particular and of other spin systems. Good references on such representations are \cite{roberto1992}, \cite{rwrepres2000}, \cite{cmp/1103920749}, and Chapter 8 of \cite{friedlivelenik2017} for the particular case of the lattice GFF. The connection between a spin system and its associated random walk allows to get a hand on correlations in the system and usually involves the potential of the random walk (through its Green function). Those aspects bear some similarities with the connection we present in the following subsection. However, one of the limits to the analogy is that, in our case, the spin system is connected to a process on $\mathbb{R}$ instead of a process living on the same lattice. 

\subsection{Main results} \label{mainres}

The following result establishes the exponential decay of the survival probability $\mathbb{P}(\zeta_x > t)$ and relates in particular the decay rate with the normalized free energy of the spin system defined in Section \ref{interpartsyst}. 
\begin{theo} \label{encadrementtransfreeenergythtrans}
Assume that $r>0$ and \eqref{mheavytailed} and \eqref{hypcaractexpol-1} hold. There is a positive constant $\gamma_{m,\xi,r}$ and, for any $x\in \mathbb{R}$, a positive constant $K_{m,\xi,r}(x)$, such that 
\begin{align}
\mathbb{P}(\zeta_x > t) \underset{t \rightarrow \infty}{\sim} K_{m,\xi,r}(x) e^{-t\gamma_{m,\xi,r}}. \label{expodecaysurvproba}
\end{align}
Under the choice $\pot:=-\log(v^r_\xi(\cdot))$ and $\pen:=-\log(m(\cdot))$, the coefficients $(\partfct_n)_{n\geq 2}$ and the normalized free energy $\mathcal{E}$ (defined in respectively \eqref{defzntrans} and \eqref{deffreeenergytrans}) are well-defined and we have 
\begin{align}
\gamma_{m,\xi,r} = e^{\mathcal{E}}, \label{freeenergyth0trans}
\end{align}
and the infinite-volume Gibbs states $\mathcal{L}_k$ (defined in \eqref{defdnesityln}) are well-defined and absolutely continuous for all $k \geq 1$. Let us denote their densities by $\ell_k$. Then $\ell_1/m \in \crz \cap L^1(\mathbb{R})$, $\ell_1$ is positive on $\mathbb{R}$, $\ell_1$ is the point-wise limit of the density $l_1^n$ (defined in \eqref{defdnesityln}) as $n$ goes to infinity, and 
\begin{align}
K_{m,\xi,r}(x) = \sqrt{\frac{\ell_1(x)}{m(x)}} \int_{\mathbb{R}} \sqrt{m(z) \ell_1(z)} dz. \label{microstaedensity}
\end{align}
If moreover we assume that \eqref{mheavytailed1} holds then we additionally have the bound 
\begin{align}
\frac1{\|m\|_{L^1(\mathbb{R})}} \times \frac1{v^r_\xi(0)} \leq \gamma_{m,\xi,r} \leq \frac{\|m\|_{L^1(\mathbb{R})} \times v^r_\xi(0)}{\partfct_2}. \label{inegspecgaptrans}
\end{align}
\end{theo}

\begin{remark} \label{perboundcond}
The periodic boundary condition in the spin system appears naturally from our study of $X^{m,\xi,r}$, however that boundary condition can be removed. Indeed, if for $n\geq 2$ we define 
\begin{align}
\mathcal{H}^f_{\Lambda_n}(\omega_1,\dots,\omega_n) = \pot(\omega_2-\omega_1) + \dots + \pot(\omega_{n}-\omega_{n-1}) + \pen(\omega_1) + \dots + \pen(\omega_n), \label{perboundcondexpr}
\end{align}
$\partfct_{n}^f := \int_{\mathbb{R}^{\Lambda_n}} e^{-\mathcal{H}^f_{\Lambda_n}(\omega_1,\dots,\omega_n)} d\omega_1\dots d\omega_n$ and $\mathcal{E}_f$ as the limit of $-\log (\partfct_{n}^f)/n$, then under the assumptions of Theorem \ref{encadrementtransfreeenergythtrans} and the choice $\pot:=-\log(v^r_\xi(\cdot))$ and $\pen:=-\log(m(\cdot))$, $\mathcal{E}_f$ is also well-defined and we have $\mathcal{E}_f=\mathcal{E}$. This is justified in Appendix \ref{proofremovboundcond}. 
\end{remark}

In the following result, we characterize the asymptotic behavior of $\gamma_{m,\xi,r}$ as $r$ goes to $0$ when the characteristic exponent $\psi_\xi(\cdot)$ is regular enough near zero. 
\begin{cor} [Behavior for small $r$] \label{asymptrgoto0}
Assume that \eqref{mheavytailed1} and \eqref{hypcaractexpol-1} hold. Assume moreover that there are $\alpha \geq 1$ and $c>0$ such that $-\psi_\xi(y) \sim c |y|^{\alpha}$ as $y$ goes to $0$. Then, for the coefficient $\gamma_{m,\xi,r}$ from Theorem \ref{encadrementtransfreeenergythtrans} we have 
\begin{align}
\gamma_{m,\xi,r} \underset{r \rightarrow 0}{\sim} \frac1{\|m\|_{L^1(\mathbb{R})}} \times \frac{c^{\frac1{\alpha}} \pi}{\int_0^{\infty} \frac{1}{1+u^{\alpha}} du} r^{\frac{\alpha-1}{\alpha}} \ \ \ & \text{if} \ \alpha > 1, \label{asymptrgoto0equivregvar} \\
\gamma_{m,\xi,r} \underset{r \rightarrow 0}{\sim} \frac1{\|m\|_{L^1(\mathbb{R})}} \times c \pi \left ( \log \left ( \frac1{r} \right ) \right )^{-1} \ \ \ & \text{if} \ \alpha = 1. \label{asymptrgoto0equivregvaralpha=1}
\end{align}
\end{cor}
Corollary \ref{asymptrgoto0} shows in particular that the bound $\gamma_{m,\xi,r} \geq r/\|m\|_{\infty}$ (resulting from the trivial inequality $\mathbb{P}(\zeta_x > t) \leq e^{-rt/\|m\|_{\infty}}$) is bad for small values of $r$. We give, in Section \ref{rgoesto0}, a proof of Corollary \ref{asymptrgoto0} that mainly relies on spin system-related objects. More precisely, it relies on the inequality \eqref{inegspecgaptrans} that bounds the free energy in terms of the partition function, and on the duality from Remark \ref{altexpr} that allows to determine the asymptotic of the partition function as $r$ goes to $0$. A byproduct of this is that \eqref{inegspecgaptrans} is optimal for small values of $r$. 

\begin{example} \label{casel2}
Assume that \eqref{mheavytailed1} and \eqref{hypcaractexpol-1} hold and that $\mathbb{E}[\xi(1)^2]<\infty$. In this case, since $\mathbb{E}[\xi(1)]=0$ by the symmetry of $\xi$, we have $-\psi_\xi(y) \sim_{y \rightarrow 0} \mathbb{E}[\xi(1)^2] y^2/2$. By Corollary \ref{asymptrgoto0} applied with $c=\mathbb{E}[\xi(1)^2]/2$ and $\alpha=2$ we get 
\begin{align*}
\gamma_{m,\xi,r} \underset{r \rightarrow 0}{\sim} \frac{\sqrt{2r\mathbb{E}[\xi(1)^2]}}{\|m\|_{L^1(\mathbb{R})}}. 
\end{align*}
\end{example}

\begin{example} \label{casestable}
Assume that \eqref{mheavytailed1} holds and that $\xi$ is a symmetric $\alpha$-stable L\'evy process on $\mathbb{R}$ with $\alpha>1$, then $-\psi_\xi(y) = c |y|^{\alpha}$ for some $c>0$ (see for example Theorem 14.15 in \cite{Sato}). In particular all the assumptions from Corollary \ref{asymptrgoto0} are satisfied (with this $\alpha$ and $c$) so $\gamma_{m,\xi,r}$ satisfies the estimate \eqref{asymptrgoto0equivregvar}. 
\end{example}

\begin{remark}
Under more general assumptions on the L\'evy process $\xi$, an equivalent for $\gamma_{m,\xi,r}$ as $r$ goes to $0$ will be provided in Remark \ref{corocaserec} from Section \ref{rgoesto0}. 
\end{remark}

The following theorem establishes further spectral properties of the semigroup of the process $X^{m,\xi,r}$. $L^2(m)$ denotes the space of measurable functions $f : \mathbb{R} \rightarrow \mathbb{C}$ such that $\int_{\mathbb{R}} |f(y)|^2 m(y) dy < \infty$. 
\begin{theo} \label{diagogenetsg}
Assume that $r>0$ and \eqref{mheavytailed} and \eqref{hypcaractexpol-1} hold. The following claims hold: 
\begin{itemize}
\item $(P_t)_{t \geq 0}$ extends uniquely to a strongly continuous contraction semigroup on $L^2(m)$ (we still denote this extension by $(P_t)_{t \geq 0}$) and, for $t>0$, $P_t(L^2(m))\subset \crz$. 
\item There is an orthonormal Hilbert basis $(h_n)_{n\geq 1}$ of $L^2(m)$ such that for any $n \geq 1$, $h_n \in \crz \cap L^2(\mathbb{R}) \subset L^2(m)$ and $P_t.h_n = e^{-t\lambda^X_{n}} h_n$ where $0 < \lambda^X_{1} \leq \lambda^X_{2} \leq ...$. For any $f \in L^2(m)$ and $t>0$ we have 
\begin{align}
\forall x \in \mathbb{R}, \ P_t.f(x) = \sum_{n \geq 1} e^{- t\lambda^X_{n}} \psl f, h_n \psrm h_n(x), \label{decompsgbonexpr0}
\end{align}
where the series of functions in the right-hand side converges absolutely in $(\crz, \| \cdot \|_{\infty})$. 
\item We have $\lambda^X_{1} < \lambda^X_{2}$ and 
\begin{align}
\lambda^X_{1} = \gamma_{m,\xi,r}, \qquad \qquad h_1(x) = \sqrt{\frac{\ell_1(x)}{m(x)}}, \label{ident1stvpandfp}
\end{align}
where $\gamma_{m,\xi,r}$ and $\ell_1$ are as in Theorem \ref{encadrementtransfreeenergythtrans}. 
\item If moreover we assume that \eqref{mheavytailed1} holds then $(P_t)_{t \geq 0}$ has the strong Feller property in the sense that, for any $t>0$, we have $P_t(\mathcal{B}_b(\mathbb{R})) \subset \crz$. 
\end{itemize}
\end{theo}

The function $h_1$ is sometimes referred to as the \textit{ground state} and the quantity $\lambda^X_{1}$ as the \textit{ground state energy}. Note from \eqref{ident1stvpandfp} and Theorem \ref{encadrementtransfreeenergythtrans} that $h_1$ is positive on $\mathbb{R}$ so, under the assumptions of that theorem, we can define a family of measures $(\tilde p_t(x,dy))_{t \geq 0, x \in \mathbb{R}}$ on $\mathbb{R}$ by 
\begin{align}\label{defsgxgstrans}
\tilde p_t(x,dy):= e^{t\lambda^X_{1}} \frac{h_1(y)}{h_1(x)} \mathbb{P}(X^{m,\xi,r}_x(t) \in dy). 
\end{align}
The following proposition justifies that a Markov process can be associated to $(\tilde p_t(x,dy))_{t \geq 0, x \in \mathbb{R}}$. Such a process is classically called a \textit{ground state transform}. 

\begin{prop} \label{fellerityofgroundstatetrans}
Assume that $r>0$ and \eqref{mheavytailed} and \eqref{hypcaractexpol-1} hold. There exists an homogeneous Markov process $\tilde X^{m,\xi,r}$ on $\mathbb{R}$ such that $\mathbb{P}(\tilde X^{m,\xi,r}_x(t) \in dy)=\tilde p_t(x,dy)$ for any $t\geq 0$ and $x \in \mathbb{R}$. This process $\tilde X^{m,\xi,r}$ is unique, up to equality in law. 
\end{prop}

We now state results that show how the spin system from Section \ref{interpartsyst} is related to the Markov process $\tilde X^{m,\xi,r}$. In particular, the following proposition shows that the infinite-volume Gibbs states $\mathcal{L}_{k}$ can be represented from the process $\tilde X^{m,\xi,r}$ taken at arrival times of a Poisson process. As mentioned in the end of Section \ref{interpartsyst}, that representation can be paralleled with the random walk representations for spin systems, or also with the fact that the lattice GFF in dimension one can be represented by a random walk with Gaussian increments (see Exercise 8.6 in \cite{friedlivelenik2017}). 
\begin{prop} \label{deeperconnectiongs-ips}
Assume that $r>0$ and \eqref{mheavytailed} and \eqref{hypcaractexpol-1} hold. Let the distributions $(\mathcal{L}_k)_{k \geq 1}$ be as in Theorem \ref{encadrementtransfreeenergythtrans}. Let $(T_j)_{j \geq 1}$ be the sequence of arrival times of a standard Poisson process with parameter $\lambda^X_{1}$, independent of $\tilde X^{m,\xi,r}$. Set $T_0:=0$ for convenience and denote $\tilde X^{m,\xi,r}_{\mathcal{L}_1}$ for the process $\tilde X^{m,\xi,r}$ with initial distribution $\mathcal{L}_1(dy)$. Then for any $k\geq 1$ the random vector $(\tilde X^{m,\xi,r}_{\mathcal{L}_1}(T_0),\dots,\tilde X^{m,\xi,r}_{\mathcal{L}_1}(T_k))$ has law $\mathcal{L}_{k+1}$. 
\end{prop}

The following corollary shows that the correlations (defined in \eqref{defcorrel}) in the spin system from Section \ref{interpartsyst} decay exponentially in $k$. 
\begin{cor} [Correlation decay] \label{cordecay}
Assume that $r>0$ and \eqref{mheavytailed} and \eqref{hypcaractexpol-1} hold. Under the choice $\pot:=-\log(v^r_\xi(\cdot))$ and $\pen:=-\log(m(\cdot))$ in the spin system from Section \ref{interpartsyst}, there exist $\mathcal{C}>0$ and a continuous non-zero bilinear form $B(\cdot,\cdot) : \mathcal{C}_b(\mathbb{R}) \times \mathcal{C}_b(\mathbb{R}) \rightarrow \mathbb{C}$ such that for any $f,g \in \mathcal{C}_b(\mathbb{R})$ we have 
\begin{align}
e^{k \mathcal{C}} C_k(f,g) \underset{k \rightarrow \infty}{\longrightarrow} B(f,g). \label{cordecayequiv}
\end{align}
\end{cor}

The following theorem shows that the rate of correlation decay in the spin system from Section \ref{interpartsyst} is related to the spectral gap of the process $\tilde X^{m,\xi,r}$. 
\begin{theo} \label{ergoofgroundstatetrans}
Assume that $r>0$ and \eqref{mheavytailed} and \eqref{hypcaractexpol-1} hold. $\tilde X^{m,\xi,r}$ is ergodic and its unique stationary distribution is $\mathcal{L}_1(dy)$ (where $\mathcal{L}_1(dy)$ is as in Theorem \ref{encadrementtransfreeenergythtrans}). For any $f \in \mathcal{C}_b(\mathbb{R})$ and $t \geq 0$ we have the following spectral gap inequality 
\begin{align}
\int_{\mathbb{R}} \left | \mathbb{E} \left [ f \left ( \tilde X^{m,\xi,r}_x(t) \right ) \right ] - \int_{\mathbb{R}} f(y) \mathcal{L}_1(dy) \right |^2 \mathcal{L}_1(dx) \leq e^{- 2tc} \int_{\mathbb{R}} \left | f(x) - \int_{\mathbb{R}} f(y) \mathcal{L}_1(dy) \right |^2 \mathcal{L}_1(dx), \label{specgapineg}
\end{align}
where $c:=e^{\mathcal{E}}(e^{\mathcal{C}}-1)$. Here $\mathcal{E}$ is the normalized free energy, as in Theorem \ref{encadrementtransfreeenergythtrans}, and $\mathcal{C}$ is the decay rate of correlations, as in Corollary \ref{cordecay}. Moreover, the choice $c=e^{\mathcal{E}}(e^{\mathcal{C}}-1)$ is the largest choice of $c$ such that \eqref{specgapineg} holds true for all $f \in \mathcal{C}_b(\mathbb{R})$. 
\end{theo}

Let us finish this section with two particular examples where several quantities and functions from the above theorems can be obtained explicitly. 
\begin{example} \label{exepl1}
Let $m(x):=1/(1+|x|)$, $\xi := B$, where $B$ denotes the standard Brownian motion, and $r=1/2$ (this corresponds to choosing $\pot(x)=|x|$ and $\pen(x)=\log(1+|x|)$ in the spin system from Section \ref{interpartsyst}). Then we can show that the ground state $h_1$ from Theorem \ref{diagogenetsg} has the simple expression $h_1(x)=\sqrt{2/3}(1+|x|)e^{-|x|}$ and that $\lambda^X_{1}=1$. This is justified in Section \ref{proofexepl1}. As a consequence, we get from \eqref{ident1stvpandfp}, \eqref{freeenergyth0trans} and \eqref{microstaedensity} that the quantities from Theorem \ref{encadrementtransfreeenergythtrans} have the following expressions: $\gamma_{m,B,1/2}=1$, $\mathcal{E}=0$, $\ell_1(x)=2(1+|x|)e^{-2|x|}/3$ and $K_{m,B,1/2}(x)=4(1+|x|)e^{-|x|}/3$. 
\end{example}

\begin{example} \label{exepl2}
Let $m(x):=(2x^2+6|x|+3)/(1+|x|)^4$, $\xi := B$, where $B$ denotes the standard Brownian motion, and $r=1/2$ (this corresponds to choosing $\pot(x)=|x|$ and $\pen(x)=4\log(1+|x|)-\log(2x^2+6|x|+3)$ in the spin system from Section \ref{interpartsyst}). Then we can show that the ground state $h_1$ from Theorem \ref{diagogenetsg} has the simple expression $h_1(x)=ce^{-x^2/(1+|x|)}$ where $c$ is a normalizing constant, and that $\lambda^X_{1}=1/2$. The justification of this is similar to the justification of Example \ref{exepl1} in Section \ref{proofexepl1}. As a consequence, we get from \eqref{ident1stvpandfp}, \eqref{freeenergyth0trans} and \eqref{microstaedensity} that the quantities from Theorem \ref{encadrementtransfreeenergythtrans} have the following expressions: $\gamma_{m,B,1/2}=1/2$, $\mathcal{E}=-\log(2)$, $\ell_1(x)=c^2(2x^2+6|x|+3)e^{-2x^2/(1+|x|)}/(1+|x|)^4$ and $K_{m,B,1/2}(x)=\tilde c e^{-x^2/(1+|x|)}$, for some constant $\tilde c>0$. 
\end{example}
The above examples illustrate in particular cases the fact that, whenever we are able to get information on the eigenfunctions and eigenvalues of the generator of $X^{m,\xi,r}$, this yields information on the associated spin system. More generally, the connections built in the above results between, on one hand the processes $X^{m,\xi,r}$ and $\tilde X^{m,\xi,r}$, and on the other hand the associated spin system, allow to transform information obtained on one of these two objects into information on the other. 

\subsection{Sketch of proofs and organization of the paper}

In Section \ref{optrdual} we study some linear operators that provide some understanding on the generator of $X^{m,\xi,r}$. More precisely, in Subsection \ref{defopbasprop} we introduce those operators and study their basic properties. In Subsection \ref{refren} we study the traces of the successive compositions of one of these operators. In Subsection \ref{dualwithx} we prove that a duality holds between that operator and the generator of $X^{m,\xi,r}$. 

In Section \ref{decompsg} we basically diagonalize the semigroup $(P_t)_{t \geq 0}$ of $X^{m,\xi,r}$ and deduce the asymptotic behavior of the survival probability $\mathbb{P}(\zeta_x > t)$. More precisely, in Subsection \ref{basiseigfctgen} we use the duality from Section \ref{optrdual} to construct and study a basis of eigenfunctions of the generator of $X^{m,\xi,r}$ on some natural Hilbert space, and to study the sequence of eigenvalues of that generator. In Subsection \ref{decompsgbasis} we deduce a decomposition of the semigroup $(P_t)_{t \geq 0}$ and we study the special properties of the first eigenvalue and its associated eigenfunction. In Subsection \ref{decompsurvprobandasympt} we apply the previous results to obtain a representation of the survival probability $\mathbb{P}(\zeta_x > t)$ and determine its asymptotic behavior. In Subsection \ref{relwithtrofr} we study how the first eigenvalue of the generator of $X^{m,\xi,r}$ is related to the spin system from Section \ref{interpartsyst}. In Subsection \ref{limdistribforlocalobs} we study the infinite-volume Gibbs states $\mathcal{L}_k$ defined in \eqref{deflimstaeonetyppart} and relate them to the first eigenfunction of the generator of $X^{m,\xi,r}$. 

In Section \ref{proofmainresults} we prove the main results on $X^{m,\xi,r}$ from Section \ref{mainres}. More precisely, in Subsection \ref{studysurvproba} we gather the pieces collected in Section \ref{decompsg} to prove Theorem \ref{encadrementtransfreeenergythtrans}. In Subsection \ref{rgoesto0} we prove Corollary \ref{asymptrgoto0} by studying the potential density of the killed L\'evy process $\xi$ at $0$ and using the duality from Remark \ref{altexpr} and the inequality \eqref{inegspecgaptrans}. In Subsection \ref{diagogenetsgproof} we prove Theorem \ref{diagogenetsg} from the results of Section \ref{decompsg}. In Subsection \ref{proofexepl1} we justify Example \ref{exepl1}. 

In Section \ref{gstxrelips} we prove the rest of the main results from Section \ref{mainres} that are about the ground state transform of $X^{m,\xi,r}$ and its relation with the spin system from Section \ref{interpartsyst}. More precisely, in Subsection \ref{exgst} we prove Proposition \ref{fellerityofgroundstatetrans}, in Subsection \ref{represlktildex} we prove Proposition \ref{deeperconnectiongs-ips}, in Subsection \ref{specgapgstandcorr} we prove Theorem \ref{ergoofgroundstatetrans}, and in Subsection \ref{seccorrel} we prove Corollary \ref{cordecay}. 

In Section \ref{factsaboutx} we establish properties of the process $X^{m,\xi,r}$ that are used all along the paper. More precisely, in Subsection \ref{5.1time-change} we study properties of the random time change $A_x^{-1}(\cdot)$. In Subsection \ref{5.2Markov-Feller} we proves the Markov property for $X^{m,\xi,r}$ and the Feller property (Proposition \ref{fellerity}). In Subsection \ref{5.3support} we study the support of $X^{m,\xi,r}$. In Subsection \ref{5.3generator} we study the generator of $X^{m,\xi,r}$. In Subsection \ref{5.4semigroup} we study some technical properties of the semigroup $(P_t)_{t \geq 0}$ of $X^{m,\xi,r}$. 

Appendix \ref{factsLP} contains some technical facts about the real symmetric L\'evy process $\xi$, Appendices \ref{proofaltexpr}, and \ref{proofremovboundcond} contain the proofs of Remarks \ref{altexpr} and \ref{perboundcond} respectively. 

\subsection{Another route for the survival probability} \label{anotherroute}

The spectral properties of the semigroup of $X^{m,\xi,r}$ are a main focus of this paper and the direct object of several of our results. This is why our approach is mainly spectral, in the sense that it relies substantially on establishing and using spectral properties of the semigroup of $X^{m,\xi,r}$. However, for the asymptotic of the survival probability of $X^{m,\xi,r}$ studied in Theorem \ref{encadrementtransfreeenergythtrans}, other approaches are also possible. One can start by noticing that the moments of the survival time $\zeta_x$ are related to the spin system from Section \ref{interpartsyst}. Indeed, recall from Section \ref{afortclp} that $\zeta_x=A_x(e_r)$ so, since $A_x(\cdot)$ is an additive functional of a Markov process, we can use formula (4) from \cite{FITZSIMMONS1999117} and get 
\begin{align}
\frac{\mathbb{E}[(\zeta_x)^n]}{n!}= \int_{\mathbb{R}^n} v^r_\xi(\omega_1-x) v^r_\xi(\omega_2-\omega_1) \dots v^r_\xi(\omega_{n}-\omega_{n-1}) m(\omega_1)\dots m(\omega_n) d\omega_1\dots d\omega_n. \label{momentexpression}
\end{align}
The right-hand side is quite similar to the expression \eqref{defzntrans} of the partition function $\partfct_n$ with the choice $\pot:=-\log(v^r_\xi(\cdot))$ and $\pen:=-\log(m(\cdot))$, but with a different boundary condition. Let us denote by $\tilde \partfct_n(x)$ the right-hand side of \eqref{momentexpression}. This suggests that the following alternative route may be used to prove \eqref{expodecaysurvproba}-\eqref{freeenergyth0trans} (and maybe \eqref{microstaedensity}) or results in the same spirit: 
\begin{enumerate}
\item Prove directly that $-\log (\tilde \partfct_{n}(x))/n$ converges to a limit that does not depend on $x$ and that the free energy $\mathcal{E}$ (defined in \eqref{deffreeenergytrans}) is well-defined, and identify the limit of $-\log (\tilde \partfct_{n}(x))/n$ with $\mathcal{E}$. 
\item Using $\mathbb{E}[(\zeta_x)^n]=n! \tilde \partfct_{n}(x)$, the previous point, and the Cauchy–Hadamard theorem, show that the Laplace transform of $\zeta_x$ has a pole at $e^{\mathcal{E}}$ and study relevant properties of this pole. 
\item Prove or use an appropriate Tauberian theorem to deduce an estimate similar to \eqref{expodecaysurvproba} for the right distribution tail of $\zeta_x$. \eqref{freeenergyth0trans} would then follow from the pole being located at $e^{\mathcal{E}}$. 
\item Prove directly the existence of the infinite-volume Gibbs state $\mathcal{L}_1$ (defined in \eqref{defdnesityln}) and of its positive density $\ell_1$, and identify the constant, in the equivalent of $\mathbb{P}(\zeta_x > t)$ found in the previous point, with $\sqrt{\ell_1(x)/m(x)} \int_{\mathbb{R}} \sqrt{m(z) \ell_1(z)} dz$. This would yield \eqref{microstaedensity}. 
\end{enumerate}
It is certainly not too difficult to carry out at least some steps of the above route by (or along with) using spectral properties of the semigroup of $X^{m,\xi,r}$, but in this case the approach that we use in this paper seems to be more direct. See Remark \ref{khasminskii} for a moment-based proof of the lower bound in \eqref{inegspecgaptrans}, provided that \eqref{expodecaysurvproba} has been proved. We are not sure how easy or difficult it would be to carry out each point of the route outlined above without making any use of spectral properties, and if further results like Corollary \ref{asymptrgoto0} can be obtained in this way. 

\subsection{Facts and notations} \label{notations}

Let $\mathcal{F}^{\xi}:=(\mathcal{F}^{\xi}_t)_{t \geq 0}$ and $\mathcal{F}^X:=(\mathcal{F}^X_t)_{t \geq 0}$ denote the right continuous filtrations generated by respectively the L\'evy process $\xi$ and the process $X^{m,\xi,r}_x$. 

Let $\crz$ denote the space of continuous functions $f : \mathbb{R} \rightarrow \mathbb{C}$ that converge to $0$ at $\infty$ and $-\infty$. $\crz$ is equipped with the norm $\| \cdot \|_{\infty}$, defined by $\| f \|_{\infty} := \sup_{x \in \mathbb{R}} |f(x)|$, which make it a Banach space. We sometimes denote $(\crz, \| \cdot \|_{\infty})$ to emphasis that we consider $\crz$ equipped with this norm. For $k \geq 1$, let $\mathcal{C}_0^k(\mathbb{R})$ denote the space of functions $f\in\crz$ that are $k$ times differentiable and whose successive derivatives of order $1,2,...,k$ belong to $\crz$. $L^1(\mathbb{R})$ denotes the space of measurable functions $f : \mathbb{R} \rightarrow \mathbb{C}$ such that $\|f\|_{L^1(\mathbb{R})} := \int_{\mathbb{R}} |f(y)| dy < \infty$. $\mathcal{B}_b(\mathbb{R})$ and $\mathcal{C}_b(\mathbb{R})$ denote respectively the spaces of bounded measurable functions and bounded continuous functions $f : \mathbb{R} \rightarrow \mathbb{C}$. 

$L^2(\mathbb{R})$ and $L^2(m)$ denote the Hilbert spaces of measurable functions $f : \mathbb{R} \rightarrow \mathbb{C}$ such that, respectively, $\int_{\mathbb{R}} |f(y)|^2 dy < \infty$ and $\int_{\mathbb{R}} |f(y)|^2 m(y) dy < \infty$. We denote by $\psl \cdot , \cdot \psrr$ and $\psl \cdot , \cdot \psrm$ respectively the natural inner products on these spaces. We also denote $\|f\nr := (\psl f,f \psrr)^{1/2}$. $L^2((-1,1))$ denotes the Hilbert space of measurable functions $g : (-1,1) \rightarrow \mathbb{C}$ such that $\int_{(-1,1)} |g(y)|^2 dy < \infty$. We denote by $\psl \cdot, \cdot \psri$ the natural inner product on $L^2((-1,1))$ and $\|g \nri := (\psl g,g \psri)^{1/2}$. 

For the Fourier transform we use the convention 
\begin{align}
(\mathcal{F} f)(x) := \int_{\mathbb{R}} f(y) e^{-2i\pi y x} dy, \ \ \ (\mathcal{F}^{-1} f)(x) := \int_{\mathbb{R}} f(y) e^{2i\pi y x} dy, \ \ \ \text{for} \ f \in L^1(\mathbb{R}). \label{convfouriertransform0}
\end{align}
According to Riemann-Lebesgue lemma, $\mathcal{F}$ and $\mathcal{F}^{-1}$ are contractions from $L^1(\mathbb{R})$ to $(\crz, \| \cdot \|_{\infty})$. The Fourier inversion theorem says that, if $f \in L^1(\mathbb{R})$ is such that $\mathcal{F} f \in L^1(\mathbb{R})$ (equivalently, $\mathcal{F}^{-1} f \in L^1(\mathbb{R})$) then $\mathcal{F}^{-1}(\mathcal{F} f)=\mathcal{F}(\mathcal{F}^{-1} f)=f$. With the convention \eqref{convfouriertransform0}, the extensions of $\mathcal{F}$ and $\mathcal{F}^{-1}$ to $L^2(\mathbb{R})$ (that we also denote by $\mathcal{F}$ and $\mathcal{F}^{-1}$) are isometric and satisfy Plancherel's identity, i.e. for any $f,g \in L^2(\mathbb{R})$ we have $\psl f,g \psrr = \psl \mathcal{F}f, \mathcal{F}g \psrr = \psl \mathcal{F}^{-1}f, \mathcal{F}^{-1}g \psrr$. 

We denote the null function by $\textbf{0}$. For any functional Banach space $\mathcal{U}$ equipped with a norm $\|\cdot\|_{\mathcal{U}}$, we denote the operator norm by $\||\cdot\||$, i.e. for any continuous operator $T : \mathcal{U} \rightarrow \mathcal{U}$, $\||T\|| := \sup_{f \in \mathcal{U} \setminus \{\textbf{0}\}} \|T.f\|_{\mathcal{U}}/\|f\|_{\mathcal{U}}$. We denote by $\mathcal{L}(\mathcal{U})$ the space of continuous operators of $\mathcal{U}$ equipped with the operator norm, which makes it a Banach space. For an operator $T$ and $n \geq 1$, $T^n$ denotes the operator obtained by $n$ successive compositions of $T$. 

According to Proposition \ref{fellerity}, $(P_t)_{t \geq 0}$ is a Feller semigroup under \eqref{mheavytailed} and $r>0$. As mentioned after Proposition \ref{fellerity}, we will always make those assumptions when dealing with $(P_t)_{t \geq 0}$ so that, in particular, $(P_t)_{t \geq 0}$ enjoys all the properties of Feller semigroups. We refer to \cite{levymatters3} for background on the theory of Feller semigroups and their generators. The domain $\mathcal{D}(\mathcal{A}_{X^{m,\xi,r}})$ is the subspace of $\crz$ defined by $\{ f \in \crz \ \text{s.t.} \ \lim_{t \rightarrow 0} t^{-1}(P_t.f-f) \ \text{exists in} \ (\crz, \| \cdot \|_{\infty}) \}$. The generator $\mathcal{A}_{X^{m,\xi,r}} : \mathcal{D}(\mathcal{A}_{X^{m,\xi,r}}) \rightarrow \crz$ is defined by $\mathcal{A}_{X^{m,\xi,r}}f := \lim_{t \rightarrow 0} t^{-1}(P_t.f-f)$ for $f \in \mathcal{D}(\mathcal{A}_{X^{m,\xi,r}})$. $\mathcal{D}(\mathcal{A}_{X^{m,\xi,r}})$ is dense in $(\crz, \| \cdot \|_{\infty})$ (see for example Theorem 3.2.6 in \cite{applebaum_2009}). 
It is known (see for example Theorem 3.2.6 in \cite{applebaum_2009}) that, if $g \in \mathcal{D}(\mathcal{A}_{X^{m,\xi,r}})$, then $P_t.g \in \mathcal{D}(\mathcal{A}_{X^{m,\xi,r}})$ for any $t\geq 0$ and $\frac{d}{dt} P_t.g = \mathcal{A}_{X^{m,\xi,r}}P_t.g=P_t.\mathcal{A}_{X^{m,\xi,r}}g$. We refer to this as the differentiation rule for the semigroup. It is important to note that the limit involved in the definitions of this derivative is a limit in $(\crz, \| \cdot \|_{\infty})$. 

It is classical that the un-killed (resp. killed) L\'evy process $\xi$ (resp. $\xi^r$) is also a Feller process, that is, its semigroup is Feller in the sense of Definition 1.2 of \cite{levymatters3}. We denote its generator by $\mathcal{A}_{\xi}$ (resp. $\mathcal{A}_{\xi^r}$) and the domain of its generator by $\mathcal{D}(\mathcal{A}_{\xi})$ (resp. $\mathcal{D}(\mathcal{A}_{\xi^r})$). Note that $\mathcal{D}(\mathcal{A}_{\xi})=\mathcal{D}(\mathcal{A}_{\xi^r})$ and that $\mathcal{A}_{\xi^r}f=\mathcal{A}_{\xi}f-rf$. For any $\alpha> 0$, the $\alpha$-potential measure of $\xi$ is defined by $V^{\alpha}_\xi(dx) := \int_0^{\infty} e^{-\alpha t} \mathbb{P} (\xi(t) \in dx) dt$. 
The resolvent operator at $\alpha$ associated with $\xi$ is defined by $V^{\alpha}_\xi f(x) := \int_{\mathbb{R}} f(x+y) V^{\alpha}_\xi(dy)$ for $f \in \crz$. 

\section{Some operators, their traces, and a duality} \label{optrdual}

\subsection{Definition of the operators and basic properties} \label{defopbasprop}

We now introduce some definitions. Assume that $r>0$ and \eqref{hypcaractexpol-1} holds and let 
\[ \varphi_{r}(x):= \frac{2 }{v^{r}_\xi(0)} \int_0^x \frac1{- \psi_\xi(2\pi y) + r} dy. \]
We see from Lemma \ref{linkpsizpotential} and the symmetry of $\psi_\xi$ that we have $\varphi_{r}(\infty)=1$ and $\varphi_{r}(-\infty)=-1$. Let $\varphi_{r}^{-1}: (-1,1) \rightarrow \mathbb{R}$ be the inverse function of $\varphi_{r}$. For any $y \in \mathbb{R}$, let $\phi^{r}_y:(-1,1) \rightarrow \mathbb{C}$ be the function defined by $\phi^{r}_y(x):=e^{-2i\pi y \varphi_{r}^{-1}(x)}/\sqrt{2}$. It can be seen using dominated convergence that $(y \mapsto \phi^{r}_y)$ is continuous from $\mathbb{R}$ to $L^2((-1,1))$. 

Let $Q_{r} : L^2(\mathbb{R}) \rightarrow L^2(\mathbb{R})$ be the operator defined via $Q_{r}.f := \mathcal{F}^{-1}((- \psi_\xi(2\pi \cdot) + r)^{-1/2} \times f)$. Note that $(- \psi_\xi(2\pi \cdot) + r)^{-1/2}$ is bounded by $r^{-1/2}$. Therefore, if $f \in L^2(\mathbb{R})$, $(- \psi_\xi(2\pi \cdot) + r)^{-1/2} \times f \in L^2(\mathbb{R})$ so $\mathcal{F}^{-1}((- \psi_\xi(2\pi \cdot) + r)^{-1/2} \times f) \in L^2(\mathbb{R})$, i.e. $Q_{r}.f \in L^2(\mathbb{R})$. $Q_{r}$ is thus well defined. Also, it is easy to check that $\mathsf{Ker}\ Q_{r} = \{\textbf{0}\}$. 

The following lemma states some properties of $Q_{r}$ and relates it with $\phi^{r}_y$. 
\begin{lemme} \label{opnormq=rsrdualnew}
Assume that $r>0$ and \eqref{hypcaractexpol-1} holds. 
\begin{enumerate}
\item $Q_{r}$ is a bounded operator and $\||Q_{r}\||^2 \leq 1/r$. 
\item $Q_{r}(L^2(\mathbb{R}))\subset \crz$. In particular, for any $f \in L^2(\mathbb{R})$ and $y \in \mathbb{R}$, $Q_{r}.f(y)$ is defined for every $y \in \mathbb{R}$ and is continuous. 
\item There is a bijective isometry $J : L^2(\mathbb{R}) \rightarrow L^2((-1,1))$ such that 
\begin{align}
\forall f \in L^2(\mathbb{R}), \forall y \in \mathbb{R}, \ Q_{r}.f(y) = \sqrt{v^{r}_\xi(0)} \psl J.f,\phi^{r}_y \psri. \label{normqfdansbase10new}
\end{align}
\end{enumerate}
\end{lemme}

An important point of the proof of the third point is to choose a suitable Hilbert basis of $L^2(\mathbb{R})$ on which $Q_{r}$ acts nicely. The action of $Q_{r}$ on the elements of the basis is then determined by noticing a duality \eqref{feksurspi=phiygknew} between the Fourier transforms of some functions built from the basis, and Fourier coefficients of the functions $\phi^{r}_y$. 

\begin{proof}[Proof of Lemma \ref{opnormq=rsrdualnew}]
Using the definition of $Q_{r}$ and Plancherel's identity we get
\begin{align}
\|Q_{r}.f\nr^2 & = \|\mathcal{F}^{-1}((- \psi_\xi(2\pi \cdot) + r)^{-1/2} \times f)\|_{L^2(\mathbb{R})}^2 = \|(- \psi_\xi(2\pi \cdot) + r)^{-1/2} \times f\|_{L^2(\mathbb{R})}^2 \nonumber \\
& \leq \|(- \psi_\xi(2\pi \cdot) + r)^{-1/2}\|_{\infty}^2.\|f\|_{L^2(\mathbb{R})}^2=\|f\|_{L^2(\mathbb{R})}^2/r. \label{bornaltqr}
\end{align}
This proves the first claim. 

Note from \eqref{hypcaractexpol-1} that $(- \psi_\xi(2\pi \cdot) + r)^{-1/2} \in L^2(\mathbb{R})$. Therefore, if $f \in L^2(\mathbb{R})$, $(- \psi_\xi(2\pi \cdot) + r)^{-1/2} \times f \in L^1(\mathbb{R})$ by Cauchy-Schwartz inequality so $Q_{r}.f=\mathcal{F}^{-1}((- \psi_\xi(2\pi \cdot) + r)^{-1/2} \times f) \in \crz$. This proves the second claim. 

Let us define 
\begin{align}
e_n(x) := \sqrt{\frac{1}{v^{r}_\xi(0)}} \times \frac{e^{i n \pi \varphi_{r}(x)}}{\sqrt{- \psi_\xi(2\pi x) + r}}, \ n \in \mathbb{Z}. \label{defen}
\end{align}
It is not difficult to check that $(e_n)_{n \in \mathbb{Z}}$ is an Hilbert basis of $L^2(\mathbb{R})$. Let $(g_n)_{n \in \mathbb{Z}}$ be the classical Hilbert basis of $L^2((-1,1))$, namely, $g_n(x):=e^{i n \pi x}/\sqrt{2}$ for $n \in \mathbb{Z}$. We define an isometry $J : L^2(\mathbb{R}) \rightarrow L^2((-1,1))$ via $J.e_n:=g_n$ for $n \in \mathbb{Z}$ and linear isometrical extension. Since $Q_{r}$ is continuous on $ L^2(\mathbb{R})$ we have for any $f \in L^2(\mathbb{R})$ that
\begin{align}
Q_{r}.f = \sum_{k \in \mathbb{Z}} \psl f,e_k \psrr Q_{r}.e_k. \label{normqfdansbasenew}
\end{align}
We have $(- \psi_\xi(2\pi \cdot) + r)^{-1/2} \times e_k \in L^2(\mathbb{R}) \cap L^1(\mathbb{R})$ so, for $k\in \mathbb{Z}$ and $y \in \mathbb{R}$ we have 
\begin{align}
Q_{r}.e_k(y) = \left (\mathcal{F}^{-1} \frac{e_k}{\sqrt{- \psi_\xi(2\pi \cdot) + r}} \right )(y) & = \sqrt{\frac{1}{v^{r}_\xi(0)}} \int_{\mathbb{R}} \frac{e^{i k \pi \varphi_{r}(x)}}{- \psi_\xi(2\pi x) + r} e^{2i\pi xy} dx \nonumber \\
& = \sqrt{v^{r}_\xi(0)} \int_{-1}^1 \frac{e^{2i\pi y \varphi_{r}^{-1}(u)}}{\sqrt{2}} \frac{e^{i k \pi u}}{\sqrt{2}} du \nonumber \\
& = \sqrt{v^{r}_\xi(0)} \times \overline{\psl \phi^{r}_y, g_{k} \psri}, \label{feksurspi=phiygknew}
\end{align}
where we have made the change of variable $u=\varphi_{r}(x)$. Since $\psl f,e_k \psrr = \psl J.f,J.e_k \psri = \psl J.f,g_{k} \psri$, we get for any $y \in \mathbb{R}$, 
\begin{align*}
\sum_{k \in \mathbb{Z}} \psl f,e_k \psrr Q_{r}.e_k(y) & = \sqrt{v^{r}_\xi(0)} \sum_{k \in \mathbb{Z}} \psl J.f,g_{k} \psri \overline{\psl \phi^{r}_y, g_{k} \psri} = \sqrt{v^{r}_\xi(0)} \psl J.f, \phi^{r}_y \psri, 
\end{align*}
where the last equality comes from Parseval's identity. Combining the above with \eqref{normqfdansbasenew} we get that \eqref{normqfdansbase10new} holds for a.e. $y \in \mathbb{R}$. Since $Q_{r}.f \in \crz$ and $(y \mapsto \phi^{r}_y)$ is continuous, both sides of the equality are continuous so \eqref{normqfdansbase10new} holds for every $y \in \mathbb{R}$. This proves the third claim. 
\end{proof}

\begin{lemme} \label{contintpsfphiy}
Assume that $r>0$ and \eqref{hypcaractexpol-1} holds. There is a constant $C>0$ such that for any $h \in L^2((-1,1))$ we have $\int_{\mathbb{R}} |\psl \phi^{r}_y,h \psri |^2 dy \leq C \|h \nri^2$. 
\end{lemme}

\begin{proof}
Lemma \ref{opnormq=rsrdualnew} shows that the claim holds with $C=1/r v^{r}_\xi(0)$. 
\end{proof}

For any $y \in \mathbb{R}$, let $P^{r}_y : L^2((-1,1)) \rightarrow L^2((-1,1))$ be defined by $P^{r}_y.f := \psl f,\phi^{r}_y \psri \phi^{r}_y$. Since $\| \phi^{r}_y \nri=1$, $P^{r}_y$ is the orthogonal projection on $Span \{ \phi^{r}_y \}$ and $\||P^{r}_y\|| = 1$. From the continuity of $y \mapsto \phi^{r}_y$ we can see that $y \mapsto P^{r}_y$ is continuous from $\mathbb{R}$ to $\mathcal{L}(L^2((-1,1)))$. 

The following lemma defines a linear operator $R_{r}$ and shows some of its useful properties. 
\begin{lemme} \label{oprwelldeefinedanddiag}
Assume that $r>0$ and \eqref{mheavytailed} and \eqref{hypcaractexpol-1} hold. 
\begin{enumerate}
\item For any $h \in L^2((-1,1))$ we have $\int_{\mathbb{R}} m(y) \| P^{r}_y.h \nri dy<\infty$. We can thus define an operator $R_{r}: L^2((-1,1)) \rightarrow L^2((-1,1))$ by $R_{r}.h:= \int_{\mathbb{R}} m(y) (P^{r}_y.h) dy$, where the integral is in the sense of Bochner. 
\item For any $f,g \in L^2((-1,1))$ we have $\psl R_{r}.f,g \psri = \int_{\mathbb{R}} m(y) \psl P^{r}_y.f,g \psri dy$. 
\item $R_{r}$ is a bounded operator, $\mathsf{Ker}\ R_{r} = \{\textbf{0}\}$, and $R_{r}$ is compact and self-adjoint in $L^2((-1,1))$. In particular there is a Hilbert basis of $L^2((-1,1))$ consisting of eigenfunctions of $R_{r}$. 
\item The eigenvalues of $R_{r}$ are real, positive, and form a sequence converging to $0$. Let $\lambda^R_{1} \geq \lambda^R_{2} \geq \dots$ be the ordered sequence of eigenvalues of $R_{r}$, repeated with their respective multiplicities. 
\end{enumerate}
\end{lemme}

\begin{proof}
Using the definition of $P^{r}_y$ we get 
\begin{align}
\forall h \in L^2((-1,1)), \ |\psl \phi^{r}_y,h \psri |^2=\psl P^{r}_y.h,h \psri =\| P^{r}_y.h \nri^2. \label{usefulid}
\end{align}
Therefore Lemma \ref{contintpsfphiy} can be re-written as $\int_{\mathbb{R}} \| P^{r}_y.h \nri^2 dy \leq C \|h \nri^2$. Combining this with $m \in L^2(\mathbb{R})$ (from \eqref{mheavytailed}) and Cauchy-Schwartz inequality we get 
\begin{align}
\int_{\mathbb{R}} m(y) \| P^{r}_y.h \nri dy \leq \sqrt{C} \|m\nr \|h \nri<\infty, \label{bochintwelldef}
\end{align}
which proves the first claim of the lemma. By properties of Bochner integrals, if $T$ is a continuous linear operator from $L^2((-1,1))$ to another Banach space, then the Bochner integral $\int_{\mathbb{R}} m(y) (T.P^{r}_y.h) dy$ is well-defined and we have $T.R_{r}.h=\int_{\mathbb{R}} m(y) (T.P^{r}_y.h) dy$. For any $g \in L^2((-1,1))$, we can apply this to the continuous linear operator $\psl \cdot,g \psri$ from $L^2((-1,1))$ to $\mathbb{R}$, which yields the second claim of the lemma. 

By properties of Bochner integrals we have 
\begin{align}
\| R_{r}.h \nri \leq \int_{\mathbb{R}} m(y) \| P^{r}_y.h \nri dy. \label{inegtriboch}
\end{align}
Combining with \eqref{bochintwelldef} we get that $R_{r}$ is a bounded operator. Using the second claim of the lemma, \eqref{usefulid} and \eqref{normqfdansbase10new} from Lemma \ref{opnormq=rsrdualnew} we get for any $h \in L^2((-1,1))$, 
\begin{align}
\psl R_{r}.h,h \psri = \int_{\mathbb{R}} m(y) |\psl \phi^{r}_y,h \psri |^2 dy = \frac1{v^{r}_\xi(0)} \int_{\mathbb{R}} m(y) |Q_{r}.J^{-1}.h(y)|^2 dy, \label{relrrqr}
\end{align}
where $J^{-1}$ denotes the inverse of the bijective isometry $J$ from Lemma \ref{opnormq=rsrdualnew}. Since $m$ is positive on $\mathbb{R}$ (by \eqref{mheavytailed}), $\mathsf{Ker}\ Q_{r} = \{\textbf{0}\}$, and $J^{-1}$ is bijective we get $h \in \mathsf{Ker}\ R_{r} \Rightarrow h=\textbf{0}$ so $\mathsf{Ker}\ R_{r} = \{\textbf{0}\}$. The second claim of the lemma and the fact that $P^{r}_y$ is self-adjoint for all $y \in \mathbb{R}$ show that $R_{r}$ is self-adjoint as well. For $M>0$, since $m \in \crz$ (by \eqref{mheavytailed}) we have $\int_{[-M,M]} m(y) \|| P^{r}_y \|| dy\leq 2M\|m\|_{\infty}<\infty$. We can thus define an operator $R_{r,M} \in \mathcal{L}(L^2((-1,1)))$ by $R_{r,M}:= \int_{[-M,M]} m(y) P^{r}_y dy$, where the integral is in the sense of Bochner. Using \eqref{inegtriboch} with $R_{r} - R_{r,M}$ in place of $R_{r}$ and \eqref{bochintwelldef} with $m\textbf{1}_{[-M,M]^c}$ in place of $m$ we get $\||R_{r} - R_{r,M}\|| \leq \sqrt{C} \|m \textbf{1}_{[-M,M]^c}\nr$ so $R_{r,M}$ converges to $R_{r}$ in $\mathcal{L}(L^2((-1,1)))$ as $M$ goes to infinity. By properties of Bochner integrals, $R_{r,M}$ (and therefore $R_{r}$) can be approximated by finite linear combinations of operators $P^{r}_y$. Since each operator $P^{r}_y$ has rank one, $R_{r}$ is the limit of a sequence of finite rank operators and is thus compact. The existence of a Hilbert basis of $L^2((-1,1))$ consisting of eigenfunctions of $R_{r}$ then follows from the spectral theorem. This concludes the proof of the third claim of the lemma. 

That the eigenvalues of $R_{r}$ are real and form a sequence converging to $0$ follow from $R_{r}$ being self-adjoint and from the spectral theorem. From \eqref{relrrqr} we get that each eigenvalue of $R_{r}$ is non-negative. Since $\mathsf{Ker}\ R_{r} = \{\textbf{0}\}$, the eigenvalues of $R_{r}$ are even positive. This concludes the proof of the fourth claim of the lemma. 
\end{proof}

\begin{remark}
The compactness of $R_{r}$ alternatively follows from $R_{r}$ being a Hilbert-Schmidt operator, which follows from the first part of Lemma \ref{traces2new} below. 
\end{remark}

\subsection{Traces of compositions of $R_{r}$} \label{refren}

The following lemma prepares the ground to work with the traces of compositions of $R_{r}$. 

\begin{lemme} \label{traces1}
Assume that $r>0$ and \eqref{hypcaractexpol-1} holds. Let $(u_k)_{k \geq 1}$ be a Hilbert basis of $L^2((-1,1))$. For any $y \in \mathbb{R}$ and any bounded operator $H$ we have 
\begin{align*}
\sum_{j \geq 1} |\psl H.P^{r}_{y}.u_j,u_j \psri | \leq \||H\|| \ \text{and} \ \sum_{j \geq 1} \psl H.P^{r}_{y}.u_j,u_j \psri = \psl H.\phi^{r}_{y},\phi^{r}_{y} \psri. 
\end{align*}
\end{lemme}

\begin{proof}
By definition of $P^{r}_y$ we can see that we have $\psl H.P^{r}_{y}.u_j,u_j \psri = \psl H.\phi^{r}_{y},u_j \psri \times \overline{\psl \phi^{r}_{y},u_j \psri}$ so the first claim follows from Cauchy-Schwartz inequality for sums and Parseval's identity while the second claim follows from Parseval's identity for the inner product. 
\end{proof}

Recall that $R_{r}^n$ denotes the composition of $R_{r}$ by itself $n$ times. The following lemma shows that the traces of compositions of $R_{r}$ are well-defined and relate their asymptotic behaviors to $\lambda^R_{1}$. 
\begin{lemme} \label{traces2new}
Assume that $r>0$ and \eqref{mheavytailed} and \eqref{hypcaractexpol-1} hold. For any $n \geq 2$, and any Hilbert basis $(u_k)_{k \geq 1}$ of $L^2((-1,1))$ we have $\sum_{j \geq 1} |\psl R_{r}^n.u_j,u_j \psri | <\infty$ and the quantity $\sum_{j \geq 1} \psl R_{r}^n.u_j,u_j \psri$ does not depend on the choice of the Hilbert basis $(u_k)_{k \geq 1}$ so we denote it by $Tr(R_{r}^n)$. Moreover we have 
\begin{align}
Tr(R_{r}^n) \underset{n \rightarrow \infty}{\sim} N \times (\lambda_{1}^R)^n, \label{equivtraces}
\end{align}
where $N$ denotes the multiplicity of the eigenvalue $\lambda_{1}^R$ of $R_{r}$. 
\end{lemme}

\begin{remark} \label{mult1equiv}
The proof of Lemma \ref{q1posonr} below will show that the eigenvalue $\lambda_{1}^R$ of $R_{r}$ has multiplicity $1$. A consequence will be that we have actually $N=1$ in \eqref{equivtraces}. 
\end{remark}

\begin{proof} [Proof of Lemma \ref{traces2new}]
Fix $n \geq 2$. Using the second point of Lemma \ref{oprwelldeefinedanddiag} and Fubini's theorem we get 
\begin{align}
\sum_{j \geq 1} |\psl R_{r}^n.u_j,u_j \psri | \leq \int_{\mathbb{R}^{n}} m(y_1) \dots m(y_n) \left ( \sum_{j \geq 1} |\psl P^{r}_{y_{n}}\dots P^{r}_{y_1}.u_j,u_j \psri | \right ) dy_1 \dots dy_{n} =:I_n. \label{finitetracenew}
\end{align}
Using the definition of $P^{r}_y$ we get that $\psl P^{r}_{y_{n}}\dots P^{r}_{y_1}.u_j,u_j \psri $ equals 
\begin{align}
\psl u_j,\phi^{r}_{y_1} \psri \times \psl \phi^{r}_{y_1},\phi^{r}_{y_2} \psri \times \cdots \times \psl \phi^{r}_{y_{n-1}},\phi^{r}_{y_n} \psri \times \psl \phi^{r}_{y_n},u_j \psri. \label{prepfubinitraces1}
\end{align}
We now compute the terms $\psl \phi^{r}_{y_1},\phi^{r}_{y_2} \psri$. Using the definition of $\phi^{r}_y$, the substitution $x=\varphi_{r}^{-1}(u)$, and Lemma \ref{linkpsizpotential}, we get 
\begin{align}
\psl \phi^{r}_{y_1},\phi^{r}_{y_2} \psri & = \frac1{2} \int_{-1}^1 e^{2i\pi (y_2-y_1) \varphi_{r}^{-1}(u)} du = \frac{1}{v^{r}_\xi(0)} \int_{\mathbb{R}} \frac{e^{2i\pi (y_2-y_1) x}}{- \psi_\xi(2\pi x) + r} dx \label{exproftraces6} \\
& = \frac{1}{v^{r}_\xi(0)} \mathcal{F}^{-1} \left ( \frac1{- \psi_\xi(2\pi \cdot) + r} \right ) (y_2-y_1) = \frac{v^{r}_\xi(y_2-y_1)}{v^{r}_\xi(0)}. \nonumber
\end{align}
The combination of \eqref{prepfubinitraces1} and \eqref{exproftraces6} yields that $\sum_{j \geq 1} |\psl P^{r}_{y_{n}}\dots P^{r}_{y_1}.u_j,u_j \psri |$ is smaller than 
\begin{align*}
& \frac{v^{r}_\xi(y_2-y_1) \cdots v^{r}_\xi(y_n-y_{n-1})}{v^{r}_\xi(0)^{n-1}} \sum_{j \geq 1} |\psl u_j,\phi^{r}_{y_1} \psri| \times |\psl \phi^{r}_{y_n},u_j \psri| \\
\leq & \frac{v^{r}_\xi(y_2-y_1) \cdots v^{r}_\xi(y_n-y_{n-1})}{v^{r}_\xi(0)^{n-1}} \sqrt{\sum_{j \geq 1} |\psl u_j,\phi^{r}_{y_1} \psri|^2} \sqrt{\sum_{j \geq 1} |\psl \phi^{r}_{y_n},u_j \psri|^2} \\
= & \frac{v^{r}_\xi(y_2-y_1) \cdots v^{r}_\xi(y_n-y_{n-1})}{v^{r}_\xi(0)^{n-1}} \times \| \phi^{r}_{y_1} \nri \times \| \phi^{r}_{y_n} \nri = \frac{v^{r}_\xi(y_2-y_1) \cdots v^{r}_\xi(y_n-y_{n-1})}{v^{r}_\xi(0)^{n-1}}, 
\end{align*}
where we have used Cauchy-Schwartz inequality for sums, Parseval's identity, and $\| \phi^{r}_y \nri=1$. Plugging this in the definition of $I_n$ in \eqref{finitetracenew} we get 
\begin{align}
I_n & \leq \frac{1}{v^{r}_\xi(0)^{n-1}} \int_{\mathbb{R}^{n}} m(y_1) \dots m(y_n) v^{r}_\xi(y_2-y_1) \cdots v^{r}_\xi(y_n-y_{n-1}) dy_1 \dots dy_{n} \label{inegiteres} \\
& = \frac{1}{v^{r}_\xi(0)^{n-1}} \int_{\mathbb{R}} m(y_n) \left ( \dots \left ( \int_{\mathbb{R}} m(y_2) v^{r}_\xi(y_3-y_2) \left ( \int_{\mathbb{R}} m(y_1) v^{r}_\xi(y_2-y_1) dy_1 \right ) dy_2 \right ) \cdots \right ) dy_n \nonumber \\
& = \frac{1}{v^{r}_\xi(0)^{n-1}} \int_{\mathbb{R}} m(y) f_{n-1}(y) dy, \nonumber
\end{align}
where we have set $f_0(y):=1$ and $f_{k+1}(y):=\int_{\mathbb{R}} m(z) v^{r}_\xi(y-z) f_k(z) dz$ for $k \geq 0$. 

In order to prove that $I_n<\infty$, let us first show that for any function $f$, 
\begin{align}
m f \in L^2(\mathbb{R}) \Rightarrow \int_{\mathbb{R}} m(z) v^{r}_\xi(\cdot-z) f(z) dz \in L^2(\mathbb{R}). \label{claimforind}
\end{align}
Lemma \ref{linkpsizpotential} shows that $v^r_\xi(\cdot) \in L^1(\mathbb{R}) \cap \crz \subset L^2(\mathbb{R})$. Using Plancherel's identity we get 
\begin{align*}
\int_{\mathbb{R}} m(z) v^{r}_\xi(y-z) f(z) dz & =\int_{\mathbb{R}} \mathcal{F}^{-1}(m f)(x) \mathcal{F}^{-1}(v^{r}_\xi(y-\cdot))(x) dx =\int_{\mathbb{R}} e^{2i\pi yx} \mathcal{F}^{-1}(m f)(x) \mathcal{F}(v^{r}_\xi(\cdot))(x) dx. 
\end{align*}
By Plancherel's isometry we have $\mathcal{F}^{-1}(m f) \in L^2(\mathbb{R})$ and $\mathcal{F}(v^{r}_\xi(\cdot)) \in L^2(\mathbb{R})$ so the product of the two function is in $L^1(\mathbb{R})$ by Cauchy-Schwartz inequality. We can thus recognize the last term as the inverse Fourier transform of a function in $L^1(\mathbb{R})$ and get $\int_{\mathbb{R}} m(z) v^{r}_\xi(y-z) f(z) dz=\mathcal{F}^{-1}( \mathcal{F}^{-1}(m f) \times \mathcal{F}(v^{r}_\xi(\cdot)) )(y)$. 
Since, by Lemma \ref{linkpsizpotential}, $\mathcal{F}(v^{r}_\xi(\cdot))=(- \psi_\xi(2\pi \cdot) + r)^{-1}$ which is bounded by $1/r$ we get that the product $\mathcal{F}^{-1}(m f) \times \mathcal{F}(v^{r}_\xi(\cdot))$ is actually in $L^1(\mathbb{R}) \cap L^2(\mathbb{R})$. By Plancherel's isometry, its Fourier transform is in $L^2(\mathbb{R})$ and this concludes the proof of \eqref{claimforind}. 

Since $m \in \crz \cap L^2(\mathbb{R})$ by \eqref{mheavytailed}, \eqref{claimforind} shows that $f_1 \in L^2(\mathbb{R})$ and then shows by induction that $f_k \in L^2(\mathbb{R})$ for all $k \geq 1$. Therefore, $I_n<\infty$ follows from applying Cauchy-Schwartz inequality in the right-hand side of \eqref{inegiteres}. Combining this with \eqref{finitetracenew} yields the claim $\sum_{j \geq 1} |\psl R_{r}^n.u_j,u_j \psri | <\infty$. 

Then, using the second point of Lemma \ref{oprwelldeefinedanddiag}, Fubini's theorem, and Lemma \ref{traces1} we get 
\begin{align}
\sum_{j \geq 1} \psl R_{r}^n.u_j,u_j \psri & = \sum_{j \geq 1} \int_{\mathbb{R}^{n}} m(y_1) \dots m(y_{n}) \psl P^{r}_{y_{n}}\dots P^{r}_{y_1}.u_j,u_j \psri dy_1 \dots dy_{n} \label{exproftraces2new} \\
& = \int_{\mathbb{R}^{n}} m(y_1) \dots m(y_{n}) \left ( \sum_{j \geq 1} \psl P^{r}_{y_{n}}\dots P^{r}_{y_1}.u_j,u_j \psri \right ) dy_1 \dots dy_{n} \nonumber \\
& = \int_{\mathbb{R}^{n}} m(y_1) \dots m(y_{n}) \psl P^{r}_{y_{n}}\dots P^{r}_{y_2}.\phi^{r}_{y_1},\phi^{r}_{y_1} \psri dy_1 \dots dy_{n}, \nonumber
\end{align}
where the use of Fubini's theorem is allowed since $I_n<\infty$. This proves the independence with respect to the choice of the Hilbert basis $(u_k)_{k \geq 1}$. $Tr(R_{r}^n)$ is thus well-defined. 

Choosing, for $(u_k)_{k \geq 1}$, the Hilbert basis of eigenfunctions of $R_{r}$ we get $Tr(R_{r}^n) = \sum_{j \geq 1} (\lambda_{j}^{R})^{n}$ for all $n \geq 2$ so in particular, 
\begin{align}
\sum_{j \geq 1} (\lambda_{j}^{R})^{2}<\infty. \label{sumsquarrevpfinite}
\end{align}
For $n \geq 2$ we have 
\begin{align}
N \times (\lambda_{1}^{R})^{n} & \leq Tr(R_{r}^n) = \sum_{j \geq 1} (\lambda_{j}^{R})^{n} = (\lambda_{1}^{R})^{n} \left ( N + \sum_{j > N} \left (\frac{\lambda^{R}_{j}}{\lambda^{R}_{1}} \right )^{n} \right ) \leq (\lambda_{1}^{R})^{n} \left ( N + \frac1{(\lambda_{1}^{R})^{2}} \left ( \frac{\lambda_{N+1}^{R}}{\lambda_{1}^{R}} \right )^{n-2} \sum_{j > N} (\lambda_{j}^{R})^{2} \right ). \label{calcequivtrace}
\end{align}
Since $\lambda_{N+1}^{R}<\lambda_{1}^{R}$, the combination of \eqref{sumsquarrevpfinite} and \eqref{calcequivtrace} yields \eqref{equivtraces}. 
\end{proof}

The following lemma slightly improves Lemma \ref{traces2new} when the stronger assumption \eqref{mheavytailed1} holds. 
\begin{lemme} \label{traces2new1}
Assume that $r>0$ and \eqref{mheavytailed1} and \eqref{hypcaractexpol-1} hold. For any $n \geq 1$, and any Hilbert basis $(u_k)_{k \geq 1}$ of $L^2((-1,1))$ we have $\sum_{j \geq 1} |\psl R_{r}^n.u_j,u_j \psri | <\infty$ and the quantity $\sum_{j \geq 1} \psl R_{r}^n.u_j,u_j \psri$ does not depend on the choice of the Hilbert basis $(u_k)_{k \geq 1}$ so we denote it by $Tr(R_{r}^n)$. Moreover we have $Tr(R_{r}) =\|m\|_{L^1(\mathbb{R})}$ and for any $n \geq 1, Tr(R_{r}^n) = \sum_{j \geq 1} (\lambda_{j}^{R})^{n}<\infty$. 
\end{lemme}
Note that for $n\geq 2$ the result is a consequence of Lemma \ref{traces2new}. 
\begin{proof}
Fix $n\geq 1$. Note that for any $y \in \mathbb{R}$, $\||P^{r}_{y}\|| \leq 1$ because $P^{r}_{y}$ is an orthogonal projection. Proceeding as in \eqref{finitetracenew} and using Lemma \ref{traces1} and $m \in L^1(\mathbb{R})$ (from \eqref{mheavytailed1}) we get 
\begin{align*}
\sum_{j \geq 1} |\psl R_{r}^n.u_j,u_j \psri | & \leq \int_{\mathbb{R}^{n}} m(y_1) \dots m(y_n) \left ( \sum_{j \geq 1} |\psl P^{r}_{y_{n}}\dots P^{r}_{y_1}.u_j,u_j \psri | \right ) dy_1 \dots dy_{n} \\
& \leq \int_{\mathbb{R}^{n}} m(y_1) \dots m(y_n) dy_1 \dots dy_{n} = \|m\|_{L^1(\mathbb{R})}^n < \infty. 
\end{align*}
We thus get $\sum_{j \geq 1} |\psl R_{r}^n.u_j,u_j \psri | <\infty$. Moreover the above finiteness allows to use Fubini's theorem so \eqref{exproftraces2new} holds true in our case (for all $n\geq 1$). This proves the independence with respect to the choice of the Hilbert basis $(u_k)_{k \geq 1}$. $Tr(R_{r}^n)$ is thus well-defined. Choosing, for $(u_k)_{k \geq 1}$, the Hilbert basis of eigenfunctions of $R_{r}$ we get $Tr(R_{r}^n) = \sum_{j \geq 1} (\lambda_{j}^{R})^{n}<\infty$. 
Applying \eqref{exproftraces2new} with $n=1$ we get $Tr(R_{r}) = \int_{\mathbb{R}} m(y) \| \phi^{r}_{y} \nri^2 dy = \|m\|_{L^1(\mathbb{R})}$. 
\end{proof}

\subsection{A duality with $X^{m,\xi,r}$} \label{dualwithx}

\begin{prop} \label{strongduality}
Assume that $r>0$ and \eqref{mheavytailed} and \eqref{hypcaractexpol-1} hold. For any $h \in L^2((-1,1))$, 
\begin{align}
(y \mapsto \psl R_{r}.\phi^{r}_y,h \psri)\in \mathcal{D}(\mathcal{A}_{X^{m,\xi,r}}), \label{strongdualityin} \\
\forall y \in \mathbb{R}, \ \mathcal{A}_{X^{m,\xi,r}} \psl R_{r}.\phi^{r}_{\cdot},h \psri(y) = \frac{-1}{v^r_\xi(0)} \psl \phi^{r}_y,h \psri. \label{strongdualityexpr}
\end{align}
\end{prop}

\begin{proof}
Using the second point of Lemma \ref{oprwelldeefinedanddiag}, the definition of $P^{r}_z$, and \eqref{exproftraces6} we get 
\begin{align}
\psl R_{r}.\phi^{r}_y,h \psri & = \int_{\mathbb{R}} m(z) \psl P^{r}_z.\phi^{r}_y,h \psri dz = \int_{\mathbb{R}} m(z) \psl \phi^{r}_y,\phi^{r}_z \psri \psl \phi^{r}_z,h \psri dz \label{rphihegresgh} \\
& = \frac{1}{v^{r}_\xi(0)} \int_{\mathbb{R}} v^{r}_\xi(z-y) m(z) \psl \phi^{r}_z,h \psri dz = \frac{1}{v^{r}_\xi(0)} \int_{\mathbb{R}} v^{r}_\xi(u) G_h(y+u) du, \nonumber
\end{align}
where we have set $G_h(z):=m(z) \psl \phi^{r}_z,h \psri$. Let us prove that $(z \mapsto \psl \phi^{r}_z,h \psri) \in \crz$. We set $M_h(z):=h(\varphi_{r}(z))/(- \psi_\xi(2\pi z) + r)$. 
Since $h \in L^2((-1,1)) \subset L^1((-1,1))$ we have $\int_{\mathbb{R}} |M_h(u)| du = v^r_\xi(0) \int_{(-1,1)} |h(x)| dx/2 < \infty$ so $M_h \in L^1(\mathbb{R})$. Therefore $\mathcal{F}^{-1} M_h \in \crz$. Then we have 
\[ (\mathcal{F}^{-1} M_h)(z) = \int_{\mathbb{R}} e^{2i\pi u z} \frac{h(\varphi_{r}(u))}{- \psi_\xi(2\pi u) + r} du =  \frac{v^{r}_\xi(0)}{2} \int_{(-1,1)} h(x) e^{2i\pi z \varphi_{r}^{-1}(x)} dx = \frac{v^{r}_\xi(0)}{\sqrt{2}} \overline{\psl \phi^{r}_z,h \psri}. \]
Therefore $(z \mapsto \psl \phi^{r}_z,h \psri) \in \crz$. Since $m \in \crz$ (by \eqref{mheavytailed}) we get $G_h \in \crz$. $V^r_\xi G_h$ is thus well-defined (where $V^r_\xi$ is the resolvent operator at $r$, as in Section \ref{notations}) and, using Lemma \ref{linkpsizpotential}, we get 
\[ V^r_\xi G_h(y)=\int_{\mathbb{R}} G_h(y+u) V^{r}_\xi(du)=\int_{\mathbb{R}} v^{r}_\xi(u) G_h(y+u) du = v^{r}_\xi(0) \psl R_{r}.\phi^{r}_y,h \psri, \]
where we have used \eqref{rphihegresgh} for the last equality. Therefore $(y \mapsto \psl R_{r}.\phi^{r}_y,h \psri)\in V^r_\xi(\crz)$. According to Lemma 1.27 in \cite{levymatters3} we have $V^r_\xi(\crz) \subset \mathcal{D}(\mathcal{A}_{\xi})$ and for any $f \in \crz, \mathcal{A}_{\xi}V^r_\xi f-rV^r_\xi f=-f$. We thus get $(y \mapsto \psl R_{r}.\phi^{r}_y,h \psri)\in \mathcal{D}(\mathcal{A}_{\xi})=\mathcal{D}(\mathcal{A}_{\xi^r})$ and for all $y \in \mathbb{R}$, 
\begin{align}
\mathcal{A}_{\xi^r} \psl R_{r}.\phi^{r}_{\cdot},h \psri(y) = \frac{1}{v^r_\xi(0)} \left ( \mathcal{A}_{\xi} V^r_\xi G_h(y) - rV^r_\xi G_h(y) \right ) = -\frac{G_h(y)}{v^r_\xi(0)} = -\frac{m(y) \psl \phi^{r}_y,h \psri}{v^r_\xi(0)}. \label{axirphih}
\end{align}
Since $(z \mapsto \psl \phi^{r}_z,h \psri) \in \crz$ we have $\frac1{m}\mathcal{A}_{\xi^r} \psl R_{r}.\phi^{r}_{\cdot},h \psri \in \crz$ so, by Lemma \ref{generatorxnew}, we get \eqref{strongdualityin} and $\mathcal{A}_{X^{m,\xi,r}} \psl R_{r}.\phi^{r}_{\cdot},h \psri(y) = \frac{1}{m(y)} \mathcal{A}_{\xi^r} \psl R_{r}.\phi^{r}_{\cdot},h \psri(y)$. Combining this with \eqref{axirphih} we get \eqref{strongdualityexpr}. 
\end{proof}

\section{Decomposition of the survival probability of $X^{m,\xi,r}$} \label{decompsg}

\subsection{A basis of eigenfunctions of the generator} \label{basiseigfctgen}

According to Lemma \ref{oprwelldeefinedanddiag} we can, under the assumptions of that lemma, choose an orthonormal Hilbert basis $(a_n)_{n \geq 1}$ of $L^2((-1,1))$ consisting of eigenfunctions of $R_{r}$ such that $\lambda^R_{n}$ is the eigenvalue associated with $a_n$. For any $n \geq 1$ and $y \in \mathbb{R}$ let us define 
\begin{align}
q_n(y) := \frac1{\sqrt{\lambda^R_{n}}} \psl \phi^{r}_y,a_n \psri. \label{defqn}
\end{align}
The following result shows that the family $(q_n)_{n\geq 1}$ diagonalizes the generator $\mathcal{A}_{X^{m,\xi,r}}$. It is a consequence of the duality established in Proposition \ref{strongduality}. 
\begin{prop} \label{bonfctpropres}
Assume that $r>0$ and \eqref{mheavytailed} and \eqref{hypcaractexpol-1} hold. For any $n \geq 1$, $q_n \in \crz \cap L^2(\mathbb{R}) \subset L^2(m)$ and $\|q_n\nr\leq \sqrt{C/\lambda^R_{n}}$, where $C$ is as in Lemma \ref{contintpsfphiy}. $(q_n)_{n\geq 1}$ is an orthonormal Hilbert basis of $L^2(m)$. Moreover for any $n \geq 1$, $q_n \in \mathcal{D}(\mathcal{A}_{X^{m,\xi,r}})$ and 
\begin{align}
\mathcal{A}_{X^{m,\xi,r}} q_n = \frac{-1}{\lambda^R_{n} v^r_\xi(0)} q_n. \label{qnfctpropre}
\end{align}
\end{prop}

\begin{proof}
$q_n \in L^2(\mathbb{R})$ and the bound $\|q_n\nr\leq \sqrt{C/\lambda^R_{n}}$ are a consequence of \eqref{defqn} and Lemma \ref{contintpsfphiy}. Using \eqref{defqn}, that $a_n$ is an eigenfunction of $R_{r}$, and that $R_{r}$ is self-adjoint we get 
\begin{align}
\forall y \in \mathbb{R}, \ q_n(y) = \frac1{\sqrt{\lambda^R_{n}}} \psl \phi^{r}_y,a_n \psri = \frac1{(\lambda^R_{n})^{3/2}} \psl \phi^{r}_y,R_{r}.a_n \psri = \frac1{(\lambda^R_{n})^{3/2}} \psl R_{r}.\phi^{r}_y,a_n \psri. \label{qnegcterphih}
\end{align}
By \eqref{strongdualityin} we get $q_n \in \mathcal{D}(\mathcal{A}_{X^{m,\xi,r}}) \subset \crz$. Then the combination of \eqref{qnegcterphih}, \eqref{strongdualityexpr} and \eqref{defqn} yields 
\[ \mathcal{A}_{X^{m,\xi,r}} q_n = \frac1{(\lambda^R_{n})^{3/2}} \mathcal{A}_{X^{m,\xi,r}} \psl R_{r}.\phi^{r}_{\cdot},a_n \psri(y) = \frac{-1}{(\lambda^R_{n})^{3/2} v^r_\xi(0)} \psl \phi^{r}_y,a_n \psri = \frac{-1}{\lambda^R_{n} v^r_\xi(0)} q_n(y), \]
which is \eqref{qnfctpropre}. Then for any $m,n \geq 1$, using \eqref{defqn}, the definition of $P^{r}_y$, the second point of Lemma \ref{oprwelldeefinedanddiag}, and that $a_m$ is an eigenfunction of $R_{r}$ we get 
\begin{align*}
\sqrt{\lambda^R_{n} \lambda^R_{m}} \psl q_n,q_m \psrm & = \int_{\mathbb{R}} m(y) \psl \phi^{r}_y,a_n \psri \psl a_m,\phi^{r}_y \psri dy = \int_{\mathbb{R}} m(y) \psl P^{r}_y.a_m,a_n \psri dy \\
& = \psl R_{r}.a_m,a_n \psri = \lambda^R_{m} \psl a_m,a_n \psri. 
\end{align*}
Therefore $(q_n)_{n\geq 1}$ is an orthonormal family of $L^2(m)$. Now let $f \in L^2(m)$ be such that $\psl q_n,f \psrm=0$ for all $n \geq 1$ and let us prove that $f=0$. For any $g \in L^2((-1,1))$ we set $q_g(y) := \psl \phi^{r}_y,g \psri$. Note from Lemma \ref{contintpsfphiy} that $q_g \in L^2(\mathbb{R}) \subset L^2(m)$. Moreover, if $g \in Span < a_1,a_2,\dots >$ then $q_g$ is a linear combination of finitely many $q_n$'s so $\psl q_g,f \psrm=0$. For an arbitrary $g \in L^2((-1,1))$, let $(g_n)_{n \geq 1}$ a sequence in $Span < a_1,a_2,\dots >$ that converges to $g$ in $L^2((-1,1))$. Then, using Lemma \ref{contintpsfphiy}, we get 
\[ \|q_g - q_{g_n}\|_{L^2(m)}^2 \leq \| m \|_{\infty} \|q_g - q_{g_n}\nr^2 = \| m \|_{\infty} \int_{\mathbb{R}} |\psl \phi^{r}_y,g-g_n \psri |^2 dy \leq C \| m \|_{\infty} \|g-g_n \nri^2, \]
with $C$ as in Lemma \ref{contintpsfphiy}. Therefore $q_{g_n}$ converges to $q_g$ in $L^2(m)$ so $\psl q_g,f \psrm=0$. For any $g \in L^2((-1,1))$ let us define two linear forms from $L^2(m)$ to $\mathbb{R}$: 
\begin{align*}
\forall h \in L^2(m), \ A_g(h) := \psl q_g,h\psrm, \ B_g(h) := \frac1{\sqrt{2}} \psl \mathcal{F}(\overline{h}m)(\varphi_{r}^{-1}(\cdot)), g \psri. 
\end{align*}
$A_g$ is clearly continuous on $L^2(m)$. Then for any $h \in L^2(m)$ we have $\|\overline{h}m\nr^2 \leq \| m \|_{\infty} \|h\|_{L^2(m)}^2$ so $(h \mapsto \overline{h}m)$ is continuous from $L^2(m)$ to $L^2(\mathbb{R})$. $\mathcal{F}$ is continuous from $L^2(\mathbb{R})$ to $L^2(\mathbb{R})$. For any $k \in L^2(\mathbb{R})$, 
\[ \| k(\varphi_{r}^{-1}(\cdot)) \nri^2 \leq \frac1{r} \int_{(-1,1)} |k(\varphi_{r}^{-1}(x))|^2(- \psi_\xi(2\pi \varphi_{r}^{-1}(x)) + r) dx = \frac{2 }{r v^{r}_\xi(0)} \int_{\mathbb{R}} |k(y)|^2 dy = \frac{2 }{r v^{r}_\xi(0)} \|k\nr^2, \]
so $(k \mapsto k(\varphi_{r}^{-1}(\cdot)))$ is continuous from $L^2(\mathbb{R})$ to $L^2((-1,1))$. Therefore $B_g$ is continuous on $L^2(m)$. Let $h \in L^2(\mathbb{R})$. Using $m \in L^2(\mathbb{R})$ (from \eqref{mheavytailed}), two times Cauchy-Schwartz inequality, and $\| \phi^{r}_y \nri=1$ we get 
\[ \int_{\mathbb{R}} \int_{(-1,1)} m(y) \times |h(y)| \times |\phi^{r}_y(x)| \times |g(x)|dx dy \leq \|m\nr \times \|h\nr \times \| g \nri <\infty. \]
We can thus use Fubini's theorem and get 
\begin{align}
\int_{\mathbb{R}} m(y)\overline{h(y)} \left ( \int_{(-1,1)} \phi^{r}_y(x) \overline{g(x)} dx \right ) dy = \int_{(-1,1)} \left ( \int_{\mathbb{R}} m(y)\overline{h(y)} \phi^{r}_y(x) dy \right ) \overline{g(x)} dx. \label{coinconl2}
\end{align}
On the one hand, the left-hand side of \eqref{coinconl2} equals $\int_{\mathbb{R}} m(y)\overline{h(y)} \psl \phi^{r}_y,g \psri dy=\int_{\mathbb{R}} q_g(y)\overline{h(y)}m(y)dy=A_g(h)$. 
On the other hand, the right-hand side of \eqref{coinconl2} equals 
\[ \frac1{\sqrt{2}} \int_{(-1,1)} \left ( \int_{\mathbb{R}} m(y)\overline{h(y)} e^{-2i\pi y \varphi_{r}^{-1}(x)} dy \right ) \overline{g(x)} dx = \frac1{\sqrt{2}} \int_{(-1,1)} \mathcal{F}(\overline{h}m)(\varphi_{r}^{-1}(x)) \overline{g(x)} dx = B_g(h). \]
Therefore, for any $g \in L^2((-1,1))$, the linear forms $A_g$ and $B_g$ coincide on the subspace $L^2(\mathbb{R})$ that is dense in $L^2(m)$. Since $A_g$ and $B_g$ are continuous on $L^2(m)$ we get $A_g=B_g$. We have already shown that $A_g(f)=0$ for all $g \in L^2((-1,1))$ so $B_g(f)=0$ for all $g \in L^2((-1,1))$. We deduce that $\mathcal{F}(\overline{f}m)(\varphi_{r}^{-1}(\cdot))=0$ a.e. on $(-1,1)$ so $\mathcal{F}(\overline{f}m)=0$ a.e. on $\mathbb{R}$ so $\overline{f}m=0$ a.e. on $\mathbb{R}$ and since $m$ is positive on $\mathbb{R}$ (by \eqref{mheavytailed}), $f=0$ a.e. on $\mathbb{R}$. Therefore the orthonormal family $(q_n)_{n\geq 1}$ is total in $L^2(m)$ so it is an Hilbert basis. This concludes the proof. 
\end{proof}

Under the assumptions of Proposition \ref{bonfctpropres}, let $(\lambda^X_{n})_{n \geq 1}$ be the sequence defined by 
\begin{align}
\forall n \geq 1, \ \lambda^X_{n} := \frac{1}{\lambda^R_{n} v^r_\xi(0)}. \label{relvpxandr}
\end{align}
Proposition \ref{bonfctpropres} shows that $(\lambda^X_{n})_{n \geq 1}$ is the sequence of eigenvalues of $-\mathcal{A}_{X^{m,\xi,r}}$ associated with the eigenfunctions $(q_n)_{n\geq 1}$. 

\begin{lemme} \label{sumofnorms}
Assume that $r>0$ and \eqref{mheavytailed} and \eqref{hypcaractexpol-1} hold. For any $\alpha\geq 5/2$ we have 
\begin{align}
\sum_{n \geq 1} (\lambda^X_{n})^{-\alpha} (1 \vee \| q_n \|_{\infty}) < \infty. \label{sumofnorms1}
\end{align}
In particular, for any $\beta\geq 0$ and $t>0$ we have 
\begin{align}
\sum_{n \geq 1} (\lambda^X_{n})^{\beta} e^{- t\lambda^X_{n}} (1 \vee \| q_n \|_{\infty}) < \infty. \label{sumofnorms2}
\end{align}
\end{lemme}

\begin{proof}
\eqref{sumofnorms2} easily follows from \eqref{sumofnorms1} so we only prove the later. From \eqref{defqn} and Cauchy-Schwartz inequality we see that for any $n \geq 1$ and $y \in \mathbb{R}$ we have 
\begin{align*}
|q_n(y)| \leq \frac1{\sqrt{\lambda^R_{n}}} \times \| \phi^{r}_y \nri \times \| a_n \nri = \frac{1}{\sqrt{\lambda^R_{n}}}, 
\end{align*}
where, for the last inequality, we have used that $(a_n)_{n \geq 1}$ is an orthonormal family and that $\| \phi^{r}_y \nri =1$, which is easily seen from the definition of $\phi^{r}_y$ in Section \ref{defopbasprop}. We thus get $\| q_n \|_{\infty} \leq 1/\sqrt{\lambda^R_{n}}$. We see from this and \eqref{relvpxandr} that we only need to show $\sum_{n \geq 1} (\lambda^X_{n})^{-2}< \infty$, but this is a consequence of \eqref{relvpxandr} and \eqref{sumsquarrevpfinite}. This concludes the proof. 
\end{proof}

\subsection{Decomposition of the semigroup in the basis, and the first eigenfunction} \label{decompsgbasis}

The following proposition provides a decomposition of $P_t.f$ in the basis $(q_n)_{n\geq 1}$ for $f \in \crz \cap L^2(\mathbb{R}) \subset L^2(m)$. 
\begin{prop} \label{decompsgbon1}
Assume that $r>0$ and \eqref{mheavytailed} and \eqref{hypcaractexpol-1} hold. For any $f \in \crz \cap L^2(\mathbb{R})$ and $t>0$ we have 
\begin{align}
\forall x \in \mathbb{R}, \ P_t.f(x) = \sum_{n \geq 1} e^{- t\lambda^X_{n}} \psl f, q_n \psrm q_n(x), \label{decompsgbonexpr}
\end{align}
where the series of functions in the right-hand side converges absolutely in $(\crz, \| \cdot \|_{\infty})$. 
\end{prop}

\begin{proof}
Let $f \in L^2(m)$ and $t>0$. According to Proposition \ref{bonfctpropres}, $(q_n)_{n\geq 1}$ is an orthonormal family so, by Cauchy-Schwartz inequality, we have $|\psl f, q_n \psrm| \leq \|f\|_{L^2(m)}$ for any $n\geq 1$. Moreover Proposition \ref{bonfctpropres} shows that $q_n \in \crz$ for any $n\geq 1$. Combining with Lemma \ref{sumofnorms} we get that the series of functions in the right-hand side of \eqref{decompsgbonexpr} converges absolutely in $(\crz, \| \cdot \|_{\infty})$. In particular the right-hand side of \eqref{decompsgbonexpr} is continuous in $x$, let us denote it by $S_t.f(x)$. Note that there is a constant $C_t>0$ such that $\| S_t.f \|_{\infty} \leq C_t \|f\|_{L^2(m)}$. 

By Proposition \ref{bonfctpropres} we have $\|q_n\nr\leq \sqrt{C/\lambda^R_{n}}$ for some $C>0$. Combining this with $|\psl f, q_n \psrm| \leq \|f\|_{L^2(m)}$ and Lemma \ref{sumofnorms} (together with \eqref{relvpxandr}) we get that $S_t.f \in L^2(\mathbb{R}) \subset L^2(m)$. We now compute $\psl S_t.f, q_n \psrm$ for $n \geq 1$. By Proposition \ref{bonfctpropres}, $q_n \in L^2(\mathbb{R})$ so, by Cauchy-Schwartz inequality, $\overline{q_n} m\in L^1(\mathbb{R})$. Combining with the absolute convergence in $(\crz, \| \cdot \|_{\infty})$ of the series of functions defining $S_t.f$ and using that $(q_m)_{m\geq 1}$ is an orthonormal Hilbert basis of $L^2(m)$ by Proposition \ref{bonfctpropres} we get for any $n \geq 1$, 
\begin{align}
\psl S_t.f, q_n \psrm = \int_{\mathbb{R}} S_t.f(x) \overline{q_n(x)}m(x)dx = \sum_{m \geq 1} e^{- t\lambda^X_{m}} \psl f, q_m \psrm \psl q_m, q_n \psrm = e^{- t\lambda^X_{n}} \psl f, q_n \psrm. \label{calccoeffseries}
\end{align}

We now make the restriction $f \in \crz \cap L^2(\mathbb{R})$ (and recall that $\crz \cap L^2(\mathbb{R}) \subset L^2(m)$). According to Proposition \ref{fellerity} the semigroup is Feller so we have $P_t.f \in \crz$. By Lemma \ref{imptinl2m} we have $P_t.f \in L^2(m)$. We now compute $\psl P_t.f, q_n \psrm$ for $n \geq 1$. By Proposition \ref{bonfctpropres}, $q_n \in \crz \cap L^2(\mathbb{R})$. Therefore, using Lemma \ref{ptautoadj} we get for any $n \geq 1$, 
\[ \psl P_t.f, q_n \psrm = \psl f, P_t.q_n \psrm = e^{- t\lambda^X_{n}} \psl f, q_n \psrm, \]
where the last equality comes from $P_t.q_n=e^{- t\lambda^X_{n}}q_n$, which follows from Proposition \ref{bonfctpropres}, \eqref{relvpxandr} and the differentiation rule for the semigroup (see Section \ref{notations}). 

We thus get $\psl P_t.f, q_n \psrm=\psl S_t.f, q_n \psrm$ for all $n \geq 1$. Since $(q_n)_{n\geq 1}$ is an orthonormal Hilbert basis of $L^2(m)$ we get $P_t.f = S_t.f$ in $L^2(m)$ so, since $m$ is positive on $\mathbb{R}$ by \eqref{mheavytailed}, $P_t.f(x) = S_t.f(x)$ for almost every $x \in \mathbb{R}$. As both functions are continuous in $x$, \eqref{decompsgbonexpr} follows. 
\end{proof}

\begin{remark} \label{extensionrmk}
Let $t>0$. Combining \eqref{calccoeffseries} with Parseval's identity we get $\|S_t.f\|_{L^2(m)} \leq \|f\|_{L^2(m)}$ so the operator $S_t$ from the above proof is a contraction of $L^2(m)$. Also, it satisfies $S_t(L^2(m)) \subset \crz$ and it is continuous from $L^2(m)$ to $(\crz, \| \cdot \|_{\infty})$. Moreover, $S_t$ coincides with $P_t$ on the subspace $\crz \cap L^2(\mathbb{R})$ (in the sense that, for $f \in \crz \cap L^2(\mathbb{R})$, we have $P_t.f(x)=S_t.f(x)$ for all $x \in \mathbb{R}$). 
\end{remark}

The first eigenfunction is commonly called \textit{ground state} and enjoys special properties. They are gathered in the following lemma that builds on Lemma \ref{suppxtisr} from Section \ref{5.3support} and they will allow to show the positivity of the constant $K_{m,\xi,r}(x)$ in Theorem \ref{encadrementtransfreeenergythtrans}, and to relate that constant with the infinite-volume Gibbs states of the spin system from Section \ref{interpartsyst}. 

\begin{lemme} \label{q1posonr}
Assume that $r>0$ and \eqref{mheavytailed} and \eqref{hypcaractexpol-1} hold. There is $\theta \in [0,2\pi)$ such that $q_1=e^{i\theta} |q_1|$ and we have $|q_1(x)|>0$ for all $x \in \mathbb{R}$. Moreover, the eigenvalue $\lambda^X_{1}$ of $-\mathcal{A}_{X^{m,\xi,r}}$ has multiplicity $1$ (in other words, $\lambda^X_{2}>\lambda^X_{1}$). 
\end{lemme}

\begin{proof}
Since $R_{r}$ is a compact operator by Lemma \ref{oprwelldeefinedanddiag}, its largest eigenvalue $\lambda^R_{1}$ has a finite multiplicity that we denote by $N$. From \eqref{relvpxandr} we get $\lambda^X_{1}=\cdots=\lambda^X_{N}$ and $\lambda^X_{k}>\lambda^X_{1}$ for $k>N$. 

Let us first show that $Span \{q_1,\dots,q_N\}$ admits an orthonormal basis (for the inner product in $L^2(m)$) made of real eigenfunctions of operators $P_t$. For any $k \in \{1,\dots,N\}$ we have $q_k \in \crz \cap L^2(\mathbb{R}) \subset L^2(m)$ by Proposition \ref{bonfctpropres} so $Re(q_k), Im(q_k) \in \crz \cap L^2(\mathbb{R}) \subset L^2(m)$ and, by \eqref{defsgx}, $P_t.Re(q_k)$ and $P_t.Im(q_k)$ are real functions for any $t>0$. We have $P_t.q_k=e^{- t\lambda^X_{1}}q_k$ from Proposition \ref{bonfctpropres}, \eqref{relvpxandr} and the differentiation rule for the semigroup (see Section \ref{notations}). We thus get that for any $t>0$, $P_t.Re(q_k)+iP_t.Im(q_k)=P_t.q_k=e^{- t\lambda^X_{1}}q_k=e^{- t\lambda^X_{1}}Re(q_k)+ie^{- t\lambda^X_{1}}Im(q_k)$ so $P_t.Re(q_k)=e^{- t\lambda^X_{1}}Re(q_k)$ and $P_t.Im(q_k)=e^{- t\lambda^X_{1}}Im(q_k)$. Therefore $Re(q_k), Im(q_k) \in Span \{q_1,\dots,q_N\}$ and $\{Re(q_1),Im(q_1),\dots,Re(q_N),Im(q_N)\}$ is a spanning family of $Span \{q_1,\dots,q_N\}$ that only contains real eigenfunctions of operators $P_t$. Extracting a linearly independent family and applying Gram–Schmidt process we obtain an orthonormal basis $\{\tilde q_1,\dots,\tilde q_N\}$ of $Span \{q_1,\dots,q_N\}$, where each $\tilde q_k$ lies in $\crz \cap L^2(\mathbb{R}) \subset L^2(m)$ and is a real eigenfunction of all operators $P_t$. Each $\tilde q_k$ is not identically $0$ so, replacing $\tilde q_k$ by $-\tilde q_k$ if necessary, we can assume that 
\begin{align}
\forall k \in \{1,\dots,N\}, \ \exists x_k \in \mathbb{R} \ \text{s.t.} \ \tilde q_k(x_k)>0. \label{posatapoint}
\end{align}

Let us fix $k \in \{1,\dots,N\}$. We have clearly $|\tilde q_k| \in \crz \cap L^2(\mathbb{R}) \subset L^2(m)$ so Proposition \ref{decompsgbon1} yields that for all $t>0$ and $x \in \mathbb{R}$, $P_t.|\tilde q_k|(x) = \sum_{n \geq 1} e^{- t\lambda^X_{n}} \psl |\tilde q_k|, q_n \psrm q_n(x)$ where the series of functions converges absolutely in $(\crz, \| \cdot \|_{\infty})$. Proceeding as in \eqref{calccoeffseries} we get $\psl P_t.|\tilde q_k|, q_n \psrm=e^{- t\lambda^X_{n}} \psl |\tilde q_k|, q_n \psrm$ for any $n \geq 1$. Therefore, by Parseval's identity we get 
\begin{align}
\|P_t.|\tilde q_k|\|_{L^2(m)}^2=\sum_{n \geq 1} e^{- 2t\lambda^X_{n}} |\psl |\tilde q_k|, q_n \psrm|^2. \label{normabsfctprop}
\end{align}
Now note that for any $x \in \mathbb{R}$, $|P_t.\tilde q_k(x)| \leq P_t.|\tilde q_k|(x)$ so $\|P_t.\tilde q_k\|_{L^2(m)} \leq \|P_t.|\tilde q_k|\|_{L^2(m)}$. Using $\|\tilde q_k\|_{L^2(m)}=1$, $P_t.\tilde q_k=e^{- t\lambda^X_{1}}\tilde q_k$, $\|P_t.\tilde q_k\|_{L^2(m)} \leq \|P_t.|\tilde q_k|\|_{L^2(m)}$ and \eqref{normabsfctprop} we get 
\[ e^{- 2t\lambda^X_{1}} = e^{- 2t\lambda^X_{1}} \|\tilde q_k\|_{L^2(m)}^2 = \|P_t.\tilde q_k\|_{L^2(m)}^2 \leq \|P_t.|\tilde q_k|\|_{L^2(m)}^2 =\sum_{n \geq 1} e^{- 2t\lambda^X_{n}} |\psl |\tilde q_k|, q_n \psrm|^2. \]
Since $\||\tilde q_k|\|_{L^2(m)}=\|\tilde q_k\|_{L^2(m)}=1$ we have $\sum_{n \geq 1} |\psl |\tilde q_k|, q_n \psrm|^2=1$. The above thus implies $\psl |\tilde q_k|, q_n \psrm=0$ for $n>N$. Therefore, $|\tilde q_k| \in Span \{q_1,\dots,q_N\}$ so $|\tilde q_k|$ is an eigenfunction of $P_t$ with eigenvalue $e^{- t\lambda^X_{1}}$, and so is $\tilde q_k+|\tilde q_k|$. By \eqref{posatapoint} and by continuity there is an open set $\mathcal{U} \subset \mathbb{R}$ such that $\tilde q_k(x)>0$ for all $x \in \mathcal{U}$. Lemma \ref{suppxtisr} shows that $\mathbb{P}(X^{m,\xi,r}_x(1)\in \mathcal{U})>0$ for all $x \in \mathbb{R}$ so, since $\tilde q_k+|\tilde q_k|$ is positive on $\mathcal{U}$ and non-negative on $\mathbb{R}\setminus \mathcal{U}$, $P_1.(\tilde q_k+|\tilde q_k|)(x)>0$. For any $x \in \mathbb{R}$ we have $\tilde q_k(x)+|\tilde q_k(x)|=e^{\lambda^X_{1}}P_1.(\tilde q_k+|\tilde q_k|)(x)$. We thus get $\tilde q_k(x)+|\tilde q_k(x)|>0$ for all $x \in \mathbb{R}$ so there is no $x \in \mathbb{R}$ such that $\tilde q_k(x)<0$. Therefore, $\tilde q_k$ is positive on $\mathcal{U}$ and non-negative on $\mathbb{R}\setminus \mathcal{U}$ so the same argument again shows that $\tilde q_k(x)>0$ for all $x \in \mathbb{R}$. This and the orthonormality of $(\tilde q_k)_{k \in \{1,\dots,N\}}$ implies $N=1$ so $\lambda^X_{2}>\lambda^X_{1}$ and $q_1=e^{i\theta}\tilde q_1$ for some $\theta \in [0,2\pi)$. This concludes the proof. 
\end{proof}

\subsection{Decomposition of the survival probability of $X^{m,\xi,r}$} \label{decompsurvprobandasympt}

The following result is based on the decomposition from Proposition \ref{decompsgbon1} and provides a representation for $\mathbb{P}(\zeta_x > t)$. 
\begin{prop} \label{decompsgbon2}
Assume that $r>0$ and \eqref{mheavytailed} and \eqref{hypcaractexpol-1} hold. We have for any $t>0$, 
\begin{align}
\forall x \in \mathbb{R}, \ \mathbb{P}(\zeta_x > t) = \sum_{n \geq 1} e^{- t\lambda^X_{n}} \psl m,q_n\psrr q_n(x), \label{decompsg1bonexpr}
\end{align}
where the series of functions in the right-hand side converges absolutely in $(\crz, \| \cdot \|_{\infty})$. 
\end{prop}

\begin{proof}
Let us fix $t >0$. For $M>0$ let 
\begin{align} \label{truncfct}
f_M(x) := \left\{
\begin{aligned}
& 1 \ \text{if} \ x \in [-M,M], \\
& 1-(|x|-M) \ \text{if} \ x \in (M,M+1) \cup (-M-1,-M), \\
& 0 \ \text{if} \ |x| \geq M+1. \end{aligned} \right. 
\end{align}
We have clearly $f_M \in \crz \cap L^2(\mathbb{R})$ so, by Proposition \ref{decompsgbon1}, we get 
\begin{align}
\forall x \in \mathbb{R}, \ P_t.f_M(x) = \sum_{n \geq 1} e^{- t\lambda^X_{n}} \psl f_M, q_n \psrm q_n(x). \label{decompsgbonexprfm}
\end{align}
For any $n \geq 1$, $\psl f_M, q_n \psrm=\int_{\mathbb{R}}f_M(x) \overline{q_n(x)}m(x)dx$. For every $x \in \mathbb{R}$, $f_M(x)$ converges to $1$ as $M$ goes to infinity. Moreover $|f_M(x)|\leq 1$ and $\overline{q_n} m\in L^1(\mathbb{R})$ by Cauchy-Schwartz inequality since $q_n \in L^2(\mathbb{R})$ by Proposition \ref{bonfctpropres} and $m \in L^2(\mathbb{R})$ by \eqref{mheavytailed}. We thus get by dominated convergence that $\psl f_M, q_n \psrm$ converges to $\psl m,q_n\psrr$ as $M$ goes to infinity. Then, using Cauchy-Schwartz inequality and Proposition \ref{bonfctpropres}, $|\psl f_M, q_n \psrm|\leq \int_{\mathbb{R}}|q_n(x)|m(x)dx \leq \|q_n\nr \times \|m\nr \leq \|m\nr \sqrt{C/\lambda^R_{n}}$ for some $C>0$. This and Lemma \ref{sumofnorms} (together with \eqref{relvpxandr}) allow to use dominated convergence to show that the series in the right-hand side of \eqref{decompsgbonexprfm} converges to the right-hand side of \eqref{decompsg1bonexpr} as $M$ goes to infinity. Then, for any $x \in \mathbb{R}$ we see from \eqref{defsgx} and monotone convergence that the left-hand side of \eqref{decompsgbonexprfm} converges to $\mathbb{P}(\zeta_x > t)$ as $M$ goes to infinity. \eqref{decompsg1bonexpr} follows. Since $q_n \in \crz$ (by Proposition \ref{bonfctpropres}) the claim about absolute convergence in $(\crz, \| \cdot \|_{\infty})$ follows from $|\psl q_n,m\psrr|\leq \|m\nr \sqrt{C/\lambda^R_{n}}$ and Lemma \ref{sumofnorms} (together with \eqref{relvpxandr}). 
\end{proof}

The following proposition provides an equivalent for $\mathbb{P}(\zeta_x > t)$ as $t$ goes to infinity. 
\begin{prop} \label{decompsgbon3}
Assume that $r>0$ and \eqref{mheavytailed} and \eqref{hypcaractexpol-1} hold. For every $x \in \mathbb{R}$, we have $\psl m,q_1\psrr q_1(x)>0$ and 
\begin{align}
\mathbb{P}(\zeta_x > t) \underset{t \rightarrow \infty}{\sim} \psl m,q_1\psrr q_1(x) e^{- t\lambda^X_{1}}. 
\end{align}
\end{prop}
The proof builds on the representation of $\mathbb{P}(\zeta_x > t)$ from Proposition \ref{decompsgbon2} and on Lemma \ref{q1posonr}. 
\begin{proof}
From Proposition \ref{decompsgbon2} we get 
\begin{align}
\forall x \in \mathbb{R}, \ e^{t\lambda^X_{1}} \mathbb{P}(\zeta_x > t) - \psl m,q_1\psrr q_1(x) = \sum_{n>1} e^{- t(\lambda^X_{n}-\lambda^X_{1})} \psl m,q_n\psrr q_n(x). \label{decompsg1bonequiv}
\end{align}
By Lemma \ref{q1posonr} we have $\lambda^X_{n}-\lambda^X_{1}>0$ for all $n>1$ so each term in the above series converges to $0$ as $t$ goes to infinity. Since $\lambda^R_{n}$ converges to $0$ as $n$ goes to infinity (see Lemma \ref{oprwelldeefinedanddiag}) we see from \eqref{relvpxandr} that $\lambda^X_{n} \sim \lambda^X_{n}-\lambda^X_{1}$ as $n$ goes to infinity. Therefore, reasoning as in the proof of Lemma \ref{sumofnorms} we get $\sum_{n\geq 1} (\lambda^X_{n})^{1/2} e^{- (\lambda^X_{n}-\lambda^X_{1})} (\| q_n \|_{\infty}\vee 1) < \infty$. Since moreover $|\psl m,q_n\psrr|\leq \|m\nr (C v^r_\xi(0) \lambda^X_{n})^{1/2}$ (by Cauchy-Schwartz inequality, Proposition \ref{bonfctpropres} and \eqref{relvpxandr}), we get by dominated convergence that the right-hand side of \eqref{decompsg1bonequiv} converges to $0$ as $t$ goes to infinity. Therefore $e^{t\lambda^X_{1}} \mathbb{P}(\zeta_x > t)$ converges to $\psl m,q_1\psrr q_1(x)$ as $t$ goes to infinity. We are only left to show that $\psl m,q_1\psrr q_1(x) > 0$. For this recall that, by Lemma \ref{q1posonr} and the differentiation rule for the semigroup (see Section \ref{notations}), $|q_1|$ is an eigenfunction of $P_t$ associated with the eigenvalue $e^{-t\lambda^X_{1}}$ and is positive. We thus get 
\begin{align*}
0 < |q_1(x)|= e^{t\lambda^X_{1}} P_t.|q_1|(x) \leq \| q_1 \|_{\infty} e^{t\lambda^X_{1}} P_t.1(x) = \| q_1 \|_{\infty} e^{t\lambda^X_{1}} \mathbb{P}(\zeta_x > t) \underset{t \rightarrow \infty}{\longrightarrow} \| q_1 \|_{\infty} \psl m,q_1\psrr q_1(x). 
\end{align*}
This proves that $\psl m,q_1\psrr q_1(x) > 0$, concluding the proof. 
\end{proof}

\subsection{Relation with traces of compositions of $R_{r}$} \label{relwithtrofr}

The following proposition connects the traces of compositions of $R_{r}$ with the partition function $\partfct_{n}$. 
\begin{prop} \label{exproftracesrec}
Assume that $r>0$ and \eqref{mheavytailed} and \eqref{hypcaractexpol-1} hold. Let $\partfct_{n}$ be the partition function defined by \eqref{defzntrans} with the choice $\pot:=-\log(v^r_\xi(\cdot))$ and $\pen:=-\log(m(\cdot))$ and let $\mathcal{E}$ be the associated normalized free energy (defined in \eqref{deffreeenergytrans}). For any $n \geq 2$, $\partfct_{n}$ is well-defined and we have $Tr(R_{r}^n)= \partfct_{n}/v^{r}_\xi(0)^n$. Moreover, $\mathcal{E}$ is well-defined and we have $\mathcal{E}=\log(\lambda^X_{1})$. 
\end{prop}

\begin{proof}
Fix $n \geq 2$. The finiteness of the integral defining $\partfct_{n}$ in \eqref{defzntrans} is a consequence of the finiteness of the integrals in \eqref{inegiteres}, in the proof of Lemma \ref{traces2new}. $\partfct_{n}$ is thus well-defined. Recall that the well-definedness of $Tr(R_{r}^n)$ is proved by Lemma \ref{traces2new}. Using \eqref{exproftraces2new} from the proof of Lemma \ref{traces2new} and the definition of $P^{r}_y$, we get 
\begin{align}
Tr(R_{r}^n) = \int_{\mathbb{R}^{n}} m(y_1) \dots m(y_{n}) & \psl \phi^{r}_{y_1},\phi^{r}_{y_2} \psri \times \dots \times \psl \phi^{r}_{y_{n-1}},\phi^{r}_{y_{n}} \psri \label{exproftraces3} \\
\times & \psl \phi^{r}_{y_{n}},\phi^{r}_{y_1} \psri dy_1 \dots dy_{n}. \nonumber
\end{align}
Plugging \eqref{exproftraces6} into \eqref{exproftraces3} and comparing with \eqref{defzntrans} yields that $Tr(R_{r}^n)= \partfct_{n}/v^{r}_\xi(0)^n$ for each $n \geq 2$. We have clearly $\partfct_{n} > 0$ for any $n \geq 2$. Then, combining $Tr(R_{r}^n)= \partfct_{n}/v^{r}_\xi(0)^n$ with \eqref{equivtraces} we get that $\mathcal{E}$ exists and equals $-\log(v^{r}_\xi(0) \lambda^R_{1})$. Combining with \eqref{relvpxandr} we obtain $\mathcal{E}=\log(\lambda^X_{1})$. This concludes the proof. 
\end{proof}

The following proposition will be the main ingredient in the proof of \eqref{inegspecgaptrans}. 
\begin{prop} \label{inegspecgap}
Assume that $r>0$ and \eqref{mheavytailed1} and \eqref{hypcaractexpol-1} hold. We have $1/\|m\|_{L^1(\mathbb{R})} v^r_\xi(0) \leq \lambda^X_{1} \leq \|m\|_{L^1(\mathbb{R})} v^r_\xi(0)/\partfct_2$, with $\partfct_n$ as in Proposition \ref{exproftracesrec}. 
\end{prop}

\begin{proof}
Using Lemma \ref{traces2new1} and that all eigenvalues of $R_{r}$ are positive (from Lemma \ref{oprwelldeefinedanddiag}) we get 
\begin{align*}
\|m\|_{L^1(\mathbb{R})} = Tr(R_{r}) = \sum_{k \geq 1} \lambda_{k}^{R} \geq \lambda^R_{1} = \lambda^R_{1} \times \frac{Tr(R_{r})}{\|m\|_{L^1(\mathbb{R})}} & = \frac{1}{\|m\|_{L^1(\mathbb{R})}} \sum_{k \geq 1} \lambda^R_{1} \lambda_{k}^{R} \\
& \geq \frac{1}{\|m\|_{L^1(\mathbb{R})}} \sum_{k \geq 1} (\lambda_{k}^{R})^2 = \frac{Tr(R_{r}^2)}{\|m\|_{L^1(\mathbb{R})}} = \frac{\partfct_2}{\|m\|_{L^1(\mathbb{R})} \times v^r_\xi(0)^2}, 
\end{align*}
where the last equality comes from Proposition \ref{exproftracesrec}. Combining the bound $\|m\|_{L^1(\mathbb{R})} \geq \lambda^R_{1} \geq \partfct_2/\|m\|_{L^1(\mathbb{R})} v^r_\xi(0)^2$ with \eqref{relvpxandr} we get $1/\|m\|_{L^1(\mathbb{R})} v^r_\xi(0) \leq \lambda^X_{1} \leq \|m\|_{L^1(\mathbb{R})} v^r_\xi(0)/\partfct_2$. 
\end{proof}

\subsection{Infinite-volume Gibbs states} \label{limdistribforlocalobs}

Recall the sequence of probability measures $(L^k_n)_{n \geq 2 \vee k}$ defined in Section \ref{interpartsyst}.
The following proposition relates the first eigenfunction of the semigroup $(P_t)_{t \geq 0}$ with the infinite-volume Gibbs states of the spin system from Section \ref{interpartsyst}. 
\begin{prop} \label{stateonepart}
Assume that $r>0$ and \eqref{mheavytailed} and \eqref{hypcaractexpol-1} hold. Under the choice $\pot:=-\log(v^r_\xi(\cdot))$ and $\pen:=-\log(m(\cdot))$, the infinite-volume Gibbs states $\mathcal{L}_k$ defined in \eqref{deflimstaeonetyppart} is well-defined for any $k \geq 1$ and we have 
\begin{align}
\mathcal{L}_k(dy_1 \dots dy_k) = (\lambda^X_{1})^{k-1} \times m(y_1) \dots m(y_k) \times |q_1(y_1)| v^r_\xi(y_2-y_1) \dots v^r_\xi(y_k-y_{k-1}) |q_1(y_k)|dy_1 \dots dy_k. \label{stateonepartexprk}
\end{align}
In particular we have 
\begin{align}
\mathcal{L}_1(dy) = m(y) | q_1(y) |^2 dy. \label{stateonepartexpr}
\end{align}
Moreover, for any $y \in \mathbb{R}$ we have 
\begin{align}
l_1^n(y) \underset{n \rightarrow \infty}{\longrightarrow} m(y) | q_1(y) |^2. \label{cvdensitytypicalparticle}
\end{align}
\end{prop}
\eqref{stateonepartexpr} says that, if we consider a fixed particle in the infinite volume limit of the spin system from Section \ref{interpartsyst}, its interactions with the infinitely many other particles in the system are coded in the factor $| q_1(y) |^2$. 
\begin{proof}
Let $k \geq 1$ and $f_1,\dots,f_k \in \mathcal{C}_b(\mathbb{R})$ be real non-negative functions. Proposition \ref{exproftracesrec} shows that, for $n \geq 2\vee k$, $\partfct_{n}$ is finite so, in particular, the probability measures $(L^k_n)_{n \geq 2 \vee k}$ are well-defined. We study the convergence of the quantity $L^k_n(f_1,\dots,f_k)$ defined in \eqref{defznf}. For this, when $j \in \{1,\dots,k\}$, let us define the operator $R_{r,j}: L^2((-1,1)) \rightarrow L^2((-1,1))$ by $R_{r,j}.h:= \int_{\mathbb{R}} f_j(y)m(y) P^{r}_y.h dy$, where the integral is in the sense of Bochner. By \eqref{mheavytailed} we have $mf_j \in \crz \cap L^2(\mathbb{R})$ so the operator $R_{r,j}$ also satisfies the properties established for $R_{r}$ in the first two points of Lemma \ref{oprwelldeefinedanddiag} (in particular it is well-defined). The reasoning from the proof of Lemma \ref{traces2new} can be repeated and yields that for any $n \geq 2\vee k$, and any Hilbert basis $(u_j)_{j \geq 1}$ of $L^2((-1,1))$ we have $\sum_{j \geq 1} |\psl R_{r,1}\dots R_{r,k}.R_{r}^{n-k}.u_j,u_j \psri | <\infty$ and the quantity $\sum_{j \geq 1} \psl R_{r,1}\dots R_{r,k}.R_{r}^{n-k}.u_j,u_j \psri$ does not depend on the choice of the Hilbert basis $(u_j)_{j \geq 1}$ so we denote it by $Tr(R_{r,1}\dots R_{r,k}.R_{r}^{n-k})$. Recall that $(a_j)_{j \geq 1}$ is the orthonormal Hilbert basis of $L^2((-1,1))$ consisting of eigenfunctions of $R_{r}$, chosen a little before \eqref{defqn}. Since $|\psl R_{r,1}\dots R_{r,k}.a_j,a_j \psri| \leq \|| R_{r,1}\dots R_{r,k} \||$ for all $j \geq 1$, the reasoning from the proof of Lemma \ref{traces2new}, together with Remark \ref{mult1equiv}, yields that 
\begin{align}
\frac{Tr(R_{r,1}\dots R_{r,k}.R_{r}^{n-k})}{(\lambda_{1}^{R})^{n-k}} \underset{n \rightarrow \infty}{\longrightarrow} \psl R_{r,1}\dots R_{r,k}.a_1,a_1 \psri. \label{equivnewtr}
\end{align}
Proceeding as in \eqref{exproftraces2new} from the proof of Lemma \ref{traces2new} we get for any $n \geq 2\vee k$, 
\begin{align*}
Tr(R_{r,1}\dots R_{r,k}.R_{r}^{n-k}) = \int_{\mathbb{R}^{n}} f_1(y_1) \dots f_k(y_k) m(y_1) \dots m(y_{n}) \psl P^{r}_{y_{n}}\dots P^{r}_{y_2}.\phi^{r}_{y_1},\phi^{r}_{y_1} \psri dy_1 \dots dy_{n}. 
\end{align*}
Then using the definition of $P^{r}_y$ and \eqref{exproftraces6} we obtain 
\begin{align*}
Tr(R_{r,1}\dots R_{r,k}.R_{r}^{n-k}) = v^{r}_\xi(0)^{-n} \int_{\mathbb{R}^{n}} f_1(y_1) \dots f_k(y_k) m(y_1) \dots m(y_{n}) & v^{r}_\xi(y_2-y_1) \times \dots \times v^{r}_\xi(y_n-y_{n-1}) \\
\times & v^{r}_\xi(y_1-y_n) dy_1 \dots dy_{n}. 
\end{align*}
Comparing with \eqref{defznf} we get that for any $n \geq 2\vee k$, 
\begin{align}
Tr(R_{r,1}\dots R_{r,k}.R_{r}^{n-k})= \partfct_n \times L_n^k(f_1,\dots,f_k)/v^{r}_\xi(0)^n. \label{exprnewtr}
\end{align}
Then the combination of \eqref{exprnewtr}, \eqref{equivnewtr}, Proposition \ref{exproftracesrec}, Lemma \ref{traces2new} and Remark \ref{mult1equiv} yields 
\begin{align}
L^k_n(f_1,\dots,f_k) \underset{n \rightarrow \infty}{\longrightarrow} \frac1{(\lambda_{1}^{R})^k} \psl R_{r,1}\dots R_{r,k}.a_1,a_1 \psri. \label{limln}
\end{align}
Using the second point of Lemma \ref{oprwelldeefinedanddiag} (but for the $R_{r,j}$'s in place of $R_{r}$), the identities \eqref{prepfubinitraces1}, \eqref{exproftraces6} and \eqref{defqn}, Lemma \ref{q1posonr}, and \eqref{relvpxandr} we get that the right-hand side of \eqref{limln} equals 
\begin{align}
& \frac1{(\lambda_{1}^{R})^k} \int_{\mathbb{R}^k} (f_1m)(y_1)\dots (f_km)(y_k) \psl P^{r}_{y_1} \dots P^{r}_{y_k}.a_1,a_1 \psri dy_1 \dots dy_k \label{calclim} \\
= & \frac{1}{(\lambda_{1}^{R})^k} \int_{\mathbb{R}^k} (f_1m)(y_1)\dots (f_km)(y_k) (\lambda_{1}^{R})^{1/2} \overline{q_1(y_1)} \frac{v^r_\xi(y_2-y_1)}{v^r_\xi(0)} \dots \frac{v^r_\xi(y_k-y_{k-1})}{v^r_\xi(0)} (\lambda_{1}^{R})^{1/2} q_1(y_k) dy_1 \dots dy_k \nonumber \\
= & (\lambda^X_{1})^{k-1} \int_{\mathbb{R}^k} (f_1m)(y_1)\dots (f_km)(y_k) |q_1(y_1)| v^r_\xi(y_2-y_1) \dots v^r_\xi(y_k-y_{k-1}) |q_1(y_k)| dy_1 \dots dy_k. \nonumber
\end{align}
Combining with \eqref{limln} we get 
\begin{align}
L^k_n(f_1,\dots,f_k) \underset{n \rightarrow \infty}{\longrightarrow} (\lambda^X_{1})^{k-1} \int_{\mathbb{R}^k} (f_1m)(y_1)\dots (f_km)(y_k) |q_1(y_1)| v^r_\xi(y_2-y_1) \dots v^r_\xi(y_k-y_{k-1}) |q_1(y_k)| dy_1 \dots dy_k. \label{limlnexpl}
\end{align}
According to Proposition \ref{bonfctpropres}, $(q_j)_{j\geq 1}$ is an orthonormal family of $L^2(m)$, so $\int_{\mathbb{R}} m(y) | q_1(y) |^2 dy=1$. The measure appearing in the right-hand side of \eqref{limlnexpl} is thus a probability measure when $k=1$. For $k>1$, the calculation in the proof of Proposition \ref{deeperconnectiongs-ips} below will show that the measure in the right-hand side of \eqref{limlnexpl} is the law of a random vector, so in particular it is a probability measure. This and \eqref{limlnexpl} thus imply that for each $k>1$ (resp. $k=1$), we have the convergence in law of $L^k_n$ toward a limit $\mathcal{L}_k$ with expression specified in \eqref{stateonepartexprk} (resp. \eqref{stateonepartexpr}). This concludes the proof of the first part of the proposition. To prove \eqref{cvdensitytypicalparticle}, fix $y \in \mathbb{R}$ and repeat the above argument leading to \eqref{limlnexpl} (with $k=1$), but with $L^1_n(f_1)$ replaced by $l_1^n(y)$ and $R_{r,1}$ replaced by $\tilde R_{r}:= m(y) P^{r}_y$ (formally, this amounts to replacing the function $f_1$ by the distribution $\delta_y$). This concludes the proof. 
\end{proof}

\section{Proof of main results on $X^{m,\xi,r}$ from Section \ref{mainres}} \label{proofmainresults}

\subsection{Asymptotic of survival probability: Proof of Theorem \ref{encadrementtransfreeenergythtrans}} \label{studysurvproba}

Proposition \ref{decompsgbon3} shows that \eqref{expodecaysurvproba} holds with $\gamma_{m,\xi,r}=\lambda^X_{1}$ and $K_{m,\xi,r}(x)=\psl m,q_1\psrr q_1(x)$, and that $K_{m,\xi,r}(x)>0$. We recall that the positivity of $\lambda^X_{1}$ comes from the combination of Lemma \ref{oprwelldeefinedanddiag} with \eqref{relvpxandr}. Proposition \ref{exproftracesrec} (together with $\gamma_{m,\xi,r}=\lambda^X_{1}$) yields the well-definedness of the partition function $\partfct_{n}$ and of the normalized free energy $\mathcal{E}$, and \eqref{freeenergyth0trans}. Using Lemma \ref{q1posonr} we see that we have $K_{m,\xi,r}(x)=\psl m,|q_1|\psrr |q_1(x)|$ so \eqref{microstaedensity} and $\ell_1$ being the point-wise limit of $(l_1^n)_{n\geq 2}$ both follow from Proposition \ref{stateonepart}. In particular, $\ell_1/m=|q_1|^2$ so the claims that $\ell_1/m \in \crz \cap L^1(\mathbb{R})$ and that $\ell_1$ is positive on $\mathbb{R}$ follow respectively from Proposition \ref{bonfctpropres} and Lemma \ref{q1posonr}. Under \eqref{mheavytailed1}, Proposition \ref{inegspecgap} (together with $\gamma_{m,\xi,r}=\lambda^X_{1}$) shows that \eqref{inegspecgaptrans} holds. This completes the proof. 

\begin{remark} \label{khasminskii}
Once \eqref{expodecaysurvproba} has been established, the moment approach discussed in Section \ref{anotherroute} yields an alternative proof of the lower bound in \eqref{inegspecgaptrans}. Indeed, using $\zeta_x=A_x(e_r)$ and Khas'minskii's condition from \cite{FITZSIMMONS1999117} we get $\mathbb{E}[e^{\theta \zeta_x}]<\infty$ for all $\theta<1/\| G_m \|_{\infty}$ where $G_m(y):=\int_{\mathbb{R}} v^r_\xi(y-z)m(z)dz$. By Lemma \ref{linkpsizpotential} and \eqref{mheavytailed1} we get $\| G_m \|_{\infty}\leq \|m\|_{L^1(\mathbb{R})} v^r_\xi(0)$. Therefore, by Chernoff's inequality we get $\mathbb{P}(\zeta_x > t)<<e^{-\theta t}$ for all $\theta<1/\|m\|_{L^1(\mathbb{R})} v^r_\xi(0)$. Combining with \eqref{expodecaysurvproba} we obtain $1/\|m\|_{L^1(\mathbb{R})} v^r_\xi(0) \leq \gamma_{m,\xi,r}$, which is the lower bound in \eqref{inegspecgaptrans}. 
\end{remark}

\subsection{Asymptotic behavior of $\gamma_{m,\xi,r}$: Proof of Corollary \ref{asymptrgoto0}} \label{rgoesto0}

We first justify the following lemma that determines the asymptotic behavior of $v^r_\xi(0)$ when $r$ goes to $0$. 
\begin{lemme} \label{asymptdenspot0}
Assume that \eqref{hypcaractexpol-1} holds, so that the potential density $v^r_\xi(\cdot)$ exists (see Lemma \ref{linkpsizpotential}). We assume that there are $\alpha \geq 1$ and $c>0$ such that $-\psi_\xi(y) \sim c |y|^{\alpha}$ as $y$ goes to $0$. Then we have 
\begin{align}
& v^r_\xi(0) \underset{r \rightarrow 0}{\sim} r^{-\frac{\alpha-1}{\alpha}} \frac{1}{\pi c^{\frac1{\alpha}}} \int_0^{\infty} \frac{1}{1+u^{\alpha}} du \ \ \ \text{if} \ \alpha > 1, \label{asympvrzregvar} \\
& v^r_\xi(0) \underset{r \rightarrow 0}{\sim} \frac{1}{c \pi} \log \left ( \frac1{r} \right ) \ \ \ \text{if} \ \alpha = 1. \label{asympvrzregslow}
\end{align}
\end{lemme}

\begin{remark} 
If $\xi$ is a symmetric $\alpha$-stable L\'evy process on $\mathbb{R}$ with $\alpha>1$ then $-\psi_\xi(y) = c |y|^{\alpha}$ for some $c>0$ (see for example Theorem 14.15 in \cite{Sato}) and $v^r_\xi(0) = r^{-(\alpha-1)/\alpha} v^1_\xi(0)$ (see for example the discussion after Theorem II.19 in \cite{Bertoin}) so \eqref{asympvrzregvar} is even an equality and Lemma \ref{asymptdenspot0} allows to identify the constant $v^1_\xi(0)$ with $(\pi c^{1/\alpha})^{-1} \int_0^{\infty} 1/(1+u^{\alpha}) du$. 
\end{remark}

\begin{proof} [Proof of Lemma \ref{asymptdenspot0}]
For $\delta>0$ let 
\begin{align}
q_{\delta}(r) := \frac1{2\pi} \int_{-\delta}^{\delta} \frac{1}{r-\psi_\xi(z)} dz, \ \ \ \text{and} \ \ \ C_{\delta} := \frac1{2\pi} \int_{[-\delta,\delta]^c} \frac{1}{-\psi_\xi(z)} dz. \label{defqdelta}
\end{align}
Note that \eqref{hypcaractexpol-1} ensures that $C_{\delta}<\infty$. From Lemma \ref{linkpsizpotential} we see that, for any $\delta > 0$ and $r>0$, $q_{\delta}(r) \leq v^r_\xi(0) \leq q_{\delta}(r) + C_{\delta}$. From the assumption $-\psi_\xi(y) \sim c |y|^{\alpha}$, $q_{\delta}(r)$ and $v^r_\xi(0)$ converge to infinity as $r$ goes to zero. We thus have that for any $\delta>0$, 
\begin{align} \label{equivvzr}
v^r_\xi(0) \underset{r \rightarrow 0}{\sim} q_{\delta}(r). 
\end{align}

Since $-\psi_\xi(y) \sim c |y|^{\alpha}$, we can write $-\psi_\xi(y) = c |y|^{\alpha}B(y)$ where $B(y)$ converges to $1$ as $y$ goes to $0$. Let us fix $\epsilon \in(0,1)$. Then, there is $\delta>0$ such that $B(y) \in (1-\epsilon,1+\epsilon)$ for all $y \in (0,\delta]$. From now onward we fix such a $\delta>0$. Using \eqref{equivvzr}, the symmetry of $\psi_\xi(\cdot)$, and $-\psi_\xi(y) = c |y|^{\alpha}B(y)$, we get 
\begin{align} \label{equivvzr2}
v^r_\xi(0) \underset{r \rightarrow 0}{\sim} \frac1{\pi} \int_{0}^{\delta} \frac{1}{r+c y^{\alpha}B(y)} dy. 
\end{align}
Then, 
\[ r^{\frac{\alpha-1}{\alpha}} \int_{0}^{\delta} \frac{1}{r+cy^{\alpha}B(y)} dy = r^{-\frac1{\alpha}} \int_0^{\delta} \frac{1}{1+\frac{c}{r}y^{\alpha}B(y)} dy = c^{-\frac1{\alpha}} \int_0^{\delta (c/r)^{\frac1{\alpha}}} \frac{1}{1+u^{\alpha}B((r/c)^{\frac1{\alpha}}u)} du. \]
We thus get that 
\[ c^{-\frac1{\alpha}} \int_0^{\delta (c/r)^{\frac1{\alpha}}} \frac{1}{1+(1+\epsilon)u^{\alpha}} du \leq r^{\frac{\alpha-1}{\alpha}} \int_{0}^{\delta} \frac{1}{r+cy^{\alpha}B(y)} dy \leq c^{-\frac1{\alpha}} \int_0^{\delta (c/r)^{\frac1{\alpha}}} \frac{1}{1+(1-\epsilon)u^{\alpha}} du. \]
Let us now assume that $\alpha>1$. Combining the above with \eqref{equivvzr2} we get 
\[ \frac{c^{-\frac1{\alpha}}}{\pi} \int_0^{\infty} \frac{1}{1+(1+\epsilon)u^{\alpha}} du \leq \underset{r \rightarrow 0}{\liminf} \ r^{\frac{\alpha-1}{\alpha}} v^r_\xi(0) \leq \underset{r \rightarrow 0}{\limsup} \ r^{\frac{\alpha-1}{\alpha}} v^r_\xi(0) \leq \frac{c^{-\frac1{\alpha}}}{\pi} \int_0^{\infty} \frac{1}{1+(1-\epsilon)u^{\alpha}} du. \]
Letting $\epsilon$  go to $0$ and using monotone convergence we get \eqref{asympvrzregvar}. If $\alpha=1$ then a similar reasoning yields \eqref{asympvrzregslow}. 
\end{proof}

We now assume that the assumptions from Corollary \ref{asymptrgoto0} hold and prove the corollary. We first show that, under the choice $\pot:=-\log(v^r_\xi(\cdot))$ and $\pen:=-\log(m(\cdot))$, we have  
\begin{align}
\partfct_2 \underset{r \rightarrow 0}{\sim} \|m\|_{L^1(\mathbb{R})}^2 \times v^r_\xi(0)^2. \label{asymptrgoto0equiv1}
\end{align}
Recall from Remark \ref{altexpr} that $\hat m$ denotes the characteristic function of $m(x) dx$. Let us fix $\epsilon \in(0,1)$. Since $\hat m(0)=\|m\|_{L^1(\mathbb{R})}$ and $\hat m$ is continuous at $0$, there is $\delta>0$ such that $|\hat m(x)|^2/\|m\|_{L^1(\mathbb{R})}^2 \in (1-\epsilon,1]$ for all $x \in [-2\delta,2\delta]$. According to \eqref{altexpr2} we have 
\begin{align} \label{lowerboudzr2}
\frac{\partfct_2}{\|m\|_{L^1(\mathbb{R})}^2} \geq \frac1{(2\pi)^2} \int_{-\delta}^{\delta} \int_{-\delta}^{\delta} \frac{1-\epsilon}{(r-\psi_\xi(z_1)) \times (r-\psi_\xi(z_2))} dz_1 dz_2 = (1-\epsilon) q_{\delta}(r)^2, 
\end{align}
where $q_{\delta}(r)$ is defined as in \eqref{defqdelta}. Since $|\hat m(x)|^2 \leq \|m\|_{L^1(\mathbb{R})}^2$ for all $x$, 
\begin{align} 
\frac{\partfct_2}{\|m\|_{L^1(\mathbb{R})}^2} & \leq \frac1{(2\pi)^2} \int_{\mathbb{R}^2} \frac{1}{(r-\psi_\xi(z_1)) \times (r-\psi_\xi(z_2))} dz_1 dz_2 = v^r_\xi(0)^2. \label{upperboudzr2}
\end{align}
Since $q_{\delta}(r)$ and $v^r_\xi(0)$ converges to infinity as $r$ goes to infinity we get from \eqref{lowerboudzr2}, \eqref{upperboudzr2} and \eqref{equivvzr} that 
\[ 1-\epsilon \leq \underset{r \rightarrow 0}{\liminf} \frac{\partfct_2}{\|m\|_{L^1(\mathbb{R})}^2 \times v^r_\xi(0)^2} \leq \underset{r \rightarrow 0}{\limsup} \frac{\partfct_2}{\|m\|_{L^1(\mathbb{R})}^2 \times v^r_\xi(0)^2} \leq 1. \]
Since $\epsilon$ can be chosen arbitrary small we get \eqref{asymptrgoto0equiv1}. 

The combination of \eqref{asymptrgoto0equiv1} with \eqref{inegspecgaptrans} yields 
\begin{align}
\gamma_{m,\xi,r} \underset{r \rightarrow 0}{\sim} \frac1{\|m\|_{L^1(\mathbb{R})}} \times \frac1{v^r_\xi(0)}. \label{asymptrgoto0equiv}
\end{align}
The combination of \eqref{asymptrgoto0equiv} with Lemma \ref{asymptdenspot0} yields \eqref{asymptrgoto0equivregvar} and \eqref{asymptrgoto0equivregvaralpha=1}, completing the proof. 

\begin{remark} \label{corocaserec}
The assumption "there are $\alpha \geq 1$ and $c>0$ such that $-\psi_\xi(y) \sim c |y|^{\alpha}$ as $y$ goes to $0$" implies that the un-killed L\'evy process $\xi$ is recurrent (see for example Theorem 37.5 in \cite{Sato}). If we only assume that \eqref{mheavytailed1} and \eqref{hypcaractexpol-1} hold and that $\xi$ is recurrent (which occurs, in particular, if $\mathbb{E}[|\xi(1)|]<\infty$, see Theorem 36.7 in \cite{Sato}), then \eqref{asymptrgoto0equiv} is still satisfied, which can be seen as a weak version of Corollary \ref{asymptrgoto0}. Indeed, Theorem 37.5 in \cite{Sato} shows that $q_{\delta}(r)$ converges to infinity as $r$ goes to zero so \eqref{equivvzr} still holds true. Then the above proof of \eqref{asymptrgoto0equiv1} still holds and we get that \eqref{asymptrgoto0equiv} holds true. 
\end{remark}

\subsection{Further spectral properties of $X^{m,\xi,r}$: Proof of Theorem \ref{diagogenetsg}} \label{diagogenetsgproof}

For $t>0$, let $S_t$ be the operator from Remark \ref{extensionrmk}. According to that remark, $S_t$ is a contraction of $L^2(m)$ that satisfies $S_t(L^2(m)) \subset \crz$ and that is continuous from $L^2(m)$ to $(\crz, \| \cdot \|_{\infty})$. For any $x \in \mathbb{R}$ we define the linear forms $p_{t,x}:\mathcal{B}_b(\mathbb{R}) \rightarrow \mathbb{C}$ and $s_{t,x}:L^2(m) \rightarrow \mathbb{C}$ by $p_{t,x}.f=P_t.f(x)$ and $s_{t,x}.f=S_t.f(x)$. By \eqref{defsgx}, $p_{t,x}$ corresponds to the restriction to $\mathbb{R}$ of the finite measure $\mathbb{P}(X^{m,\xi,r}_x(t) \in dz)$. Since $s_{t,x}$ is continuous, by Riesz representation theorem, there is $h_{t,x} \in L^2(m)$ such that, for $f\in L^2(m)$, $s_{t,x}.f=\psl f, h_{t,x} \psrm$. According to Remark \ref{extensionrmk}, $s_{t,x}$ coincides with $p_{t,x}$ on $\crz \cap L^2(\mathbb{R})$ so it is a positive linear form on that subspace. Since any non-negative function of $L^2(m)$ can be approximated by a sequence of non-negative functions from $\crz \cap L^2(\mathbb{R})$, $s_{t,x}$ is a positive linear form on $L^2(m)$ so the function $mh_{t,x}$ is real and almost everywhere non-negative. Moreover, by \eqref{mheavytailed}, for any compact set $K \subset \mathbb{R}$ we have $\textbf{1}_K \in L^2(m)$ so $\int_K h_{t,x}(z)m(z)dz=\psl \textbf{1}_K, h_{t,x} \psrm <\infty$. Therefore, $s_{t,x}$ corresponds to a sigma-finite measure on $\mathbb{R}$. Since $p_{t,x}$ and $s_{t,x}$ coincide on the space $\crz \cap L^2(\mathbb{R})$, the corresponding measures on $\mathbb{R}$ are equal. Therefore, for any $f \in \mathcal{B}_b(\mathbb{R}) \cap L^2(m)$ and $x \in \mathbb{R}$ we have $P_t.f(x)=S_t.f(x)$. This shows that $S_t$ is indeed an extension of $P_t$ on $L^2(m)$. Since $\mathcal{B}_b(\mathbb{R}) \cap L^2(m)$ is dense in $L^2(m)$, such a contractile extension is unique. Using \eqref{calccoeffseries} and Parseval's identity we get 
\begin{align*}
\|S_t.f-f\|_{L^2(m)}^2=\sum_{n \geq 1} (1-e^{- t\lambda^X_{n}})^2 |\psl f, q_n \psrm|^2. 
\end{align*}
Since $\sum_{n \geq 1} |\psl f, q_n \psrm|^2 = \|f\|_{L^2(m)}^2<\infty$ we get by dominated convergence that the right-hand side converges to $0$ as $t$ goes to $0$. This shows that $S_t$ is strongly continuous in $L^2(m)$, concluding the proof of the first claim. 

Let $h_n:=q_n$ for $n\geq 2$ and $h_1:=|q_1|$ where $(q_n)_{n\geq 1}$ is the family of functions provided by Proposition \ref{bonfctpropres}. Note from Lemma \ref{q1posonr} that $h_1=e^{-i\theta} q_1$ for some $\theta \in [0,2\pi)$. By Proposition \ref{bonfctpropres}, $(h_n)_{n\geq 1}$ is an orthonormal Hilbert basis of $L^2(m)$ such that for any $n \geq 1$, $h_n \in \crz \cap L^2(\mathbb{R}) \subset L^2(m)$. Proposition \ref{bonfctpropres} and \eqref{relvpxandr}, together with the differentiation rule for the semigroup (see Section \ref{notations}) yield that $P_t.h_n = e^{-t\lambda^X_{n}} h_n$ for any $n \geq 1$. The combination of Lemma \ref{oprwelldeefinedanddiag} and \eqref{relvpxandr} yields $\lambda^X_{1} >0$ and that $(\lambda^X_{n})_{n \geq 1}$ is non-decreasing. \eqref{decompsgbonexpr0} follows from the definition of $S_t.f(x)$ in the proof of Proposition \ref{decompsgbon1} and from the fact, shown above, that the unique contractile extension of $P_t$ to $L^2(m)$ equals $S_t$. This proves the second claim. Lemma \ref{q1posonr} yields $\lambda^X_{2}>\lambda^X_{1}$. Finally, the two equalities from \eqref{ident1stvpandfp} have already been proved in Section \ref{studysurvproba}. This proves the third claim. Under \eqref{mheavytailed1} we have $m \in L^1(\mathbb{R})$ so $\mathcal{B}_b(\mathbb{R}) \subset L^2(m)$ and the strong Feller property follows from the first claim. 

\subsection{Proof of Example \ref{exepl1}} \label{proofexepl1}

Let us set $f(x):=\sqrt{2/3}(1+|x|)e^{-|x|}$ then we have clearly $f \in \crz \cap L^2(\mathbb{R}) \subset L^2(m)$ and $\|f\|_{L^2(m)}=1$. One can also see that $f\in \mathcal{C}_0^2(\mathbb{R})$ (one can check that $f$ and $f'$ have right and left derivatives at $0$ and that those right and left derivatives at $0$ coincide). By Theorem 31.5 of \cite{Sato} we get $f \in \mathcal{D}(\mathcal{A}_{\xi})$ and $\mathcal{A}_{\xi}f=f''/2$. Since $\mathcal{D}(\mathcal{A}_{\xi})=\mathcal{D}(\mathcal{A}_{\xi^r})$ and $\mathcal{A}_{\xi^r}f=\mathcal{A}_{\xi}f-rf$ we get $f\in \mathcal{D}(\mathcal{A}_{\xi^r})$ and, after a straightforward calculation, $\frac1{m}\mathcal{A}_{\xi^r}f=-f$. In particular we have $\frac1{m}\mathcal{A}_{\xi^r}f \in \crz$ so by Lemma \ref{generatorxnew} we get $f \in \mathcal{D}(\mathcal{A}_{X^{m,\xi,r}})$ and $\mathcal{A}_{X^{m,\xi,r}}f =-f$. By the differentiation rule for the semigroup (see Section \ref{notations}) we get $P_t.f = e^{-t} f$ for all $t \geq 0$. We have $P_t.h_1 = e^{-t\lambda^X_{1}} h_1$ from Theorem \ref{diagogenetsg}. If we had $\lambda^X_{1}\neq 1$ then $f$ and $h_1$ would be eigenfunctions of $P_t$ associated with different eigenvalues so, applying Lemma \ref{ptautoadj} with $f$ and $h_1$, we would get $\psl f, h_1 \psrm=0$. Since $f$ and $h_1$ are both positive on $\mathbb{R}$ (for $h_1$ this comes from \eqref{ident1stvpandfp} and Theorem \ref{encadrementtransfreeenergythtrans}) we have $\psl f, h_1 \psrm>0$. Therefore $\lambda^X_{1}=1$. By Theorem \ref{diagogenetsg}, the eigenspace of $\lambda^X_{1}$ has dimension $1$. Since both $f$ and $h_1$ lie in this eigenspace, $\|f\|_{L^2(m)}=1=\|h_1\|_{L^2(m)}$ and $f$ and $h_1$ are both positive, we get $f=h_1$. 

\section{Ground state transform of $X^{m,\xi,r}$ and relation with the spin system} \label{gstxrelips}

\subsection{Existence and uniqueness of $\tilde X^{m,\xi,r}$: Proof of Proposition \ref{fellerityofgroundstatetrans}} \label{exgst}
Let us first justify that $(\tilde p_t(x,dy))_{t \geq 0, x \in \mathbb{R}}$ defines a transition function on $\mathbb{R}$. From Theorem \ref{diagogenetsg} we have $h_1 \in \crz$ and $P_t.h_1=e^{- t\lambda^X_{1}}h_1$. We thus deduce from \eqref{defsgxgstrans} that, for any $t\geq 0$ and $x \in \mathbb{R}$, $\tilde p_t(x,dy)$ is a probability measure on $\mathbb{R}$. Since $h_1$ is continuous on $\mathbb{R}$ and $X^{m,\xi,r}$ is a Feller process, we easily get that for any Borel set $B\subset \mathbb{R}$, the function $(x \mapsto \tilde p_t(x,B))$ is measurable for all $t \geq 0$. Finally, the semigroup property for $(P_t)_{t \geq 0}$ (proved in Lemma \ref{markovianity}) easily leads to the Chapman-Kolmogorov relation $\int_{\mathbb{R}} \tilde p_t(x,dz) \tilde p_s(z,dy) = \tilde p_{t+s}(x,dy)$, for any $t,s\geq 0$ and $x\in \mathbb{R}$. Therefore, $(\tilde p_t(x,dy))_{t \geq 0, x \in \mathbb{R}}$ defines a transition function on $\mathbb{R}$. By Theorem 1.5 in Chapter 3 of \cite{RevYor} we deduce the existence and uniqueness of $\tilde X^{m,\xi,r}$. 

\subsection{Representation of $\mathcal{L}_{k}$ in terms of $\tilde X^{m,\xi,r}$: Proof of Proposition \ref{deeperconnectiongs-ips}} \label{represlktildex}
Let $f_0,\dots,f_k \in \mathcal{C}_b(\mathbb{R})$ be real non-negative functions. Using \eqref{defsgxgstrans}, the combination of \eqref{exprresolvante0} and Lemma \ref{linkpsizpotential}, and Fubini's theorem 
\begin{align}
& \mathbb{E} \left [ f_0(\tilde X^{m,\xi,r}_{\mathcal{L}_1}(T_0))\dots f_k(\tilde X^{m,\xi,r}_{\mathcal{L}_1}(T_k)) \right ] \label{calctXarrpoiss} \\ 
= & \mathbb{E} \left [ f_0(\tilde X^{m,\xi,r}_{\mathcal{L}_1}(T_0))\dots f_{k-1}(\tilde X^{m,\xi,r}_{\mathcal{L}_1}(T_{k-1})) \frac{\lambda^X_{1}}{h_1(\tilde X^{m,\xi,r}_{\mathcal{L}_1}(T_{k-1}))} \int_0^{\infty} P_t.(h_1f_k)(\tilde X^{m,\xi,r}_{\mathcal{L}_1}(T_{k-1})) dt \right ] \nonumber \\
= & \mathbb{E} \left [ f_0(\tilde X^{m,\xi,r}_{\mathcal{L}_1}(T_0))\dots f_{k-1}(\tilde X^{m,\xi,r}_{\mathcal{L}_1}(T_{k-1})) \frac{\lambda^X_{1}}{h_1(\tilde X^{m,\xi,r}_{\mathcal{L}_1}(T_{k-1}))} U_0(h_1 f_k)(\tilde X^{m,\xi,r}_{\mathcal{L}_1}(T_{k-1})) \right ] \nonumber \\
= & \mathbb{E} \left [ f_0(\tilde X^{m,\xi,r}_{\mathcal{L}_1}(T_0))\dots f_{k-1}(\tilde X^{m,\xi,r}_{\mathcal{L}_1}(T_{k-1})) \frac{\lambda^X_{1}}{h_1(\tilde X^{m,\xi,r}_{\mathcal{L}_1}(T_{k-1}))} \int_{\mathbb{R}} (f_kh_1m)(\tilde X^{m,\xi,r}_{\mathcal{L}_1}(T_{k-1})+y) v^r_\xi(y) dy \right ] \nonumber \\
= & \lambda^X_{1} \int_{\mathbb{R}} \mathbb{E} \left [ f_0(\tilde X^{m,\xi,r}_{\mathcal{L}_1}(T_0))\dots f_{k-1}(\tilde X^{m,\xi,r}_{\mathcal{L}_1}(T_{k-1})) \frac{v^r_\xi(y_k-\tilde X^{m,\xi,r}_{\mathcal{L}_1}(T_{k-1}))}{h_1(\tilde X^{m,\xi,r}_{\mathcal{L}_1}(T_{k-1}))} \right ] f_k(y_k)h_1(y_k)m(y_k) dy_k. \nonumber 
\end{align}
Iterating this calculation we get that the left-hand side of \eqref{calctXarrpoiss} equals 
\begin{align*}
& (\lambda^X_{1})^k \int_{\mathbb{R}^k} \mathbb{E} \left [ f_0(\tilde X^{m,\xi,r}_{\mathcal{L}_1}(T_0)) \frac{v^r_\xi(y_1-\tilde X^{m,\xi,r}_{\mathcal{L}_1}(T_0))}{h_1(\tilde X^{m,\xi,r}_{\mathcal{L}_1}(T_0))} \right ] \nonumber \\
& \qquad \qquad \qquad \qquad \qquad \qquad \times v^r_\xi(y_2-y_1) \dots v^r_\xi(y_k-y_{k-1}) h_1(y_k) (f_1m)(y_1)\dots (f_km)(y_k) dy_1 \dots dy_k \\
= & (\lambda^X_{1})^k \int_{\mathbb{R}^k} h_1(y_0) v^r_\xi(y_1-y_0) \dots v^r_\xi(y_k-y_{k-1}) h_1(y_k) (f_0m)(y_0)\dots (f_km)(y_k) dy_0 \dots dy_k, \\
= & \int_{\mathbb{R}^{k+1}} f_0(y_0)\dots f_k(y_k) \mathcal{L}_{k+1}(dy_0 \dots dy_k), 
\end{align*}
where we have used $h_1=|q_1|$ (by definition of $h_1$), $\tilde X^{m,\xi,r}_{\mathcal{L}_1}(T_0) \sim \mathcal{L}_1$, and \eqref{stateonepartexpr} for the penultimate equality, and \eqref{stateonepartexprk} for the last equality. 

\subsection{Spectral gap of $\tilde X^{m,\xi,r}$ and correlation decay in the spin system: Proof of Theorem \ref{ergoofgroundstatetrans}} \label{specgapgstandcorr}
Let $f \in \mathcal{C}_b(\mathbb{R})$. We have $q_1f \in \crz \cap L^2(\mathbb{R})$ thanks to Proposition \ref{bonfctpropres}. Recall also from the proof of Theorem \ref{diagogenetsg} that $h_1$ is a multiple of $q_1$. Therefore, combining \eqref{defsgxgstrans} with Proposition \ref{decompsgbon1} we get that for any $t>0$ and $x \in \mathbb{R}$, 
\begin{align}
F_t(x):=\mathbb{E} \left [ f \left ( \tilde X^{m,\xi,r}_x(t) \right ) \right ] = \frac{e^{t\lambda^X_{1}}}{q_1(x)} P_t.(q_1f)(x) = \psl q_1 f, q_1 \psrm + \frac1{q_1(x)} \sum_{n \geq 2} e^{- t(\lambda^X_{n}-\lambda^X_{1})} \psl q_1 f, q_n \psrm q_n(x), \label{decompsggroundstatetrans}
\end{align}
where the series of functions in the right-hand side converges absolutely in $(\crz, \| \cdot \|_{\infty})$. For any $x \in \mathbb{R}$, we get by dominated convergence that the series of functions in the right-hand side of \eqref{decompsggroundstatetrans} converges to $0$ as $t$ goes to infinity. We thus get that for any $x \in \mathbb{R}$, 
\[ \mathbb{E} \left [ f \left ( \tilde X^{m,\xi,r}_x(t) \right ) \right ] \underset{t \rightarrow \infty}{\longrightarrow} \psl q_1 f, q_1 \psrm = \int_{\mathbb{R}} f(y) \mathcal{L}_1(dy), \]
where the last equality comes from \eqref{stateonepartexpr}. This proves that $\tilde X^{m,\xi,r}$ is ergodic with stationary distribution $\mathcal{L}_1(dy)$. 

It is easily seen from \eqref{stateonepartexpr} and Proposition \ref{bonfctpropres} that $(q_n/q_1)_{n\geq 1}$ is an orthonormal family of $L^2(\mathcal{L}_1)$. Now let $g \in L^2(\mathcal{L}_1)$ such that $g$ is orthogonal to $(q_n/q_1)_{n\geq 1}$ in $L^2(\mathcal{L}_1)$. Then, by \eqref{stateonepartexpr}, $gq_1 \in L^2(m)$ and $gq_1$ is orthogonal to $(q_n)_{n\geq 1}$ in $L^2(m)$. Proposition \ref{bonfctpropres} yields $gq_1=0$ and, by Lemma \ref{q1posonr}, we deduce that $g=0$. Therefore, $(q_n/q_1)_{n\geq 1}$ is an orthonormal Hilbert basis of $L^2(\mathcal{L}_1)$. For our fixed $f \in \mathcal{C}_b(\mathbb{R})$ and $n\geq 2$ we have $\|q_n/q_1 \|_{L^2(\mathcal{L}_1)}=1$ and $|\psl q_1 f, q_n \psrm| \leq \|f\|_{\infty} \psl |q_1|, |q_n| \psrm \leq \|f\|_{\infty}$. Combining with Lemma \ref{sumofnorms} we get that the right-hand side of \eqref{decompsggroundstatetrans} is an absolutely convergent series in $L^2(\mathcal{L}_1)$. In particular $F_t \in L^2(\mathcal{L}_1)$, and proceeding as in \eqref{calccoeffseries} we get for any $n \geq 1$, 
\[ \psl F_t, q_n/q_1 \psrll = \psl q_1 F_t, q_n \psrm=e^{- t(\lambda^X_{n}-\lambda^X_{1})} \psl q_1 f, q_n \psrm =e^{- t(\lambda^X_{n}-\lambda^X_{1})} \psl f, q_n/q_1 \psrll. \]
By Parseval's identity we thus get 
\begin{align}
\| F_t-\psl f, 1 \psrll \|_{L^2(\mathcal{L}_1)}^2 = \sum_{n \geq 2} e^{- 2t(\lambda^X_{n}-\lambda^X_{1})} |\psl f, q_n/q_1 \psrll |^2 \leq e^{- 2t(\lambda^X_{2}-\lambda^X_{1})} \| f-\psl f, 1 \psrll \|_{L^2(\mathcal{L}_1)}^2. \label{specgapineg1}
\end{align}
We get \eqref{specgapineg} with $c:=\lambda^X_{2}-\lambda^X_{1}$. To show the optimality of this choice of $c$, take $f_M$ as in \eqref{truncfct} and note that $f_Mq_2/q_1 \in \mathcal{C}_b(\mathbb{R})$. For $M$ large enough we have $|\psl f_Mq_2/q_1, q_2/q_1 \psrll |^2>0$ (since the latter converges to $1$ as $M$ goes to infinity). Applying \eqref{specgapineg1} with the choice $f:=f_Mq_2/q_1$ we thus get $\| F_t-\psl f, 1 \psrll \|_{L^2(\mathcal{L}_1)}^2 \geq e^{- 2t(\lambda^X_{2}-\lambda^X_{1})} |\psl f, q_2/q_1 \psrll |^2$ with $|\psl f, q_2/q_1 \psrll |^2>0$, so \eqref{specgapineg} can hold for all $f \in \mathcal{C}_b(\mathbb{R})$ only if $c\leq \lambda^X_{2}-\lambda^X_{1}$. Finally, it will be proved in Section \ref{seccorrel} below that the coefficient $\mathcal{C}$ from Corollary \ref{cordecay} exists and equals $\log(\lambda_{2}^X/\lambda_{1}^X)$. The combination of this with Proposition \ref{exproftracesrec} will yield $c=e^{\mathcal{E}}(e^{\mathcal{C}}-1)$, completing the proof. 

\subsection{Correlation decay in the spin system: Proof of Corollary \ref{cordecay}} \label{seccorrel}
Let us fix $k \geq 3$. Using Proposition \ref{deeperconnectiongs-ips}, the decomposition \eqref{decompsggroundstatetrans} and \eqref{stateonepartexpr} we get 
\begin{align}
C_k(f,g) & = \mathbb{E} \left [ f(\tilde X^{m,\xi,r}_{\mathcal{L}_1}(0)) \left ( g(\tilde X^{m,\xi,r}_{\mathcal{L}_1}(T_k)) - \int_{\mathbb{R}} g(y) \mathcal{L}_1(dy) \right ) \right ] \nonumber \\
& = \int_{\mathbb{R}} f(x) \int_0^{\infty} \left ( \mathbb{E} \left [g(\tilde X^{m,\xi,r}_x(t))\right ] - \int_{\mathbb{R}} g(y) \mathcal{L}_1(dy) \right ) \mathbb{P}(T_k \in dt) \mathcal{L}_1(dx) \nonumber \\
& = \int_{\mathbb{R}} m(x) f(x) \overline{q_1(x)} \int_0^{\infty} \left ( \sum_{n \geq 2} e^{- t(\lambda^X_{n}-\lambda^X_{1})} \psl q_1 g, q_n \psrm q_n(x) \right ) \mathbb{P}(T_k \in dt) dx. \label{corelldecompsg1}
\end{align}
Then, using that $|\psl q_1 g, q_n \psrm| \leq \|g\|_{\infty} \psl |q_1|, |q_n| \psrm \leq \|g\|_{\infty}$ we get 
\begin{align}
\int_0^{\infty} \left ( \sum_{n \geq 2} e^{- t(\lambda^X_{n}-\lambda^X_{1})} \left | \psl q_1 g, q_n \psrm q_n(x) \right | \right ) \mathbb{P}(T_k \in dt) & = \sum_{n \geq 2} \left ( \frac{\lambda^X_{1}}{\lambda^X_{n}} \right )^k \left | \psl q_1 g, q_n \psrm q_n(x) \right | \label{corelldecompsgfub} \\
& \leq \|g\|_{\infty} \sum_{n \geq 2} \left ( \frac{\lambda^X_{1}}{\lambda^X_{n}} \right )^k \| q_n \|_{\infty}. \nonumber
\end{align}
Since $k \geq 3$, Lemma \ref{sumofnorms} shows that the right-hand side is finite. We can thus use Fubini's theorem in \eqref{corelldecompsg1} and get 
\begin{align}
C_k(f,g) & = \int_{\mathbb{R}} m(x) f(x) \overline{q_1(x)} \left ( \sum_{n \geq 2} \left ( \frac{\lambda^X_{1}}{\lambda^X_{n}} \right )^k \psl q_1 g, q_n \psrm q_n(x) \right ) dx. \label{corelldecompsg2}
\end{align}
We have $m \in L^2(\mathbb{R})$ by \eqref{mheavytailed}, $q_1 \in L^2(\mathbb{R})$ by Proposition \ref{bonfctpropres} and $f \in \mathcal{C}_b(\mathbb{R})$ so $mf\overline{q_1} \in L^1(\mathbb{R})$ by Cauchy-Schwartz inequality. Also, \eqref{corelldecompsgfub} shows that the convergence is the series in the above expression is uniform in $x$. We can thus intervene the sum and the series in \eqref{corelldecompsg2} and get 
\begin{align*}
C_k(f,g) & = \sum_{n \geq 2} \left ( \frac{\lambda^X_{1}}{\lambda^X_{n}} \right )^k \overline{\psl q_1 \overline{f}, q_n \psrm} \psl q_1 g, q_n \psrm \\
& = \left ( \frac{\lambda^X_{1}}{\lambda^X_{2}} \right )^k \left ( \sum_{j=2}^{\multip_{2}+1} \overline{\psl q_1 \overline{f}, q_j \psrm} \psl q_1 g, q_j \psrm + \sum_{n > \multip_{2}+1} \left ( \frac{\lambda^X_{2}}{\lambda^X_{n}} \right )^k \overline{\psl q_1 \overline{f}, q_n \psrm} \psl q_1 g, q_n \psrm \right ), 
\end{align*}
where $\multip_{2}$ denotes the multiplicity of the eigenvalue $\lambda_{2}^X$ (ie the number of indices $j$ such that $\lambda_{j}^X=\lambda_{2}^X$). Note that $|\psl q_1 f, q_n \psrm \psl q_1 g, q_n \psrm| \leq \|f\|_{\infty} \|g\|_{\infty}$ so, by dominated convergence, we get that the term $\sum_{n > \multip_{2}+1} \dots$ converges to $0$ as $k$ goes to infinity. This yields that \eqref{cordecayequiv} holds with $\mathcal{C}=\log(\lambda_{2}^X/\lambda_{1}^X)$ and $B(f,g)=\sum_{j=2}^{\multip_{2}+1} \overline{\psl q_1 \overline{f}, q_j \psrm} \psl q_1 g, q_j \psrm$. We have $|B(f,g)| \leq \multip_{2} \|f\|_{\infty} \|g\|_{\infty}$ so $B(\cdot,\cdot)$ is continuous. To see that $B(\cdot,\cdot)$ is non-zero, take $f_M$ as in \eqref{truncfct} and note that $f_Mq_2/q_1 \in \mathcal{C}_b(\mathbb{R})$. Then it is easy to see that $B(f_M\overline{q_2}/\overline{q_1},f_Mq_2/q_1)$ converges to $1$ as $M$ goes to infinity, so $B(\cdot,\cdot)$ is non-zero. 

\section{Some facts about $X^{m,\xi,r}$, its semigroup, and its generator} \label{factsaboutx}

\subsection{Properties of the random time-change} \label{5.1time-change}

In this subsection we prove some properties of the random time-change $A_x^{-1}(\cdot)$ defined in \eqref{defchgttps}. 

\begin{lemme} \label{axinfty}
Assume that \eqref{mheavytailed} holds. If $\xi$ is recurrent (which occurs, in particular, if $\mathbb{E}[|\xi(1)|]<\infty$, see Theorem 36.7 in \cite{Sato}) then $A_x(\infty)=\infty$ almost surely. 
\end{lemme}

\begin{proof}
Since $m$ is continuous and positive (by \eqref{mheavytailed}) we have $\inf_{y \in [-1,1]} m(y)>0$. Note that $A_x(\infty) \geq (\inf_{y \in [-1,1]} m(y)) \int_0^{\infty} \textbf{1}_{|\xi(s)| \leq 1} ds$. If $\xi$ is recurrent then by Theorem 35.4 in \cite{Sato} we have almost surely $\int_0^{\infty} \textbf{1}_{|\xi(s)| \leq 1} ds=\infty$, from which we conclude that $A_x(\infty)=\infty$. 
\end{proof}

\begin{lemme} \label{chgttps}
Assume that \eqref{mheavytailed} holds. 
\begin{enumerate}
\item For any $x \in \mathbb{R}$, $A_x(\cdot)$ and $A_x^{-1}(\cdot)$ are almost surely continuous and increasing on the intervals $[0,\infty)$ and $[0,A_x(\infty))$ respectively. 
\item Let $t \geq 0$, $x \in \mathbb{R}$, and $(x_n)_{n\geq 1}$ be a sequence converging to $x$. Then $A_{x_n}^{-1}(t)$ converges to $A_{x}^{-1}(t)$ almost surely (regardless of $A_{x}^{-1}(t)<\infty$ or $A_{x}^{-1}(t)=\infty$). 
\item Let $t \geq 0$ and $x \in \mathbb{R}$. Then $\xi(\cdot)$ is continuous at $A_x^{-1}(t)$ almost surely on $\{A_{x}^{-1}(t)<\infty\}$. 
\end{enumerate}
\end{lemme}

\begin{proof}
The first point was justified a little after \eqref{defchgttps}. Let us prove the second point. Let $t\geq 0$, $x \in \mathbb{R}$, and $(x_n)_{n\geq 1}$ be a sequence converging to $x$. We fix a realization of $\xi$ such that $A_x(\cdot)$ is continuous and increasing on $[0,\infty)$. Assume first that $A_{x}^{-1}(t)<\infty$ and that there is $\epsilon>0$ such that $A_{x_n}^{-1}(t)-A_{x}^{-1}(t)>\epsilon$ for infinitely many indices $n$ (this includes indices $n$ such that $A_{x_n}^{-1}(t)=\infty$). Let $(n(k))_{k \geq 1}$ be the enumeration of those indices. For any $k \geq 1$ we have $t\geq A_{x_{n(k)}}(A_{x}^{-1}(t)+\epsilon)$. Let $\rho(\cdot)$ denote the modulus of continuity of $m$, ie $\rho(\delta):=\sup \{ |m(y)-m(z)|, y,z \in \mathbb{R} \ \text{s.t.} \ |y-z|\leq \delta \}$. Since $m \in \crz$ (by \eqref{mheavytailed}), $m$ is uniformly continuous on $\mathbb{R}$ so $\rho(\delta)$ converges to $0$ as $\delta$ goes to $0$. Note from \eqref{defchgttps} that for any $T \geq 0$, $|A_x(T)-A_{x_{n(k)}}(T)|\leq T \rho(|x-x_{n(k)}|)$, so $A_{x_{n(k)}}(A_{x}^{-1}(t)+\epsilon)$ converges to $A_{x}(A_{x}^{-1}(t)+\epsilon)$ as $k$ goes to infinity. We thus get $t\geq A_{x}(A_{x}^{-1}(t)+\epsilon)$. Since $A_x(\cdot)$ is increasing on $[0,\infty)$ and $A_{x}^{-1}(t)<\infty$ we get $A_{x}(A_{x}^{-1}(t)+\epsilon)>A_{x}(A_{x}^{-1}(t))=t$, which is a contradiction. Still in the case $A_{x}^{-1}(t)<\infty$, if there is $\epsilon>0$ such that $A_{x_n}^{-1}(t)-A_{x}^{-1}(t)<-\epsilon$ for infinitely many indices $n$ we obtain a similar contradiction. In conclusion, if $A_{x}^{-1}(t)<\infty$, then for any $\epsilon >0$ we have for all large $n$ that $A_{x_n}^{-1}(t)<\infty$ and $|A_{x_n}^{-1}(t)-A_{x}^{-1}(t)|\leq \epsilon$. This proves the claimed result for any fixed realization of $\xi$ such that $A_x(\cdot)$ is continuous and increasing on $[0,\infty)$, and $A_{x}^{-1}(t)<\infty$. Now assume that $A_{x}^{-1}(t)=\infty$ and let $M>0$. We have $A_{x}(M+1)\leq t$ so $A_{x}(M)< t$. Using $|A_x(M)-A_{x_n}(M)|\leq M \rho(|x-x_n|)$ we get that for all large $n$, $A_{x_n}(M)< t$ so $A_{x_n}^{-1}(t)\geq M$. Therefore the claimed result is now also proved for any fixed realization of $\xi$ such that $A_x(\cdot)$ is continuous and increasing on $[0,\infty)$, and $A_{x}^{-1}(t)=\infty$. Combining with the first point of the lemma, we get the second point. 

We now prove the third point. Recall that $A_x^{-1}(t)$ is a $\mathcal{F}^{\xi}$-stopping time, where $\mathcal{F}^{\xi}$ is the filtration defined in Section \ref{notations}. By the first point, on $\{A_x^{-1}(t)<\infty\}$, we have almost surely that the sequence of $\mathcal{F}^{\xi}$-stopping times $(A_x^{-1}(t-1/n))_{n \geq 1}$ increases to $A_x^{-1}(t)$ and that $A_x^{-1}(t-1/n)<A_x^{-1}(t)$ for any $n\geq 1$. By Proposition I.7 of \cite{Bertoin} we get that $\xi$ is continuous at $A_x^{-1}(t)$ almost surely on $\{A_x^{-1}(t)<\infty\}$. This proves the third point. 
\end{proof}

\subsection{Markov property, Feller property} \label{5.2Markov-Feller}

\begin{lemme} \label{markovianity}
Assume that $r>0$ and \eqref{mheavytailed} holds. The process $X^{m,\xi,r}$ defined by \eqref{defX} is an homogeneous Markovian process. In particular the family $(P_t)_{t \geq 0}$ (defined by \eqref{defsgx}) is a semigroup (i.e. it satisfies $P_{t+s}.f=P_s.P_t.f$ for any bounded measurable $f$ and $t,s \geq 0$). 
\end{lemme}

\begin{proof}
We fix $x \in \mathbb{R}$. Recall the filtrations $\mathcal{F}^{\xi}$ and $\mathcal{F}^X$ defined in Section \ref{notations}. Note that for any $t \geq 0$, $A_x^{-1}(t)$ is a (possibly infinite) $\mathcal{F}^{\xi}$-stopping time and that $\mathcal{F}^X_t = \mathcal{F}^{\xi}_{A_x^{-1}(t)}$. Let us fix $t \geq 0$ and prove the Markov property for $X^{m,\xi,r}_x$ at time $t$. We have clearly $X^{m,\xi,r}(t+s)=\dagger$ for any $s \geq 0$ if we are on $\{X^{m,\xi,r}_x(t)=\dagger\}=\{A_x^{-1}(t)\geq e_r\}$ so let assume we are on $\{X^{m,\xi,r}_x(t)\neq \dagger\}=\{A_x^{-1}(t)< e_r\}$. By the strong Markov property for $\xi$, the process $\hat \xi:=\xi(A_x^{-1}(t)+\cdot)-\xi(A_x^{-1}(t))$ has same law as $\xi$ and is independent of $\mathcal{F}^{\xi}_{A_x^{-1}(t)}=\mathcal{F}^X_t$. Let also $\hat e_r := e_r-A_x^{-1}(t)$. For $y \in \mathbb{R}$, let $\hat A_y$ and $\hat X^{m,\xi,r}_y$ be defined by \eqref{defchgttps} and \eqref{defX}, but with $\xi$ and $e_r$ replaced by $\hat \xi$ and $\hat e_r$ respectively. We see that, conditionally on $\mathcal{F}^X_t$ and on $\{X^{m,\xi,r}_x(t)\neq \dagger\}$, $\hat X^{m,\xi,r}_y$ has the same law as $X^{m,\xi,r}_y$. Let us choose $y=X^{m,\xi,r}_x(t)=x+\xi(A_x^{-1}(t))$. For any $s \in [0, \hat A_y(\infty))$, 
\begin{align*}
s & = \int_{0}^{\hat A^{-1}_y(s)} m(y+\hat \xi(u)) du = \int_{0}^{\hat A^{-1}_y(s)} m(x+\xi(A_x^{-1}(t)+u)) du \\
& = \int_{A_x^{-1}(t)}^{A_x^{-1}(t)+\hat A^{-1}_y(s)} m(x+\xi(u)) du = A_x(A_x^{-1}(t)+\hat A^{-1}_y(s))-t. 
\end{align*}
We thus get $A_x^{-1}(t+s)<\infty$ and $A_x^{-1}(t+s)=A_x^{-1}(t)+\hat A^{-1}_y(s)$. For $s \geq \hat A_y(\infty)$, we have $s >\int_{0}^{M} m(y+\hat \xi(u)) du$ for any $M>0$, thanks to the positivity of $m$. Proceeding as above we get $s>A_x(A_x^{-1}(t)+M)-t$ for any $M>0$ so $A_x^{-1}(t+s)=\infty$. In conclusion $A_x^{-1}(t+s)< \infty \Leftrightarrow s < \hat A_y(\infty) \Leftrightarrow \hat A^{-1}_y(s) < \infty$ and in that case we have $A_x^{-1}(t+s)< e_r \Leftrightarrow \hat A^{-1}_y(s) < \hat e_r$. Therefore, 
\begin{align*}
X^{m,\xi,r}_x(t+s) & = x+\xi(A_x^{-1}(t+s)) = x+\xi(A_x^{-1}(t)+\hat A^{-1}_y(s)) = y + \hat \xi(\hat A^{-1}_y(s)) = \hat X^{m,\xi,r}_y(s). 
\end{align*}
We have thus obtained that $X^{m,\xi,r}_x(t+\cdot)$ equals $\hat X^{m,\xi,r}_y$ with $y=X^{m,\xi,r}_x(t)$ and we know that, conditionally on $\mathcal{F}^X_t$ and on $\{X^{m,\xi,r}_x(t)\neq \dagger\}$, the later is distributed as a version of $X^{m,\xi,r}$ starting at $X^{m,\xi,r}_x(t)$. This shows that $X^{m,\xi,r}$ is Markovian and homogeneous. The semigroup property for $(P_t)_{t \geq 0}$ trivially follows. 
\end{proof}

We now prove the Feller property for $X^{m,\xi,r}$. 

\begin{proof}[Proof of Proposition \ref{fellerity}]
From \eqref{defsgx} and Lemma \ref{markovianity} we see that $(P_t)_{t \geq 0}$ is a sub-Markov semigroup in the sense of Definition 1.1 of \cite{levymatters3}. In order to check that $(P_t)_{t \geq 0}$ satisfies Definition 1.2 of \cite{levymatters3} we need to show that $P_t.f \in \crz$ for any $f \in \crz$ and $t >0$ (Feller property) and that $\|P_t.f-f\|_{\infty}$ converges to $0$ as $t$ goes to $0$ for any $f \in \crz$ (strong continuity). 

We first prove the strong continuity. Let $f \in \crz$, $\epsilon>0$ and let us prove the existence of $t_0>0$ such that $t < t_0 \Rightarrow \|P_t.f-f\|_{\infty} \leq \epsilon$. If $f$ is the null function the claim trivially holds so we assume that $f$ is not identically $0$. Since $f \in \crz$ is uniformly continuous, there is $\delta>0$ such that $|x-y|\leq \delta \Rightarrow |f(x)-f(y)|\leq \epsilon/4$. Let us fix $\tilde \delta>0$ small enough and $M>0$ large enough such that 
\begin{align} \label{condtdeltam}
\mathbb{P} \left (e_r \leq \tilde \delta \right) \leq \frac{\epsilon}{8 \|f\|_{\infty}}, \ \mathbb{P} \left (\sup_{s \in [0,\tilde \delta]} |\xi(s)| > \delta \right) \leq \frac{\epsilon}{8 \|f\|_{\infty}}, \ \mathbb{P} \left (\sup_{s \in [0,e_r]} |\xi(s)| > M \right) \leq \frac{\epsilon}{8 \|f\|_{\infty}}, \ \sup_{|x| \geq M} |f(x)| \leq \frac{\epsilon}{3}. 
\end{align}
We set $t_0 := \tilde \delta \times \inf_{z\in [-3M,3M]} m(z)$. Since $m$ is positive and continuous by \eqref{mheavytailed} we have $t_0>0$. We now prove that $\|P_t.f-f\|_{\infty} \leq \epsilon$ for all $t\in (0,t_0)$. 
First let $t\in (0,t_0)$ and $|x| \leq 2M$. Note from \eqref{defsgx} that 
\begin{align} \label{boundptf-f1}
|P_t.f(x)-f(x)| \leq \mathbb{E} \left [ \left | f(x+\xi(A_x^{-1}(t))) - f(x) \right | \textbf{1}_{A_x^{-1}(t)<e_r} + \left | f(x) \right | \textbf{1}_{A_x^{-1}(t)\geq e_r} \right ]. 
\end{align}
If $\sup_{s \in [0,e_r]} |\xi(s)| \leq M$ and $e_r > \tilde \delta$ then using \eqref{defchgttps} and the definition of $t_0$, 
\[ A_x(\tilde \delta) = \int_0^{\tilde \delta} m(x+\xi(u))du \geq \tilde \delta \times \inf_{u \in [0,\tilde \delta]} m(x+\xi(u)) \geq \tilde \delta \times \inf_{u \in [0,e_r]} m(x+\xi(u)) \geq \tilde \delta \times \inf_{z\in [-3M,3M]} m(z) = t_0 > t. \]
Therefore $A_x^{-1}(t)\leq \tilde \delta < e_r$. If moreover $\sup_{s \in [0,\tilde \delta]} |\xi(s)| \leq \delta$ then $|\xi(A_x^{-1}(t))|\leq \delta$ so $| f(x+\xi(A_x^{-1}(t))) - f(x)|\leq \epsilon/4$. Plugging into \eqref{boundptf-f1} and then using \eqref{condtdeltam} we get that for $t\in (0,t_0)$ and $|x| \leq 2M$, 
\begin{align}
|P_t.f(x)-f(x)| & \leq \epsilon/4 + 2 \|f\|_{\infty} \left ( \mathbb{P} \left (\sup_{s \in [0,e_r]} |\xi(s)| > M \right) + \mathbb{P} \left (e_r \leq \tilde \delta \right) + \mathbb{P} \left (\sup_{s \in [0,\tilde \delta]} |\xi(s)| > \delta \right) \right ) \label{boundptf-f2} \leq \epsilon. 
\end{align}
Now let $t >0$ and $|x| \geq 2M$. We have from \eqref{defsgx} that
\begin{align*}
|P_t.f(x)| \leq \mathbb{E} \left [ \left | f(x+\xi(A_x^{-1}(t))) \right | \textbf{1}_{A_x^{-1}(t)<e_r} \right ]. 
\end{align*}
If $\sup_{s \in [0,e_r]} |\xi(s)| \leq M$, we have $|x+\xi(A_x^{-1}(t))|\geq M$ so, by \eqref{condtdeltam}, $|f(x+\xi(A_x^{-1}(t)))| \leq \epsilon/3$. Therefore, for any $t >0$ and $|x| \geq 2M$, 
\begin{align}
|P_t.f(x)| \leq \epsilon/3 + \|f\|_{\infty} \times \mathbb{P} \left (\sup_{s \in [0,e_r]} |\xi(s)| > M \right) \leq 2\epsilon/3, \label{boundptf-f3}
\end{align}
where the last equality comes from \eqref{condtdeltam}. We also have $|f(x)| \leq \epsilon/3$ by \eqref{condtdeltam}. Combining with \eqref{boundptf-f3} we get $|P_t.f(x)-f(x)| \leq \epsilon$ for all $t >0$ and $|x| \geq 2M$. The combination of this with \eqref{boundptf-f2} yields $\|P_t.f-f\|_{\infty} \leq \epsilon$ for all $t\in (0,t_0)$, proving the strong continuity. 

Let us fix $f \in \crz$, $t>0$ and now prove that $P_t.f \in \crz$. Let $x \in \mathbb{R}$ and $(x_n)_{n\geq 1}$ be a sequence converging to $x$. Recall from \eqref{defsgx} that $P_t.f(x_n)=\mathbb{E}[f(x_n+\xi(A_{x_n}^{-1}(t))) \textbf{1}_{A_{x_n}^{-1}(t)<e_r}]$. Using the second and third points of Lemma \ref{chgttps} and $f \in \crz$ we see that what is inside the expectation converges almost surely to $f(x+\xi(A_{x}^{-1}(t))) \textbf{1}_{A_{x}^{-1}(t)<e_r}$ (the convergence may not hold on $\{A_{x}^{-1}(t)=e_r\}$ but that event has probability $0$). Moreover, what is inside the expectation is bounded by $\|f\|_{\infty}$ so, by dominated convergence, we get that $P_t.f(x_n)$ converges to $P_t.f(x)$. It follows that $P_t.f$ is continuous. Then fix $\epsilon >0$ and note from \eqref{boundptf-f3} that there is $M>0$ such that $|x| \geq 2M \Rightarrow |P_t.f(x)| \leq \epsilon$. Therefore $P_t.f(x) \longrightarrow_{|x| \rightarrow \infty} 0$, so $P_t.f \in \crz$. This concludes the proof. 
\end{proof}

\subsection{Support of $X^{m,\xi,r}$} \label{5.3support}

The following lemma is crucial to prove Lemma \ref{q1posonr} which plays an important role in showing the positivity of the constant $K_{m,\xi,r}(x)$ in Theorem \ref{encadrementtransfreeenergythtrans}, and in relating it with the infinite-volume Gibbs states of the spin system from Section \ref{interpartsyst}. 
\begin{lemme} \label{suppxtisr}
Assume that $r>0$ and \eqref{mheavytailed} and \eqref{hypcaractexpol-1} hold. For any $t>0$ and $x \in \mathbb{R}$ we have $Supp(X^{m,\xi,r}_x(t))=\mathbb{R} \cup \{\dagger\}$. 
\end{lemme}

\begin{proof}
Let us fix $t>0$ and $x\in \mathbb{R}$. To show that $\mathbb{R}\subset Supp(X^{m,\xi,r}_x(t))$, we prove that for any $z\in \mathbb{R}$ and $\epsilon>0$ we have $\mathbb{P}(X^{m,\xi,r}_x(t)\in [z-\epsilon,z+\epsilon])>0$. Let us fix $\delta \in (0,t/\| m \|_{\infty})$ and $T>\delta+t/\min_{y \in [z-x-\epsilon,z-x+\epsilon]} m(y)$. We define the event 
\[ \mathcal{E} := \{ e_r>T \} \cap \left \{ \xi(\delta) \in (z-x-\epsilon/2,z-x+\epsilon/2) \right \} \cap \left \{ \sup_{s \in [\delta,T]}|\xi(s)-\xi(\delta)|<\epsilon/2 \right \}. \]
From \eqref{defchgttps} we see that on this event we have 
\[ A_x(\delta) \leq \delta \| m \|_{\infty} < t < (T-\delta) \times \min_{y \in [z-x-\epsilon,z-x+\epsilon]} m(y) \leq A_x(T) < A_x(e_r), \]
and $x+\xi^r(s) \in [z-\epsilon,z+\epsilon]$ for all $s \in [\delta,T]$. In conclusion, we have 
\[ \mathcal{E} \subset \left \{ \delta < A_x^{-1}(t) < T < e_r \right \} \cap \left \{ \forall s \in [\delta,T], \ x+\xi^r(s) \in [z-\epsilon,z+\epsilon] \right \} \subset \left \{ X^{m,\xi,r}_x(t) \in [z-\epsilon,z+\epsilon] \right \}, \]
where the last inclusion comes from the definition of $X^{m,\xi,r}_x(t)$ in \eqref{defX}. 

We are now left to prove that $\mathbb{P}(\mathcal{E})>0$. Using that $e_r$ is independent of $\xi$ and that $\xi(\delta)$ and $(\xi(s)-\xi(\delta))_{s \geq \delta}$ are independent we get 
\begin{align}
\mathbb{P}(\mathcal{E}) = e^{-rT} \times \mathbb{P} \left ( \xi(\delta) \in (z-x-\epsilon/2,z-x+\epsilon/2) \right ) \times \mathbb{P} \left ( \sup_{s \in [\delta,T]}|\xi(s)-\xi(\delta)|<\epsilon/2 \right ). \label{probaspecialevent}
\end{align}
By Lemma \ref{suppxiisr}, $Supp(\xi(\delta))=\mathbb{R}$, so the second factor in \eqref{probaspecialevent} is positive. Lemma \ref{xismallonint} shows that $\mathbb{P}(\sup_{s \in [0,T-\delta]}|\xi(s)|<\epsilon/2)>0$. By the Markov property at time $\delta$, we get that the third factor in \eqref{probaspecialevent} is positive. This concludes the proof of $\mathbb{R}\subset Supp(X^{m,\xi,r}_x(t))$. 

To show that $\dagger \in Supp(X^{m,\xi,r}_x(t))$, note that the event $\{e_r<t/\| m \|_{\infty}\}$ has positive probability and is included into $\{A^{-1}_x(t) \geq e_r\}=\{X^{m,\xi,r}_x(t)=\dagger\}$. Therefore $\dagger \in Supp(X^{m,\xi,r}_x(t))$, so the proof is complete. 
\end{proof}

\subsection{Properties of the generator} \label{5.3generator}

\begin{lemme} \label{generatorxnew}
Assume that $r>0$ and \eqref{mheavytailed} holds true. If $f \in \mathcal{D}(\mathcal{A}_{\xi^r})$ is such that $\frac1{m}\mathcal{A}_{\xi^r} f \in \crz$ then $f \in \mathcal{D}(\mathcal{A}_{X^{m,\xi,r}})$ and we have $\mathcal{A}_{X^{m,\xi,r}}f(x) = \frac1{m(x)} \mathcal{A}_{\xi^r} f(x)$. 
\end{lemme}

\begin{proof}
Let $f \in \mathcal{D}(\mathcal{A}_{\xi^r})$ and $g(x) := \frac1{m(x)} \mathcal{A}_{\xi^r} f(x)$. It is assumed that $g \in \crz$ so, according to Theorem 1.33 in \cite{levymatters3}, we only need to prove 
\begin{align}
\forall x \in \mathbb{R}, \ t^{-1}(P_t.f(x)-f(x)) \underset{t \rightarrow 0}{\longrightarrow} g(x). \label{pointwise}
\end{align}
For this, let us proceed similarly as for the proof of Corollary 4.2 in \cite{levymatters3}. This requires some carefulness since we are not in their setting, as $1/m$ is not a bounded function. Since $f \in \mathcal{D}(\mathcal{A}_{\xi^r})$ we have that $(M^f_t)_{t \geq 0}$ defined by $M^f_t := f(x+\xi(t))\textbf{1}_{t<e_r}-f(x)-\int_0^{t\wedge e_r} \mathcal{A}_{\xi^r}f(x+\xi(s))ds$ is a $\mathcal{F}^{\xi}$-martingale (see for example Theorem 1.36 in \cite{levymatters3}), where $\mathcal{F}^{\xi}$ is the filtration defined in Section \ref{notations}. Therefore, using that $dA_x(u)=m(x+\xi(u))du$, the definition of $g$, and making the substitution $s=A_x(u)$ we get that for any $t \geq 0$ and $n \geq 1$, 
\begin{align*}
& f(x+\xi(A_x^{-1}(t)\wedge n))\textbf{1}_{A_x^{-1}(t)\wedge n<e_r}-f(x)-M^f_{A_x^{-1}(t)\wedge n} =\int_0^{A_x^{-1}(t)\wedge n \wedge e_r} \mathcal{A}_{\xi^r}f(x+\xi(u))du \\
= & \int_0^{A_x^{-1}(t)\wedge n\wedge e_r} g(x+\xi(u)) dA_x(u) =\int_0^{t \wedge A_x(n)\wedge A_x(e_r)} g(x+\xi(A_x^{-1}(s))) ds. 
\end{align*}
Since $A_x^{-1}(t)\wedge n$ is a bounded $\mathcal{F}^{\xi}$-stopping time, we get from the optional stopping theorem 
\begin{align}
\mathbb{E}[f(x+\xi(A_x^{-1}(t)\wedge n))\textbf{1}_{A_x^{-1}(t)\wedge n<e_r}]-f(x) = \mathbb{E} \left [ \int_0^{t \wedge A_x(n) \wedge A_x(e_r)} g(x+\xi(A_x^{-1}(s))) ds \right ]. \label{generatorxcalc}
\end{align}
From the third point of Lemma \ref{chgttps} and the continuity of $f$, what is inside the expectation in the left-hand side of \eqref{generatorxcalc} converges to $f(x+\xi(A_x^{-1}(t)))\textbf{1}_{A_x^{-1}(t)<e_r}$ as $n$ increases to infinity. Since $f$ and $g$ are bounded functions, we get by dominated convergence that 
\begin{align*}
\mathbb{E}[f(x+\xi(A_x^{-1}(t)))\textbf{1}_{A_x^{-1}(t)<e_r}]-f(x) = \mathbb{E} \left [ \int_0^{t \wedge A_x(e_r)} g(x+\xi(A_x^{-1}(s))) ds \right ]. 
\end{align*}
By \eqref{defsgx}, this translates into 
\begin{align}
t^{-1}(P_t.f(x)-f(x)) = \mathbb{E} \left [ t^{-1}\int_0^{t \wedge A_x(e_r)} g(x+\xi(A_x^{-1}(s))) ds \right ]. \label{dersgx}
\end{align}
From the first point of Lemma \ref{chgttps}, $A_x^{-1}(\cdot)$ is almost surely continuous on $[0,A_x(e_r))$ and $A_x(e_r)$ is almost surely positive. Moreover $\xi$ is almost surely right-continuous by definition of a L\'evy process and $g$ is continuous. We deduce that what is inside the expectation in the right-hand side of \eqref{dersgx} converges almost surely to $g(x)$ as $t$ goes to $0$. Moreover it is bounded by $\|g\|_{\infty}$. By dominated convergence we obtain \eqref{pointwise}, which concludes the proof. 
\end{proof}

\subsection{Properties of the semigroup and resolvent} \label{5.4semigroup}

\begin{lemme} \label{imptinl2m}
Assume that $r>0$ and \eqref{mheavytailed} holds. We have $P_t(\crz)\subset L^2(\mathbb{R}) \subset L^2(m)$ for all $t>0$. 
\end{lemme}

\begin{proof}
Let $t>0$ and $f \in \crz$. By \eqref{defsgx}, Markov inequality, two times Jensen inequality, \eqref{defchgttps}, Fubini's theorem and $m \in L^2(\mathbb{R})$ (from \eqref{mheavytailed}) we get 
\begin{align*}
\int_{\mathbb{R}} |P_t.f(x)|^2 dx & \leq \| f \|_{\infty}^2 \int_{\mathbb{R}} \mathbb{P} ( A_x^{-1}(t)<e_r)^2 dx = \| f \|_{\infty}^2 \int_{\mathbb{R}} \mathbb{P} ( A_x(e_r) > t)^2 dx \leq \frac{\| f \|_{\infty}^2}{t^2} \int_{\mathbb{R}} \mathbb{E} [A_x(e_r)]^2 dx \\
& \leq \frac{\| f \|_{\infty}^2}{t^2} \int_{\mathbb{R}} \mathbb{E} \left [A_x(e_r)^2 \right ] dx \leq \frac{\| f \|_{\infty}^2}{t^2} \int_{\mathbb{R}} \mathbb{E} \left [e_r \int_0^{e_r} m(x+\xi(u))^2du \right ] dx \\
& = \frac{\| f \|_{\infty}^2}{t^2} \mathbb{E} \left [e_r \int_0^{e_r} \int_{\mathbb{R}} m(x+\xi(u))^2 dx du \right ] = \frac{\| f \|_{\infty}^2 \|m\nr^2}{t^2} \mathbb{E} [e_r^2] = \frac{2 \| f \|_{\infty}^2 \|m\nr^2}{r^2 t^2} < \infty.  
\end{align*}
\end{proof}
The above proof even shows that the operator $P_t$ is continuous from $(\crz, \| \cdot \|_{\infty})$ to $L^2(\mathbb{R})$ and $L^2(m)$. 
\begin{remark} \label{survprobinl2}
Applying the above proof to the constant function equal to $1$, instead of $f \in \crz$, one gets that $P_t.1=(x \mapsto \mathbb{P}(\zeta_x > t)) \in L^2(\mathbb{R}) \subset L^2(m)$. 
\end{remark}

Before proving the next lemma, let us recall that for any $\alpha\geq 0$, the resolvent operator at $\alpha$ associated with the semigroup $(P_t)_{t \geq 0}$ is defined by $U_{\alpha}f := \int_0^{\infty} e^{-\alpha t} P_t.f dt$ for $f \in \crz$ (see Definition 1.21 in \cite{levymatters3}). Assume that $r>0$ and \eqref{mheavytailed} holds. Using \eqref{defsgx}, Fubini's theorem, the change of variable $s=A_x^{-1}(t)$, and again Fubini's theorem, we get that for any $x \in \mathbb{R}$, 
\begin{align}
U_{\alpha}f(x) & = \mathbb{E} \left [ \int_0^{A_x(e_r)} e^{-\alpha t} f(x+\xi(A_x^{-1}(t))) dt \right ] = \mathbb{E} \left [ \int_0^{e_r} e^{-\alpha A_x(s)} (fm)(x+\xi(s)) ds \right ] \label{exprresolvante0} \\
& = \mathbb{E} \left [ \int_0^{\infty} \textbf{1}_{s \leq e_r} e^{-\alpha A_x(s)} (fm)(x+\xi(s)) ds \right ] = \int_0^{\infty} e^{-rs} \mathbb{E} \left [ e^{-\alpha A_x(s)} (fm)(x+\xi(s)) \right ] ds. \nonumber
\end{align}

\begin{lemme} \label{resselfadj}
Assume that $r>0$ and \eqref{mheavytailed} holds. For any $\alpha>0$ we have $U_{\alpha}(\crz)\subset \crz \cap L^2(\mathbb{R}) \subset L^2(m)$. Moreover, we have $\psl U_{\alpha}f, g \psrm=\psl f, U_{\alpha}g \cdot \psrm$ for all $\alpha>0$ and $f,g \in \crz \cap L^2(\mathbb{R})$. 
\end{lemme}

\begin{proof}
From Lemma 1.27 in \cite{levymatters3} we get $U_{\alpha}(\crz)\subset \mathcal{D}(\mathcal{A}_{X^{m,\xi,r}}) \subset \crz$. Let $f \in \crz$. Using \eqref{exprresolvante0}, two times Jensen's inequality, Fubini's theorem, and $m \in L^2(\mathbb{R})$ (from \eqref{mheavytailed}) we get 
\begin{align*}
\int_{\mathbb{R}} \left | U_{\alpha}f(x) \right |^2 dx & \leq \frac1{r} \int_{\mathbb{R}} \int_0^{\infty} e^{-rs} \mathbb{E} \left [ \left | e^{-\alpha A_x(s)} (fm)(x+\xi(s)) \right |^2 \right ] ds dx \\
& \leq \frac{\| f \|_{\infty}^2}{r} \int_0^{\infty} e^{-rs} \mathbb{E} \left [ \int_{\mathbb{R}} \left | m(x+\xi(s)) \right |^2 dx \right ] ds = \frac{\| f \|_{\infty}^2 \|m\nr^2}{r^2} <\infty. 
\end{align*}
Therefore $U_{\alpha}(\crz)\subset L^2(\mathbb{R})$. 

Now let $f,g \in \crz \cap L^2(\mathbb{R})$. Therefore $U_{\alpha}f,U_{\alpha}g,f$ and $g$ are all in $L^2(\mathbb{R}) \subset L^2(m)$. Using \eqref{exprresolvante0} we have 
\begin{align}
\psl U_{\alpha}f, g \psrm=\int_{\mathbb{R}} U_{\alpha}f(x)\overline{g(x)}m(x)dx=\int_{\mathbb{R}} \int_0^{\infty} e^{-rs} \mathbb{E} \left [ e^{-\alpha A_x(s)} (fm)(x+\xi(s))(\overline{g}m)(x) \right ] ds dx. \label{resselfadj1}
\end{align}
By $m \in L^2(\mathbb{R})$ and Cauchy-Schwartz inequality, $\overline{g}m \in L^1(\mathbb{R})$, so $\int_{\mathbb{R}} \int_0^{\infty} e^{-rs} \mathbb{E} [ |e^{-\alpha A_x(s)} (fm)(x+\xi(s))(\overline{g}m)(x)| ] ds dx \leq \| f \|_{\infty} \|m\|_{\infty} \|\overline{g}m\|_{L^1(\mathbb{R})}/r<\infty$. We can thus use Fubini's theorem in \eqref{resselfadj1} and get 
\begin{align}
\psl U_{\alpha}f, g \psrm=\int_0^{\infty} e^{-rs} \mathbb{E} \left [ \int_{\mathbb{R}} e^{-\alpha A_x(s)} (fm)(x+\xi(s))(\overline{g}m)(x) dx \right ] ds. \label{resselfadj2}
\end{align}
For any $s>0$ let us define the process $\xi^s$ on $[0,s]$ by $\xi^s(s):=-\xi(s)$ and $\xi^s(u):=\xi((s-u)-)-\xi(s)$ when $u \in [0,s)$. By Lemma II.2 of \cite{Bertoin}, $(\xi^s(u))_{u \in [0,s]}$ is a L\'evy process equal in law to $(-\xi(u))_{u \in [0,s]}$ which, by the symmetry of $\xi$, is equal in law to $(\xi(u))_{u \in [0,s]}$. Using \eqref{defchgttps}, the change of variable $v=s-u$, that $\xi$ is continuous at almost every time, and the definition of $\xi^s$, we see that we have a.s. for all $x \in \mathbb{R}$, 
\[ A_x(s) = \int_0^{s} m(x+\xi(s)+\xi((s-u)-)-\xi(s))du = \int_0^{s} m(x-\xi^s(s)+\xi^s(u))du = \tilde A_{x-\xi^s(s)}(s), \]
where we have set $\tilde A_{y}(t):= \int_0^{t} m(y+\xi^s(u))du$ for any $y \in \mathbb{R}$ and $t \in [0,s]$. Plugging this in \eqref{resselfadj2} we get 
\begin{align}
\psl U_{\alpha}f, g \psrm & =\int_0^{\infty} e^{-rs} \mathbb{E} \left [ \int_{\mathbb{R}} e^{-\alpha \tilde A_{x-\xi^s(s)}(s)} (fm)(x-\xi^s(s))(\overline{g}m)(x) dx \right ] ds \label{resselfadj3} \\
& =\int_0^{\infty} e^{-rs} \mathbb{E} \left [ \int_{\mathbb{R}} e^{-\alpha \tilde A_{y}(s)} (fm)(y)(\overline{g}m)(y+\xi^s(s)) dy \right ] ds. \nonumber
\end{align}
Since $(\xi^s(u))_{u \in [0,s]}$ is equal in law to $(\xi(u))_{u \in [0,s]}$ we have $\mathbb{E}[F(\tilde A_{\cdot}(s),\xi^s(s))]=\mathbb{E}[F(A_{\cdot}(s),\xi(s))]$ for any integrable function $F$. \eqref{resselfadj3} thus becomes 
\begin{align*}
\psl U_{\alpha}f, g \psrm =\int_0^{\infty} e^{-rs} \mathbb{E} \left [ \int_{\mathbb{R}} e^{-\alpha A_{y}(s)} (fm)(y)(\overline{g}m)(y+\xi(s)) dx \right ] ds = \psl f, U_{\alpha}g \psrm, 
\end{align*}
where, for the last equality, we have used \eqref{resselfadj2} but where the roles of $f$ and $g$ are switched. 
\end{proof}

The following lemma shows that $P_t$ is self-adjoint for $\psl \cdot , \cdot \psrm$ on $\crz \cap L^2(\mathbb{R}) \subset L^2(m)$. 
\begin{lemme} \label{ptautoadj}
Assume that $r>0$ and \eqref{mheavytailed} holds. We have $\psl P_t.f, g \psrm=\psl f, P_t.g \cdot \psrm$ for all $t\geq 0$ and $f,g \in \crz \cap L^2(\mathbb{R})$. 
\end{lemme}

\begin{proof}
The claim is obvious for $t=0$ so let us assume $t>0$ and fix $f,g \in \crz \cap L^2(\mathbb{R})$. By Lemma \ref{resselfadj} we get that for any $n \geq 1$, 
\begin{align}
\int_{\mathbb{R}} \left ( \left (\frac{n}{t} U_{n/t} \right )^nf \right )(x)\overline{g(x)}m(x)dx=\int_{\mathbb{R}} f(x)\left ( \left (\frac{n}{t} U_{n/t} \right )^n\overline{g} \right )(x)m(x)dx. \label{autoadjapproxsg}
\end{align}
Since $f$ and $\overline{g}$ are in $\crz$, by (1.31) in \cite{levymatters3} we have that $(\frac{n}{t} U_{n/t})^nf$ and $(\frac{n}{t} U_{n/t})^n\overline{g}$ converge to respectively $P_t.f$ and $P_t.\overline{g}$ in $(\crz, \| \cdot \|_{\infty})$. Since, by $m \in L^2(\mathbb{R})$ (from \eqref{mheavytailed}) and Cauchy-Schwartz inequality, $\overline{g}m$ and $fm$ are in $L^1(\mathbb{R})$ we can let $n$ go to infinity on both sides of \eqref{autoadjapproxsg} and get $\int_{\mathbb{R}} (P_t.f)(x)\overline{g(x)}m(x)dx=\int_{\mathbb{R}} f(x)(P_t.\overline{g})(x)m(x)dx$. By Lemma \ref{imptinl2m}, $P_t.f$ and $P_t.g$ are in $L^2(m)$ and, by assumption $f$ and $g$ are in $L^2(\mathbb{R}) \subset L^2(m)$. Therefore the last equality can be re-written as $\psl P_t.f, g \psrm=\psl f, P_t.g \cdot \psrm$. 
\end{proof}

\begin{appendix}
\section{Some facts about the L\'evy process $\xi$} \label{factsLP}

Recall that $\xi$ is a real symmetric L\'evy process and that $\psi_\xi(\cdot)$ denotes its characteristic exponent. We prove below some properties of $\xi$, sometimes assuming that \eqref{hypcaractexpol-1} is satisfied. 

\begin{lemme} \label{linkpsizpotential}
Under \eqref{hypcaractexpol-1} we have that, for any $r>0$, the potential measure $V^r_\xi(dx)$ (defined in Section \ref{notations}) has a continuous density $v^r_\xi(\cdot) \in L^1(\mathbb{R}) \cap \crz$, and for any $x \in \mathbb{R}$, 
\begin{align}
v^r_\xi(x) = \frac1{2\pi} \int_{\mathbb{R}} \frac{e^{ixy}}{r-\psi_\xi(y)} dy = \mathcal{F}^{-1} \left ( \frac1{r-\psi_\xi(2\pi \cdot)} \right ) (x). \label{linkpsizpotentialq}
\end{align}
In particular, for any $x \in \mathbb{R}$ we have $0 \leq v^r_\xi(x) \leq \|(- \psi_\xi(2\pi \cdot) + r)^{-1}\|_{L^1(\mathbb{R})}=v^{r}_\xi(0)$. 
\end{lemme}

As mentioned in the Introduction, the existence, under \eqref{hypcaractexpol-1}, of a continuous density for the potential measure is a consequence of the combination of Remark 43.6 and Theorem 43.5 from \cite{Sato}. However, let us provide a short and direct proof of Lemma \ref{linkpsizpotential} for the sake of completeness. 

\begin{proof}[Proof of Lemma \ref{linkpsizpotential}]
According to Proposition 37.4 of \cite{Sato} we have for any $r>0$ and $x \in \mathbb{R}$, 
\begin{align}
\int_{\mathbb{R}} e^{izx} V^r_\xi(dz) = \frac1{r-\psi_\xi(x)}. \label{fourierqpot}
\end{align}
Note that the convention used in \cite{Sato} for the Fourier transform is different from the one we use here, which is why we do not state \eqref{fourierqpot} in term of $\mathcal{F}$. From \eqref{hypcaractexpol-1} we see that for any $r>0$ the right hand side of \eqref{fourierqpot} is in $L^1(\mathbb{R})$ so $V^r_\xi$ has a continuous density $v^r_\xi(\cdot)$ that satisfies 
\[ v^r_\xi(x) = \frac1{2\pi} \int_{\mathbb{R}} \frac{e^{ixy}}{r-\psi_\xi(y)} dy = \int_{\mathbb{R}} \frac{e^{2i\pi xz}}{r-\psi_\xi(2\pi z)} dz = \mathcal{F}^{-1} \left ( \frac1{r-\psi_\xi(2\pi \cdot)} \right ) (x), \]
where we have used the symmetry of $\xi$. This is precisely \eqref{linkpsizpotentialq}. Finally, since $v^r_\xi(\cdot)$ is the density of a finite measure we have $v^r_\xi(\cdot) \in L^1(\mathbb{R})$ and, by \eqref{linkpsizpotentialq}, $v^r_\xi(\cdot) \in \mathcal{F}^{-1}(L^1(\mathbb{R}))$ so $v^r_\xi(\cdot) \in \crz$. 
\end{proof}

\begin{lemme} \label{suppxiisr}
Under \eqref{hypcaractexpol-1} we have $Supp(\xi(s))=\mathbb{R}$ for any $s>0$. 
\end{lemme}

\begin{proof}
As explained a little below the statement of \eqref{hypcaractexpol-1}, that condition implies that $\xi$ is of type C in the sense of Definition 11.9 of \cite{Sato}. By Theorem 24.10 of \cite{Sato} we get $Supp(\xi(s))=\mathbb{R}$ for all $s>0$. 
\end{proof}

\begin{lemme} \label{xismallonint}
For any $t>0$ and $\epsilon>0$ we have $\mathbb{P}(\sup_{s \in [0,t]}|\xi(s)|<\epsilon)>0$. 
\end{lemme}

\begin{proof}
Let us fix $\eta>0$. By the L\'evy-Ito decomposition we have 
\begin{align}
\xi(s) & = \sqrt{A_{\xi}} W(s) + \int_0^s \int_{[-\eta,\eta]} z \tilde M_1(ds,dz) + \int_0^s \int_{\mathbb{R}\setminus [-\eta,\eta]} z M_2(ds,dz), \label{levyitoxieta}
\end{align}
where the three terms are independent. $W$ is a standard Brownian motion. $M_1(ds,dz)$ (resp. $M_2(ds,dz)$) is a Poisson random measure on $[0,\infty) \times [-\eta,\eta]$ (resp. $[0,\infty) \times (\mathbb{R}\setminus [-\eta,\eta])$) with intensity measure $\textbf{1}_{z \in [-\eta,\eta]} ds \times \Pi_{\xi}(dz)$ (resp. $\textbf{1}_{z \notin [-\eta,\eta]} ds \times \Pi_{\xi}(dz)$), and $\tilde M_1(ds,dz) := M_1(ds,dz) - \textbf{1}_{z \in [-\eta,\eta]}ds \times \Pi_{\xi}(dz)$. 

It is well-known that $\mathbb{P}(\sup_{s \in [0,t]}|W(s)|<\epsilon/2\sqrt{A_{\xi}})>0$. For the second term in the right-hand side of \eqref{levyitoxieta}, Doob's martingale inequality yields 
\[ \mathbb{E} \left [ \left ( \sup_{s \in [0,t]} \int_0^s \int_{[-\eta,\eta]} z \tilde M_1(ds,dz) \right )^2 \right ] \leq 4\mathbb{E} \left [ \left ( \int_0^{t} \int_{[-\eta,\eta]} z \tilde M_1(ds,dz) \right )^2 \right ] =4t\int_{[-\eta,\eta]}u^2\Pi_{\xi}(du). \]
Since $\int_{\mathbb{R}} (1 \wedge u^2) \Pi_{\xi}(du)<\infty$, the right-hand side goes to $0$ as $\eta$ goes to $0$. For $\eta$ chosen small enough we thus have $\mathbb{P}(\sup_{s \in [0,t]}|\int_0^s \int_{[-\eta,\eta]} z \tilde M_1(ds,dz)|<\epsilon/2)>0$. Finally, with probability $e^{-t\Pi_{\xi}(\mathbb{R}\setminus [-\eta,\eta])}$, the third term in the right-hand side of \eqref{levyitoxieta} is null for all $s \in [0,t]$. We conclude that $\mathbb{P}(\sup_{s \in [0,t]}|\xi(s)|<\epsilon)>0$. 
\end{proof}

\section{Fourier duality for the partition function: Proof of Remark \ref{altexpr}} \label{proofaltexpr}

We work under the assumptions of Remark \ref{altexpr} and the choice $\pot:=-\log(v^r_\xi(\cdot))$ and $\pen:=-\log(m(\cdot))$. $v^r_\xi(\cdot) \in \crz$ by Lemma \ref{linkpsizpotential} and $m \in L^1(\mathbb{R})$ so, for all $n\geq 2$, $\partfct_n$ is well-defined and we have 
\begin{align}
\partfct_n = \int_{\mathbb{R}^n} v^r_\xi(y_2-y_1) \dots v^r_\xi(y_{n}-y_{n-1}) v^r_\xi(y_1-y_{n}) m(y_1)\dots m(y_n) dy_1\dots dy_n. \label{defzntransva} 
\end{align}
Using \eqref{defzntransva}, Lemma \ref{linkpsizpotential}, Fubini's theorem, and the definitions of $\hat m(\cdot)$ and $\hat \partfct_n$ we get 
\begin{align*}
\partfct_n = & \frac1{(2\pi)^n} \int_{\mathbb{R}^n} \left ( \int_{\mathbb{R}^n} \frac{e^{i(y_2-y_1)z_1}\times \dots \times e^{i(y_{n}-y_{n-1})z_{n-1}} \times e^{i(y_{1}-y_{n})z_{n}}}{(r-\psi_\xi(z_1)) \times \dots \times (r-\psi_\xi(z_n))} m(y_1)\dots m(y_n) dz_1\dots dz_n \right ) dy_1\dots dy_n \\
= & \frac1{(2\pi)^n} \int_{\mathbb{R}^n} \left ( \int_{\mathbb{R}^n} \frac{e^{-i(z_1-z_n)y_1} \times e^{-i(z_2-z_1)y_2} \times \dots \times e^{-i(z_{n}-z_{n-1})y_{n}}}{(r-\psi_\xi(z_1)) \times \dots \times (r-\psi_\xi(z_n))} m(y_1)\dots m(y_n) dy_1\dots dy_n \right ) dz_1\dots dz_n \\
= & \frac1{(2\pi)^n} \int_{\mathbb{R}^n} \frac{\hat m(z_1-z_n) \times \hat m(z_2-z_1) \times \dots \times \hat m(z_{n}-z_{n-1})}{(r-\psi_\xi(z_1)) \times \dots \times (r-\psi_\xi(z_n))} dz_1\dots dz_n = \frac{\hat \partfct_n}{(2\pi)^n}. 
\end{align*}

\section{Removal of periodic boundary condition: Proof of Remark \ref{perboundcond}} \label{proofremovboundcond}

We assume that the assumptions of Theorem \ref{encadrementtransfreeenergythtrans} are satisfied and set $\pot:=-\log(v^r_\xi(\cdot))$ and $\pen:=-\log(m(\cdot))$ in the definitions \eqref{hamilt} and \eqref{perboundcondexpr}. 
From Lemma \ref{linkpsizpotential} we have $0 \leq v^r_\xi(y) \leq v^{r}_\xi(0)$ for any $y \in \mathbb{R}$ so $\partfct_n \leq v^{r}_\xi(0) \partfct_{n}^f$. Recall from the combination of Lemma \ref{traces2new}, Remark \ref{mult1equiv}, and Proposition \ref{exproftracesrec} that $\partfct_n \sim (v^{r}_\xi(0)\lambda^R_{1})^n$ for large $n$. We thus get that $\liminf_{n \rightarrow \infty}\partfct_{n}^f/(v^{r}_\xi(0)\lambda^R_{1})^n>0$. 

Let us define $g_0:=\int_{\mathbb{R}} m(y) \phi^{r}_{y} dy$, where $\phi^{r}_{y}$ is defined in Section \ref{defopbasprop}. We also define the operator $H_r \in \mathcal{L}(L^2((-1,1)))$ by $H_r.f := \psl f,g_0 \psri g_0$. Let $n\geq 4$ and repeat the proof of Lemma \ref{traces2new} to compute $Tr(H_r.R_{r}^{n-2})$. We get $Tr(H_r.R_{r}^{n-2})=\sum_{j \geq 1} (\lambda_{j}^R)^{n-2} |\psl g_0,a_j \psri|^2$, with $(a_j)_{j \geq 1}$ as in Section \ref{basiseigfctgen}. Repeating the proof of Proposition \ref{exproftracesrec} we get $Tr(H_r.R_{r}^{n-2})=\partfct_{n}^f/v^{r}_\xi(0)^{n-1}$. Combining both expressions we get $\partfct_{n}^f/v^{r}_\xi(0)^{n-1}=\sum_{j \geq 1} (\lambda_{j}^R)^{n-2} |\psl g_0,a_j \psri|^2$. Recall from Remark \ref{mult1equiv} that the eigenvalue $\lambda_{1}^R$ of $R_{r}$ has multiplicity $1$ so, if $|\psl g_0,a_1 \psri|^2=0$ then, proceeding as in the proof of Lemma \ref{traces2new} we get $\partfct_{n}^f/v^{r}_\xi(0)^{n-1} = o((\lambda^R_{1})^n)$, which contradicts $\liminf_{n \rightarrow \infty}\partfct_{n}^f/(v^{r}_\xi(0)\lambda^R_{1})^n>0$. Therefore $|\psl g_0,a_1 \psri|^2>0$ so, proceeding as in the proof of Lemma \ref{traces2new} we get $\partfct_{n}^f/v^{r}_\xi(0)^{n-1} \sim |\psl g_0,a_1 \psri|^2 (\lambda^R_{1})^{n-2}$. Therefore $\lim_{n \rightarrow \infty} \partfct_{n}/\partfct_{n}^f = v^{r}_\xi(0)(\lambda^R_{1})^2/|\psl g_0,a_1 \psri|^2 \in (0,\infty)$, and the result follows. 

\end{appendix}

\textbf{Acknowledgments:} This paper is supported by NSFC grant No. 11688101. The author is grateful to Professor Fuzhou Gong and to Eric Endo for interesting discussions and references. 

\bibliographystyle{plain}
\bibliography{thbiblio}

\end{document}